\documentclass[11pt,reqno]{amsart}
\usepackage{amsmath,amssymb}
\usepackage{tensor}
\usepackage{enumerate}
\usepackage[colorlinks=true,allcolors=blue]{hyperref}

\usepackage{tikz}
\usetikzlibrary{cd,arrows,calc}

\newcommand{\C}{\mathbb{C}}
\newcommand{\R}{\mathbb{R}}
\newcommand{\Z}{\mathbb{Z}}
\newcommand{\N}{\mathbb{N}}
\newcommand{\Q}{\mathbb{Q}}

\newcommand{\bK}{\mathbb{K}}
\newcommand{\bT}{\mathbb{T}}
\newcommand{\g}{\mathsf{g}}
\newcommand{\liet}{\mathfrak{t}}
\newcommand{\lieg}{\mathfrak{g}}
\newcommand{\affW}{\widetilde{W}}
\newcommand{\extW}{\widehat{W}}

\newcommand{\cR}{\mathcal{R}}
\newcommand{\cRQ}{\mathcal{R}_{\mathbb{Q}}}

\newcommand{\Aut}[1]{\text{\normalfont Aut}(#1)}
\newcommand{\id}[1]{\operatorname{id}_{#1}}
\newcommand{\obj}[1]{\text{obj}(#1)}
\newcommand{\EV}[1]{\text{EV}(#1)}
\newcommand{\cE}{\mathcal{E}}

\newcommand{\cA}{\mathcal{A}}
\newcommand{\cC}{\mathcal{C}}
\newcommand{\cG}{\mathcal{G}}

\newcommand{\sfX}{\mathsf{X}}

\newcommand{\Eig}[2]{\text{Eig}(#1,#2)}
\newcommand{\Endo}[1]{\operatorname{End}(#1)}
\newcommand{\op}{\text{op}}

\newcommand{\lscal}[3]{\tensor*[_{#1}]{\langle #2, #3 \rangle}{}}
\newcommand{\rscal}[3]{\tensor*{\langle #1, #2 \rangle}{_{#3}}}

\newcommand{\rcpt}[2]{\mathcal{K}_{\circ #1}\left(#2\right)}
\newcommand{\lcpt}[2]{\mathcal{K}_{#1 \circ}\left(#2\right)}
\newcommand{\rbdd}[2]{\mathcal{L}_{\circ #1}\left(#2\right)}

\newcommand{\Vfin}{\mathcal{V}^{{\mathrm{fin}}}_{\C}}
\newcommand{\Viso}{\mathcal{V}^{{\mathrm{iso}}}_{\C}}

\newcommand{\MFa}{\mathsf{M}_{\mathnormal{F}}}
\newcommand{\MF}{\mathsf{M}_\mathnormal{F}^{\infty}}

\DeclareMathOperator{\Ad}{Ad}

\newtheorem{theorem}{Theorem}[section]
\newtheorem{lemma}[theorem]{Lemma}
\newtheorem{corollary}[theorem]{Corollary}
\newtheorem{prop}[theorem]{Proposition}

\newtheorem{definition}[theorem]{Definition}
\newtheorem{remark}[theorem]{Remark}
\newtheorem{example}[theorem]{Example}

\newcommand*\circled[1]{\tikz[baseline=(char.base)]{
            \node[shape=circle,draw,inner sep=0.8pt] (char) {#1};}}
            
\makeatletter
\newcommand{\extp}{\@ifnextchar^\@extp{\@extp^{\,}}}
\def\@extp^#1{\mathop{\bigwedge\nolimits^{\!#1}}}
\makeatother

\begin{document}
\title[Equivariant higher twisted $K$-theory of $SU(n)$]{Equivariant higher twisted $K$-theory of $SU(n)$ \\ for exponential functor twists}
\author{David E.\ Evans \and Ulrich Pennig}

\maketitle

\begin{abstract}
	We prove that each exponential functor on the category of finite-dimensional complex inner product spaces and isomorphisms gives rise to an equivariant higher (ie.\ non-classical) twist of $K$-theory over $G=SU(n)$. This twist is represented by a Fell bundle $\cE \to \cG$, which reduces to the basic gerbe for the top exterior power functor. The groupoid $\cG$ comes equipped with a $G$-action and an augmentation map $\cG \to G$, that is an equivariant equivalence. The $C^*$-algebra $C^*(\cE)$ associated to $\cE$ is stably isomorphic to the section algebra of a locally trivial bundle with stabilised strongly self-absorbing fibres. Using a version of the Mayer-Vietoris spectral sequence we compute the equivariant higher twisted $K$-groups $K^G_*(C^*(\cE))$ for arbitrary exponential functor twists over $SU(2)$, and also over $SU(3)$ after rationalisation.
\end{abstract}

\section{Introduction}
In three groundbreaking articles \cite{FreedHopkinsTelemanI:2011,FreedHopkinsTelemanII:2013,FreedHopkinsTelemanIII:2011} Freed, Hopkins and Teleman proved a close connection between the Verlinde algebra of a compact Lie group $G$ and its twisted equivariant $K$-theory, where $G$ acts on itself by conjugation. In case $G$ is simply connected their theorem boils down to the following statement: Let $R_k(G)$ be the Verlinde ring of positive energy representations of the loop group $LG$ at level~$k \in \Z$. Then the following $R(G)$-modules are naturally isomorphic 
\begin{equation} \label{eqn:FHT}
	R_k(G) \cong \,^{\tau(k)}K^{\dim(G)}_G(G)\ .
\end{equation}
This identification turns into an isomorphism of rings if the left hand side is equipped with the fusion product and the right hand side with the product induced by Poincar\'{e} duality and the group multiplication. The representation theory of loop groups also dictates the fusion rules of sectors in conformal field theories associated to these groups. In joint work with Gannon the first named author proved that it is in fact possible to recover the full system of modular invariant partition functions of these CFTs from the twisted $K$-theory picture \cite{EvansGannon-Modular:2009,EvansGannon-ModularII:2013}. This approach has been particularly successful in the case of the loop groups of tori, where the modular invariants can be expressed as $KK$-elements. Even exotic fusion categories, like the ones constructed by Tambara-Yamagami have elegant descriptions in terms of $K$-theory as shown in \cite{EvansGannon-tori:2019}.

For a simple and simply connected Lie group $G$ the classical equivariant twists of $K$-theory over $G$ are classified up to isomorphism by the equivariant cohomology group $H^3_G(G;\Z) \cong \Z$. The twist $\tau(k)$ in the FHT theorem corresponds to $(k+\check{h}(G))$ times the generator, where $\check{h}(G)$ is the dual Coxeter number of $G$. There are several ways to represent the generator of $H^3_G(G;\Z)$ geometrically: As a Dixmier-Douady bundle (a locally trivial bundle of compact operators) \cite{Meinrenken-conjugacy:2009}, as a bundle of projective spaces \cite{AtiyahSegal-TwistedK:2004}, in terms of (graded) central extensions of groupoids \cite{FreedHopkinsTelemanI:2011} or as a (bundle) gerbe usually called the basic gerbe \cite{Mickelsson:2003,Meinrenken-basicgerbe:2003,MurrayStevenson-basic_gerbe:2008}.

From a homotopy theoretic viewpoint (and neglecting the group action for a moment) twisted $K$-theory is an example of a twisted cohomology theory. If $R$ denotes an $A_{\infty}$ ring spectrum, then it comes with a space of units $GL_1(R)$ and has a classifying space of $R$-lines $BGL_1(R)$, which turns out to be an infinite loop space for $E_{\infty}$ ring spectra \cite{AndoBlumbergGepner-RLines:2014,AndoBlumbergGepner-Twists:2010}. In this situation the twists of $R$-theory are classified by $[X, BGL_1(R)]$. If $KU$ denotes a ring spectrum representing $K$-theory, then the group $[X,BGL_1(KU)]$ splits off  
\[
	H^1(X; \Z/2\Z) \times H^3(X;\Z)
\]
equipped with the multiplication 
\[
	(\omega_1, \delta_1) \cdot (\omega_2, \delta_2) = (\omega_1 + \omega_2, \delta_1 + \delta_2 + \beta(\omega_1 \cup \omega_2))\ ,
\] 
where $\beta$ denotes the Bockstein homomorphism. The twists classified by $H^1(X;\Z/2\Z)$ can easily be included in the classical picture for example by using graded central extensions as in \cite{FreedHopkinsTelemanI:2011} or graded projective bundles \cite{AtiyahSegal-TwistedK:2004}.

However, it was already pointed out by Atiyah and Segal in \cite{AtiyahSegal-TwistedK:2004} that the group $[X,BGL_1(KU)]$ is in general more subtle than ordinary cohomology. In joint work with Dadarlat the second author found an operator-algebraic description of the twists of $K$-theory which covers the full group $[X,BGL_1(KU)]$ and is based on locally trivial bundles of stabilised strongly self-absorbing $C^*$-algebras \cite{DadarlatP-DixmierDouady:2016}. This picture is also easily adapted to include groups of the form $[X, BGL_1(KU[\tfrac{1}{d}])]$, i.e.\ the twists of the localisation of $K$-theory away from an integer $d$. 

Descriptions of the basic gerbe over $SU(n)$ were developed by Meinrenken in~\cite{Meinrenken-basicgerbe:2003} and by Mickelsson in \cite{Mickelsson:2003}. Based on these results Murray and Stevenson found another construction in terms of bundle gerbes \cite{MurrayStevenson-basic_gerbe:2008}. Motivated by their work, the isomorphism \eqref{eqn:FHT} in the FHT theorem and the operator algebraic model in the non-equivariant case \cite{DadarlatP-DixmierDouady:2016} we introduce higher (i.e.\ non-classical) equivariant twists over $G = SU(n)$ in this paper. Our construction takes an exponential functor $F \colon \Viso \to \Viso$ on the category of complex finite-dimensional inner product spaces and isomorphisms as input and produces an equivariant Fell bundle $\cE \to \cG$. As in \cite{MurrayStevenson-basic_gerbe:2008} the groupoid $\cG = Y^{[2]}$ is obtained as the fibre product of the surjective submersion $Y \to G$ with itself, where  
\begin{equation} \label{eqn:surj_submersion}
	Y = \left\{ (g,z) \in G \times S^1 \setminus \{1\} \ | \ z \notin \EV{g} \right\}\ .
\end{equation}
The functor $F$ gives rise to a strongly self-absorbing $C^*$-algebra $\MF$ (a UHF-algebra) defined as the infinite tensor product
\(
	\MF = \Endo{F(\C^n)}^{\otimes \infty}
\). Note that $S^1 \setminus \{1\}$ inherits a total order from an orientation of the circle. The fibre of $\cE$ over $(g,z_1,z_2) \in \cG$ with $z_1 \leq z_2$ is given by 
\[
	\cE_{(g,z_1,z_2)} = F\left(\bigoplus_{z_1 < \lambda < z_2 \atop \lambda \in \EV{g}} \Eig{g}{\lambda} \right) \otimes \MF\ ,
\]
where the empty sum is understood as the zero vector space. The exponential structure of $F$ gives natural isomorphisms $F(0) \to \C$ and $F(V \oplus W) \to F(V) \otimes F(V)$. These give rise to the Fell bundle multiplication
\[
	\cE_{(g,z_1,z_2)} \otimes \cE_{(g,z_2,z_3)} \to \cE_{(g,z_1,z_3)}
\]
over points with $z_1 \leq z_2 \leq z_3$. In Thm.~\ref{thm:extension_of_Fell_bdls} we show that this structure extends to a Fell bundle over all of $\cG$, when we take
\(
	\cE_{(g,z_1,z_2)} = (\cE_{(g,z_2,z_1)})^{\text{op}}
\).

The space $Y$ and the groupoid $\cG$ come with canonical actions of $G$, such that the augmentation map $\cG \to G$ is an equivariant equivalence, if $G$ is equipped with the adjoint action. This action extends to the Fell bundle $\cE \to \cG$. Our main examples of exponential functors are the top exterior power functor $\extp^{\mathrm{top}}$ and the full exterior algebra $\extp^*$. For $F = \left(\extp^{\mathrm{top}}\right)^{\otimes m}$ with $m \in \N$ our construction reproduces the basic gerbe from \cite{MurrayStevenson-basic_gerbe:2008} (note that in this case $\MF = \C$) giving rise to the classical equivariant twist discussed above. 

A family of equivariant higher twists we will then focus on arises from $F = \left(\extp^*\right)^{\otimes m}$. Further examples of exponential functors and a classification result in terms of involutive $R$-matrices are discussed in \cite{Pennig:2018}.

The $C^*$-algebra $C^*(\cE)$ associated to the Fell bundle $\cE$ is a $C(G)$-algebra that is stably $C(G)$-isomorphic  to the section algebra of a locally trivial bundle $\cA \to G$ with fibre $\MF \otimes \bK$. Neglecting the $G$-action, this bundle is classified by a continuous map
\(
	G \to BGL_1\left(KU[\tfrac{1}{d}]\right)
\), 
where $d = \dim(F(\C))$. At this point our work makes contact with \cite{DadarlatP-DixmierDouady:2016}. The functor $F$ gives rise to group homomorphisms $U(k) \to U(F(\C)^{\otimes k})$, which induce a  map $BBU_{\oplus} \to BBU_{\otimes}[\tfrac{1}{d}]$ as shown in \cite{Pennig:2018}. We conjecture that the classifying map $G \to BGL_1\left(KU[\tfrac{1}{d}]\right)$ agrees up to homotopy with 
\[
	\tau^n_F \colon SU(n) \to SU \simeq BBU_{\oplus} \to BBU_{\otimes}[\tfrac{1}{d}] \to BGL_1\left(KU[\tfrac{1}{d}]\right)
\]
considered in \cite{Pennig:2018}, but defer the proof to future work. We expect an analogous statement to be true in an equivariant setting, but since the units of genuine $G$-equivariant ring spectra are a matter of current research in equivariant stable homotopy theory (see for example \cite[Ex.~5.1.17]{book:Schwede}) we will come back to this question in future work as well. 

The $G$-equivariant $K$-theory of $C^*(\cE)$ is a module over the localisation $K_0^G(\MF) \cong R_F(G) = R(G)[F(\rho)^{-1}]$ of the representation ring $R(G)$ at $F(\rho)$, where $\rho$ denotes the standard representation of $SU(n)$ on $\C^n$. In general these $R_F(G)$-modules are computable via a spectral sequence similar to the one used in \cite{Meinrenken-conjugacy:2009}. Our computations for $SU(2)$ are summarised in Thm.~\ref{thm:twisted_eq_K_of_SU2}. In this case the spectral sequence reduces to the Mayer-Vietoris sequence. 

We also compute the rationalised higher equivariant twisted $K$-theory of $SU(3)$ for general exponential functor twists, see  Thm.~\ref{thm:main_theorem_SU(3)}. Here we adapt the approach developed in  \cite{AdemCantareroGomez-TwistedK:2018} to our situation. In particular, we identify the rationalised chain complex on the $E^1$-page of the spectral sequence as the one computing Bredon cohomology of the maximal torus $\bT^2$ of $SU(3)$ with respect to a certain local coefficient system. 

If $F$ is given by a tensor power of the top exterior algebra functor, we recover the groups on the left hand side of \eqref{eqn:FHT} in both cases.

Even though to our knowledge equivariant exponential functor twists of the form studied here have not been considered in the literature before, similar exponential morphisms played a crucial role in \cite{Teleman-ModuliOfSurfaces:2004}. Instead of a localisation of $KU$, the ring spectrum considered by Teleman is $KU[[t]]$, the power series completion of $K$-theory and he shows that 
\[
	^\tau\!K^{\dim(G)}_G(G)\otimes \C[[t]]
\] 
is a Frobenius algebra and therefore extends to a 2D topological field theory for admissible higher twists $\tau$. 

Our low dimensional computations for $n = 2$ and for $n=3$ (after rationalisation) suggest that the higher twisted $K$-groups $K^G_{\dim(G)}(C^*(\cE))$ share a lot of the remarkable properties of the classical ones: 
\begin{itemize}
	\item[$\bullet$] The spectral sequence collapses on the $E_2$-page for all exponential functor twists.
	\item[$\bullet$] The $R_F(G)$-module $K_{\dim(G)}^G(C^*(\cE))$ is a quotient of $R_F(G)$ by an ideal. In particular, it carries a ring structure. 
	\item[$\bullet$] The local coefficient system in the $SU(3)$-case over the Lie algebra~$\liet$ of $\bT^2$ is determined by a homomorphism 
	\[
		\pi_1(\bT^2) \to GL_1(R_F(\bT^2))
	\]
	similar to the one constructed in \cite[Prop.~3.4]{AdemCantareroGomez-TwistedK:2018}, which is reminiscent of the appearance of the flat line bundles in \cite[(3.4)]{FreedHopkinsTeleman-Complex:2008}.
\end{itemize}
The case of odd tensor powers of the full exterior algebra twist over $SU(2)$ is of particular interest, since the ring $K^G_{\dim(G)}(C^*(\cE))$ is isomorphic to a fusion ring that has the same fusion rules as the even part of the classical twist on $SU(2)$ at odd levels (see Remark~\ref{rem:fusion_rules}). The question whether the ring structure we discovered stems from an intrinsic structure of the Fell bundle (similar to the multiplicative gerbe we get in the classical case) is one of the main future directions of this work. Even if this turns out to be false, we still expect these $K$-groups to come from modules over fusion categories, and we see our results as an indication for this.

The paper is structured as follows: Section 2 contains preliminary material about Morita-Rieffel equivalences between $C^*$-algebras. Exponential functors (Def.~\ref{def:exp_functor}) are revisited as well. In Section~3 we construct equivariant Fell bundles that represent equivariant higher twists. The corresponding equivariant $C^*$-algebras of sections are analysed in Sec.~4, with explicit computations of the equivariant $K$-theory for $SU(2)$ and $SU(3)$ (the latter after rationalisation) in the final section. More details can be found in the roadmap in the next section.

\subsection*{Roadmap to the main results}
In Sec.~3 we first consider for $G = SU(n)$ the groupoid $\cG = Y^{[2]}$ with $Y$ as in \eqref{eqn:surj_submersion}. It comes equipped with a $G$-action and a $G$-equivariant map $\cG \to G$ and is motivated by \cite{MurrayStevenson-basic_gerbe:2008}, where $\cG$ is used to construct the basic gerbe for unitary groups. It is more convenient for the analytical parts of this paper than the construction in \cite{FreedHopkinsTelemanI:2011}, which uses infinite-dimensional spaces, like the path space of the group. Composable elements in $\cG$ live in the same fibre over $G$. 

The main goal of Sec.~3 is the construction of the equivariant Fell bundle $\cE \to \cG$ from an exponential functor $F$, which is achieved in several steps: The fibre of the map $\cG \to G$ over $g \in G$ is a subgroupoid that embeds into $(S^1 \setminus \{1\})^2 \cong \R \times \R$. This gives rise to an orientation of the groupoid elements, ie.\ given $(g,z_1,z_2) \in \cG$ we could have $z_1 < z_2$ or $z_1 \geq z_2$. Using this orientation the groupoid $\cG$ decomposes into three disjoint components $\cG_-$, $\cG_0$ and $\cG_+$ that are closed with respect to composition, but not preserved by taking inverses. 

Any Fell bundle $\cE \to \cG$ comes with an associative multiplication $\cE_{g_1} \times \cE_{g_2} \to \cE_{g_1g_1}$ covering the groupoid composition. Our Fell bundle $\cE \to \cG$ will first be constructed over the subcategory $\cG_{0} \cup \cG_+$, where the multiplication is easy to define using the structural data from the exponential functor. This produces a saturated half-bundle $\cE_{0,+}$ in the sense of Def.~\ref{def:sat_half_bundle}. There is a canonical choice that extends the bundle over $\cG_-$. The issue is that we have to extend the multiplication as well and ensure its associativity. In Sec.~3.1 we state the technical result, proven in the appendix, that any saturated half-bundle extends uniquely up to isomorphism to a saturated Fell bundle (Thm.~\ref{thm:extension_of_Fell_bdls}). Moreover, if the half-bundle carries a $G$-action in an appropriate sense, then this action extends uniquely to one on the Fell bundle (Cor.~\ref{cor:G-action_on_ext}). 

We then focus on the construction of the half-bundle $\cE_{0,+}$ associated to an exponential functor $F$ in Sec.~3.2 (see Lem.~\ref{lem:cE_is_a_half_bundle}) and discuss the group action on it in Sec.~3.3 (see Cor.~\ref{cor:Fell_bundle}). Combining the results from all three sections we obtain a saturated $G$-equivariant Fell bundle $\cE \to \cG$ describing the exponential functor twist over $SU(n)$ associated to $F$.

In Sec.~4 we look at the $C^*$-algebra $C^*(\cE)$ associated to $\cE$ and prove that it is a continuous $C(G)$-algebra (see Lem.~\ref{lem:C(G)-algebra}). Its fibre $C^*(\cE_g)$ over $g \in G$ is Morita equivalent to $\MF$ (Lem.~\ref{lem:Fell_condition}). Since $C^*(\cE)$ satisfies the generalised Fell condition from \cite{DadarlatP-DixmierDouady:2016}, it is stably isomorphic to a locally trivial bundle classified by a continuous map to $BGL_1(KU[d^{-1}])$ where $d =\dim(F(\C))$ by Cor.~\ref{cor:CStarBundle}. Forgetting the group action, the Fell bundle $\cE \to \cG$ provided by our construction therefore gives rise to a twist of $KU[d^{-1}]$ in the sense of stable homotopy theory. Using results from strongly self-absorbing $C^*$-dynamical systems we then prove in Prop.~\ref{prop:module_structure} in Sec.~4.2 that its terms are in fact modules over $R_F(G) \cong R(G)[F(\rho)^{-1}]$ and so is $K_*^G(C^*(\cE))$.

The final section contains the computations of the equivariant higher twisted $K$-groups in the cases $SU(2)$ (Sec.~5.1) and $SU(3)$ (Sec.~5.2). For $SU(2)$ the result boils down to a Mayer-Vietoris argument and is summarised in Thm.~\ref{thm:twisted_eq_K_of_SU2}. In the general case, ie.\ for $SU(n)$, the Mayer-Vietoris sequence has to be replaced by a spectral sequence that is constructed in Prop.~\ref{prop:spectral_seq}. The simplex $\Delta^{n-1}$ parametrises the eigenvalues of group elements in $SU(n)$. Pulling back an appropriate cover of $\Delta^{n-1}$ along the map $q \colon SU(n) \to \Delta^{n-1}$ results in an equivariant cover of $SU(n)$ and the sequence in Prop.~\ref{prop:spectral_seq} is the Mayer-Vietoris spectral sequence associated to that cover. It is concentrated in even degrees, where it is given by the chain complex \eqref{eqn:chain_cplx} in the case of $SU(3)$. We determine the differentials on the $E_1$-page in Lemma~\ref{lem:diff_d0_d1} and compare the resulting chain complex with the one computing the Bredon cohomology of the Lie algebra of the maximal torus in Lem.~\ref{lem:coh_with_coeff} in Sec.~5.2.1. The identification \eqref{eqn:needs_Q} that allows us to compute these groups as fixed points under a Weyl group action is only true after rationalisation. As shown in Thm.~\ref{thm:Koszul} the $E_2$-page of the spectral sequence obtained from \eqref{eqn:chain_cplx} has only one non-vanishing term in the even rows, which is determined in Lem.~\ref{lem:submod} and allows us to compute the rational equivariant higher twisted $K$-groups in Thm.~\ref{thm:main_theorem_SU(3)}.

\subsection*{Acknowledgements}
The second author would like to thank Dan Freed and Steffen Sagave for helpful discussions. Part of this work was completed while the authors were staying at the Newton Institute during the programme ``Operator algebras: subfactors and their applications''. Both authors would like to thank the institute for its hospitality. Their research was supported in part by the EPSRC grants EP/K032208/1 and EP/N022432/1. 




\section{Preliminaries}

\subsection{Bimodules and Morita-Rieffel equivalences}
In this section we collect some well-known facts about Hilbert bimodules and Morita-Rieffel equivalences. This is mainly to fix notation. A detailed introduction to Hilbert $C^*$-modules can be found in \cite{Lance-Toolkit:1995}. Let $A,B$ be separable unital $C^*$-algebras.
\begin{definition} \label{def:AB-bimod}
	An \emph{$A$-$B$-bimodule} is a right Hilbert $B$-module $V$ together with a $*$-homomorphism 
	\[
		\psi_A \colon A \to \rcpt{B}{V}\ ,
	\]
	where $\rcpt{B}{V}$ denotes the compact adjointable right $B$-linear operators on~$V$. An $A$-$B$-bimodule is called a \emph{(Morita-Rieffel) equivalence bimodule} if $V$ is full and $\psi_A$ is an isomorphism.
\end{definition}

Given a right Hilbert $B$-module $V$ with inner product $\rscal{\cdot}{\cdot}{B}$ we can associate a left Hilbert $B$-module $V^\op$ to it in a natural way. The vector space underlying $V^\op$ is $\overline{V}$, i.e.\ $V$ equipped with the conjugate linear structure. For a given element $v \in V$ we denote the corresponding element in $V^\op$ by $v^*$. The left multiplication by $b \in B$ is defined by
\[
	b\,v^* = (v\,b^*)^*
\]
and the left $B$-linear inner product is $\lscal{B}{v_1^*}{v_2^*} = \left(\rscal{v_2}{v_1}{B}\right)^*$. 

The space $\hom_B(V,B)$ of right $B$-linear adjointable morphisms is a left Hilbert $B$-module via the left multiplication $(b \cdot \varphi)(v) = b\varphi(v)$ and the inner product $\lscal{B}{\varphi_1}{\varphi_2} = \varphi_1 \circ \varphi_2^* \in \hom_B(B,B) \cong B$. The map 
\[
	V^\op \to \hom_B(V,B) \qquad , \qquad v^* \mapsto \rscal{v}{\,\cdot\,}{B}
\] 
provides a canonical isomorphism of left Hilbert $B$-modules and we will sometimes identify the two. Note that there is a conjugate linear bijection
\begin{equation} \label{eqn:star_on_mod}
	V \to V^\op \quad , \quad v \mapsto v^* 	
\end{equation}
which satisfies $(v\,b)^* = b^*\,v^*$.

The definition of $A$-$B$ equivalence bimodules may seem asymmetric in $A$ and $B$. It is actually not: Let $V$ be an $A$-$B$ equivalence bimodule. It carries a left multiplication by $a \in A$ defined by
\(
	a\,v = \psi_A(a)v
\)
and a left $A$-linear inner product given by
\[
	\lscal{A}{v_1}{v_2} = \psi_A^{-1}\left(v_1\rscal{v_2}{\,\cdot\,}{B}\right)\ .
\]
With respect to this multiplication and inner product $V$ is a full left Hilbert $A$-module. The rank $1$ operator $\lscal{A}{\,\cdot\,}{v_1}\,v_2$ agrees with the right multiplication by $\rscal{v_1}{v_2}{B}$. Since $V$ is full, the compact left $A$-linear operators therefore agree with $B$, but the multiplication is reversed, i.e.\ we obtain an isomorphism
\[
	\psi_B \colon B^\op \to \lcpt{A}{V} 
\]
that sends $b$ to right multiplication by $b$. Thus, we could alternatively define an $A$-$B$ equivalence bimodule as a full left Hilbert $A$-module together with an isomorphism $\psi_B$ as above. 

If $V$ is an $A$-$B$ equivalence bimodule, then $V^\op$ is a full left Hilbert $B$-module. Let 
\[
	\psi_A^\op \colon A^\op \to \lcpt{B}{V^\op}
\]
be the $*$-homomorphism given by $\psi_A^\op(a)(v^*) = (\psi_A(a^*)v)^*$. Conjugation by $v \mapsto v^*$ induces a conjugate linear isomorphism $\rcpt{B}{V} \cong \lcpt{B}{V^\op}$. From this we deduce that $\psi_A^\op$ is a (linear) $*$-isomorphism. Thus, $V^\op$ is a $B$-$A$ equivalence bimodule. 

Note that the left $A$-linear and right $B$-linear inner product on an $A$-$B$-equivalence bimodule $V$ satisfy the compatibility condition 
\begin{equation} \label{eqn:comp_inner_prod}
	v_1\rscal{v_2}{v_3}{B} = \lscal{A}{v_1}{v_2}v_3
\end{equation}
for all $v_1,v_2,v_3 \in V$.

Let $A$, $B$ and $C$ be separable unital $C^*$-algebras and let $V$ be an $A$-$B$ equivalence bimodule and $W$ be a $B$-$C$ equivalence bimodule. The tensor product over $B$ gives an $A$-$C$ equivalence bimodule that we will denote by 
\[
	V \otimes_B W\ .
\]
For details about this construction we refer the reader to \cite[Chap.~4]{Lance-Toolkit:1995} or \cite{Rieffel-MoritaEquiv:1982}. The left $A$-linear inner product provides an $A$-$A$ bimodule isomorphism 
\[
	\lscal{A}{\cdot}{\cdot} \colon V \otimes_B V^\op \to A \ .
\]
Similarly, $\rscal{\cdot}{\cdot}{B} \colon V^\op \otimes_A V \to B$ is a bimodule isomorphism as well. Concerning the opposite bimodule of a tensor product, there is a canonical isomorphism 
\begin{equation} \label{eqn:op_of_tensor}
	(V \otimes_B W)^\op \cong W^\op \otimes_B V^\op
\end{equation}
given on elementary tensors by $(v \otimes w)^* \mapsto w^* \otimes v^*$. 

Let $G$ be a compact group and let $\alpha \colon G \to \Aut{B}$ be a continuous action of $G$ on $B$, where $\Aut{B}$ is equipped with the pointwise-norm topology. We will call $B$ a $G$-algebra for short. A \emph{$G$-equivariant right Hilbert $B$-module} \cite[Def.~1 and Def.~2]{Kasparov-ThmStinespring:1980} is defined to be a right Hilbert $B$-module $V$ together with an action of $G$ (denoted by $g \cdot v$ for $g \in G, v \in V$) that satisfies
\begin{enumerate}[a)]
	\item $g \cdot (vb) = (g \cdot v)\alpha_g(b)$,
	\item $\rscal{g\cdot v}{g \cdot w}{B} = \alpha_g(\rscal{v}{w}{B})$,
	\item $(g, v) \mapsto g \cdot v$ is continuous.
\end{enumerate} 
If $V$ is a $G$-equivariant right Hilbert $B$-module, then $V^\op$ equipped with the action $g \cdot v^* = (g \cdot v)^*$ is a $G$-equivariant left Hilbert $B$-module. The group $G$ acts continuously on the $C^*$-algebra $\rcpt{B}{V}$ by conjugation. If $A$ denotes another $G$-algebra, then a $G$-equivariant $A$-$B$-bimodule is an $A$-$B$-bimodule~$V$ where the structure map $\psi \colon A \to \rcpt{B}{V}$ is $G$-equivariant.

\subsection{Exponential functors}
Let $\Vfin$ be the category of finite-dimensional complex inner product spaces and linear maps and denote by $\Viso \subset \Vfin$ the subgroupoid with the same objects but unitary isomorphisms as its morphisms. The higher twists we are going to construct will depend on the choice of an exponential functor on $\Viso$. In the context of higher twists these were first considered in \cite{Pennig:2018}, which also contains a classification of those exponential functors that arise from restrictions of polynomial exponential functors on $\Vfin$ in terms of involutive solutions to the Yang-Baxter equation (involutive $R$-matrices). The following definition is taken from \cite[Def.~2.1]{Pennig:2018} and we refer the reader to that paper for a detailed description of the three conditions a), b) and c) stated below.

\begin{definition} \label{def:exp_functor}
	An \emph{exponential functor on } $\Vfin$ (resp.\ $\Viso$) is a triple consisting of a functor $F \colon \Vfin \to \Vfin$ (resp.\ $F \colon \Viso \to \Viso$) together with natural unitary isomorphisms
	\[
		\tau_{V,W} \colon F(V \oplus W) \to F(V) \otimes F(W)
	\]
	and $\iota \colon F(0) \to \C$ that satisfy the following conditions
	\begin{enumerate}[a)]
		\item $F$ preserves adjoints,
		\item $\tau$ is associative,
		\item $\tau$ is unital with respect to $\iota$.
	\end{enumerate}
For an exponential functor $F$ (on $\Vfin$ or $\Viso$) let $d(F)=\dim(F(\C))$. We define the \emph{dimension spectrum of $F$} to be 
\[
\text{Dim}(F) := \{ \dim(F(V))\ |\ V \in \obj{\Viso}\} = \{ d(F)^n \ | \ n \in \N_0\}\ .
\]
\end{definition}
The exterior algebra functor $F(V) = \extp^*V$ provides a natural example of an exponential functor. The symmetric algebra $\text{Sym}^*(V)$ of a vector space $V$ comes with natural transformations $\tau$ and $\iota$ as above. It is, however, ruled out by the fact that $\text{Sym}^*(V)$ is infinite-dimensional. The exterior algebra functor can be modified as follows: Let $W$ be a finite-dimensional inner product space and consider
\[
	F^W(V) = \bigoplus_{k=0}^\infty W^{\otimes k} \otimes \extp^k V
\]
As outlined in \cite[Sec.~2.2]{Pennig:2018} this provides a polynomial exponential functor $F^W \colon \Vfin \to \Vfin$. 
	
\section{Higher twists via Fell bundles}
In this section we will consider a groupoid $\cG$ that carries an action of $G = SU(n)$ and comes with a surjection $\cG \to G$ that is equivariant with respect to the conjugation action of $G$ on itself. In fact, $\cG$ will be Morita equivalent to $G$. We will then construct a Fell bundle $\pi \colon \cE \to \cG$ such that its total space $\cE$ comes with an action of $G$ and $\pi$ is equivariant. The groupoid $\cG$ decomposes into three disjoint parts $\cG_{-} \cup \cG_0 \cup \cG_{+}$ in such a way that $\cG_{0,+} = \cG_0 \cup \cG_{+}$ generates the whole groupoid. We will construct the analogue of a saturated Fell bundle over $\cG_{0,+}$ first, before extending it to all of $\cG$. To achieve this we will need the extension theorem proven in the next section.

\subsection{An extension theorem for saturated Fell bundles}
In this section we consider the following situation: Let $\cG$ be a topological groupoid with object space $\cG^{(0)}$. Suppose that we have a decomposition of the space of arrows
\begin{equation} \label{eqn:decomp}
	\cG = \cG_{-} \cup \cG_0 \cup \cG_{+}	
\end{equation}
into disjoint open (and thus also closed) subspaces. Let $\cG^{(2)}$ be the space of composable arrows. For $\mathcal{U}, \mathcal{V} \subset \cG$ define 
\[
	\mathcal{U} \cdot \mathcal{V} = \left\{ \g_1 \cdot \g_2 \in \cG \ |\ (\g_1,\g_2) \in \cG^{(2)} \text{ and } \g_1 \in \mathcal{U}, \g_2 \in \mathcal{V} \right\} 
\]
and $\mathcal{U}^{-1} = \left\{ \g^{-1} \in \cG \ | \ \g \in \mathcal{U} \right\}$. We will assume that the decomposition (\ref{eqn:decomp}) satisfies the following conditions
\begin{align}
	\left(\cG_+\right)^{-1} &= \cG_- \label{eqn:inv_pm} \\
	\left(\cG_0\right)^{-1} &= \cG_0 \label{eqn:inv_zero} \\
	\cG_+ \cdot \cG_+ &\subseteq \cG_+ \label{eqn:m_pp}\\
	\cG_0 \cdot \cG_+ &= \cG_+ \label{eqn:m_zp} \\
	\cG_+ \cdot \cG_0 &= \cG_+ \label{eqn:m_pz} \\
	\cG_0 \cdot \cG_0 &= \cG_0 \label{eqn:m_zz}
\end{align}
Since the identities on the objects of $\cG$ are fixed points of the inversion and the decomposition is disjoint, we obtain from (\ref{eqn:inv_pm}) and (\ref{eqn:inv_zero}) that they must be contained in $\mathcal{G}_0$. Therefore (\ref{eqn:m_pz}) is actually equivalent to $\cG_+ \cdot \cG_0 \subseteq \cG_+$ and likewise for (\ref{eqn:m_zp}). Taking inverses we also obtain 
\begin{align*}
	\cG_- \cdot \cG_- &\subseteq \cG_-\\ 
	\cG_0 \cdot \cG_- &= \cG_- \\ 
	\cG_- \cdot \cG_0 &= \cG_-
\end{align*}
Let $\cG_{0,+} = \cG_0 \cup \cG_{+}$.

\begin{definition} \label{def:sat_half_bundle}
Let $A$ be a separable unital $C^*$-algebra. A \emph{saturated (Fell) half-bundle} is given by the following data: A Banach bundle $\cE_{0,+} \to \cG_{0,+}$ with the property that 
\[
	\left.\cE_{0,+}\right|_{\cG_0} = \cG_0 \times A
\]
and a continuous multiplication map $\mu \colon \cE_{0,+}^{(2)} \to \cE_{0,+}$ where 
\[
\cE_{0,+}^{(2)} = \left\{ (e_1,e_2) \in \cE_{0,+}^2 \ |\ (\pi(e_1),\pi(e_2)) \in \cG^{(2)} \right\} \subset \left(\cE_{0,+}\right)^2
\] 
is equipped with the subspace topology. This data has to satisfy the following conditions:

\begin{enumerate}[a)]
	\item The multiplication $\mu$ is bilinear and associative. It extends the canonical one on $\left.\cE_{0,+}\right|_{\cG_0} = \cG_0 \times A$ and fits into a commutative diagram 
	\[
	\begin{tikzcd}
			\cE_{0,+}^{(2)} \ar[d,"\pi \times \pi" left] \ar[r,"\mu"] & \cE_{0,+} \ar[d,"\pi"] \\
			\cG_{0,+}^{(2)}  \ar[r] & \cG_{0,+}
	\end{tikzcd}
	\]
	in which the lower horizontal map is the groupoid multiplication. We will use the abbreviated notation $e_1 \cdot e_2 := \mu(e_1,e_2)$.
	\item There is a continuous inner product $\rscal{\,\cdot\,}{\,\cdot\,}{A} \colon \cE_{0,+} \times_{\cG_{0,+}} \cE_{0,+} \to A \times \cG^{(0)}$ that is right $A$-linear with respect to the multiplication $(e,a) \mapsto e \cdot a$ induced by $\mu$ with $\pi(e) = \g \in \cG_{0,+}$ and $\pi(a) = \id{s(\g)} \in \cG_0$. It fits into the commutative diagram
	\[
	\begin{tikzcd}[column sep=2cm]
		\cE_{0,+} \times_{\cG_{0,+}} \cE_{0,+} \ar[r,"\rscal{\,\cdot\,}{\,\cdot\,}{A}"] \ar[d,"\pi" left] & A \times \cG^{(0)} \ar[d,"\pi"] \\
		\cG_{0,+} \ar[r,"s" below] & \cG^{(0)}
	\end{tikzcd}
	\]
	and restricts to $\rscal{(a_1,\g)}{(a_2,\g)}{A} = (a_1^*a_2, s(\g))$ for $(a_i,\g) \in \left.\cE_{0,+}\right|_{\cG_0}$. It is compatible with the norm in the sense that 
	\begin{equation} \label{eqn:inner_prod_and_norm}
		\lVert \rscal{e}{e}{A} \rVert = \lVert e \rVert^2
	\end{equation}
	and turns each fibre $\left(\cE_{0,+}\right)_{\g}$ into a right Hilbert $A$-module. The left multiplication $(a,e) \mapsto a \cdot e$ with $\pi(e) = \g$ and $\pi(a) = \id{r(\g)}$ is compact adjointable with respect to this inner product with adjoint given by $a^*$. Moreover, this left multiplication induces a $*$-isomorphism
	\[
		\psi_{A,\g} \colon A \to \rcpt{A}{\left(\cE_{0,+}\right)_{\g}}
	\]
	between $A$ and the compact $A$-linear operators on each fibre. In other words, each fibre $\left(\cE_{0,+}\right)_\g$ is an $A$-$A$ equivalence bimodule. 
	\item The Hilbert $A$-bimodule structure on the fibres is compatible with the multiplication in the sense that $\mu$ induces an $A$-$A$ bimodule isomorphism 
	\[
	\begin{tikzcd}[column sep=2cm]
			\left( \cE_{0,+} \right)_{\g_1} \otimes_A \left( \cE_{0,+} \right)_{\g_2} \ar[r,"\mu" above, "\cong" below] & \left( \cE_{0,+} \right)_{\g_1\g_2}		
	\end{tikzcd}
	\]
	for each composable pair $(\g_1,\g_2) \in \cG^{(2)}_{0,+}$.
\end{enumerate}
\end{definition}

\begin{remark}
Note that Def.~\ref{def:sat_half_bundle} c) and \eqref{eqn:inner_prod_and_norm} imply that the norm on $\cE_{0,+}$ is submultiplicative in the sense that $\lVert e_1 \cdot e_2 \rVert \leq \lVert e_1 \rVert \cdot \lVert e_2 \rVert$. 
\end{remark}

The main reason for considering half-bundles is the following extension theorem that will be proven in the appendix:
\begin{theorem} \label{thm:extension_of_Fell_bdls}
	Let $(\cE_{0,+}, \mu, \rscal{\,\cdot\,}{\,\cdot\,}{A})$ be a saturated half-bundle. There is a saturated Fell bundle $\pi \colon \cE \to \cG$ with the property that $\left.\cE\right|_{\cG_{0,+}} = \cE_{0,+}$, the multiplication of $\cE$ restricts to the one of $\cE_{0,+}$ and for $e_1,e_2 \in \cE_g$ we have  
	\(
		e_1^* \cdot e_2 = \rscal{e_1}{e_2}{A}\ .
	\) Moreover, $\pi \colon \cE \to \cG$ is unique up to (a canonical) isomorphism of Fell bundles. 
\end{theorem}

Since we want to construct an equivariant Fell bundle from an equivariant half-bundle, we need to understand group actions on both structures. Restricting to the following kind of actions is natural in this context:
\begin{definition} \label{def:admissible_action}
	Let $G$ be a compact group and let $\cG$ be a groupoid with a decomposition as in \eqref{eqn:decomp} that has the properties \eqref{eqn:inv_pm} -- \eqref{eqn:m_zz}. We call an action of $G$ on $\cG$ by groupoid automorphisms  \emph{admissible} if it preserves the decomposition from \eqref{eqn:decomp}, i.e.\ each group element yields a homeomorphism $\cG_- \to \cG_-$ and similarly for $\cG_0$ and $\cG_+$, respectively. 
\end{definition}

\begin{definition} \label{def:eq_sat_half_bundle}
Let $G$ be a compact group that acts admissibly on $\cG$ and let $A$ be a separable unital $G$-algebra. A \emph{$G$-equivariant saturated half-bundle} is a saturated half-bundle $\cE_{0,+}$ carrying a continuous $G$-action such that the projection map $\cE_{0,+} \to \cG_{0,+}$ is equivariant and the following properties hold:
\begin{enumerate}[a)]
	\item On $\cE_0 = \left.\cE_{0,+}\right|_{\cG_0} = \cG_0 \times A$ the action restricts to the diagonal action of $G$ on $\cG_0$ and $A$.
	\item The multiplication map $\mu \colon \cE_{0,+}^{(2)} \to \cE_{0,+}$ is $G$-equi\-variant (where the domain is equipped with the diagonal $G$-action) and the inner product satisfies 
	\[
		\rscal{g \cdot e_1}{g \cdot e_2}{A} = \alpha_g(\rscal{e_1}{e_2}{A})
	\]
	for all $g \in G$. 
\end{enumerate}  
\end{definition}

\begin{corollary} \label{cor:G-action_on_ext}
Suppose that $\cG$ carries an admissible action by a compact group $G$. Let $(\cE_{0,+}, \mu, \rscal{\,\cdot\,}{\,\cdot\,}{A})$ be a $G$-equivariant saturated half-bundle. Let $\cE$ be the extension of $\cE_{0,+}$ to a saturated Fell bundle as in Thm.~\ref{thm:extension_of_Fell_bdls}. 

Then the $G$-action on $\cE_{0,+}$ extends to a continuous $G$-action on $\cE$ in such a way that the multiplication map and the projection $\pi \colon \cE\to \cG$ are equivariant and $g\cdot e^* = (g \cdot e)^*$ for all $e \in \cE$ and $g \in G$. This extension is unique.
\end{corollary}

\begin{proof}
The condition $g \cdot e^* = (g \cdot e)^*$ uniquely fixes the group action on $\cE_-$ and has all properties stated in the corollary.
\end{proof}

\subsection{Construction of the Fell bundle over $SU(n)$} \label{subsec:Fell_bundle}
The groupoid $\cG$ alluded to in the introduction to this section is now constructed as follows: Let $G = SU(n)$, $\bT \subset \C$ the unit circle and let $Z = \bT \setminus \{1\} \cong (0,1)$. For $g \in G$ denote the set eigenvalues of $g$ (in its standard representation on $\C^n$) by $\EV{g}$. Let $Y$ be the space
\[
	Y = \left\{ (g,z) \in G \times Z \ | \ z \notin \EV{g} \right\}\ .
\]
There is a canonical quotient map $\pi \colon Y \to G$. The groupoid $\cG$ is now given by the fibre product $Y^{[2]}$ of $Y$ with itself over $G$, i.e.\  
\[
	\cG = Y^{[2]} = \{ (y_1,y_2) \in Y \times Y\ |\ \pi(y_1) = \pi(y_2) \}
\] 
equipped with the subspace topology\footnote{We will view an element $(g,z_1,z_2) \in Y^{[2]}$ as a morphism from $(g,z_2)$ to $(g,z_1)$. Thus, the composition is $(g,z_1,z_2) \cdot (g,z_2,z_3) = (g,z_1,z_3)$.}. Note that we can identify this space with
\[
	Y^{[2]} = \left\{ (g,z_1,z_2) \in G \times Z \times Z \ | \ z_i \notin \EV{g} \text{ for } i \in \{1,2\} \right\}\ .
\]
Since $Z$ is homeomorphic to an open interval via $e \colon (0,1) \to Z$ with $e(\varphi) = \exp(2\pi i\,\varphi)$ we can equip it with a total ordering by defining $z_2 = e(\varphi_2) \geq z_1 = e(\varphi_1)$ if and only if $\varphi_2 \geq \varphi_1$. Now we can decompose $Y^{[2]}$ into disjoint subspaces $Y^{[2]} = Y^{[2]}_+ \cup Y^{[2]}_0 \cup Y^{[2]}_-$ with 
\begin{align*}
	Y^{[2]}_+ & = \{ (g,z_1,z_2) \in Y^{[2]}\ | \ z_2 > z_1 \text{ and }  \exists \lambda \in \EV{g},\ z_2 > \lambda > z_1 \}\ , \\
	Y^{[2]}_0 & = \{ (g,z_1,z_2) \in Y^{[2]}\ | \ \nexists \lambda \in \EV{g}, \max(z_1,z_2) > \lambda > \min(z_1,z_2) \}\ , \\
	Y^{[2]}_- & = \{ (g,z_1,z_2) \in Y^{[2]}\ | \ z_2 < z_1 \text{ and }  \exists \lambda \in \EV{g},\ z_2 < \lambda < z_1 \}\ .
\end{align*}

Fix a (continuous) exponential functor $(F,\tau,\iota)$ on $\Viso$ as in Def.~\ref{def:exp_functor}. Consider the standard representation of $G$ on $\C^n$ and let $\MFa = \Endo{F(\C^n)}$. Let $\MF$ be the UHF-algebra given by the infinite tensor product 
\[
	\MF = \bigotimes_{i=1}^{\infty} \MFa\ .
\]
We will construct a saturated half-bundle $\cE_{0,+}$ over $\cG_{0,+} = Y^{[2]}_+ \cup Y^{[2]}_0$ such that over $\cG_0 = Y^{[2]}_0$ it coincides with the trivial $C^*$-algebra bundle
\[
	\cE_0 = \cG_0 \times \MF\ .
\]
To understand the fibre of $\cE_{0,+}$ over $\cG_+$ fix $g \in G$ and let $(g,z_1,z_2) \in \cG_+$. Consider the following subspaces of $\C^n$: 
\begin{gather*}
	E(g,z_1,z_2) = \bigoplus_{\overset{z_1 < \lambda < z_2}{\lambda \in \EV{g}}} \Eig{g}{\lambda} \\
	E^{\prec}(g,z_1) = \bigoplus_{\overset{\lambda < z_1}{\lambda \in \EV{g}}} \Eig{g}{\lambda} \quad , \quad
	E^{\succ}(g,z_2) = \bigoplus_{\overset{z_2 < \lambda}{\lambda \in \EV{g}}} \Eig{g}{\lambda} \ .
\end{gather*}
and note that the natural transformation $\tau$ from the exponential functor turns the direct sum decomposition of $\C^n$ shown in the next line into a tensor product decomposition shown in the line below:
\begin{align} \label{eqn:FCn_decomp}
	\C^n &= \Eig{g}{1} \oplus E^{\prec}(g,z_1) \oplus E(g,z_1,z_2) \oplus E^{\succ}(g,z_2)  \\
	F(\C^n) &\cong F(\Eig{g}{1}) \otimes F(E^{\prec}(g,z_1)) \otimes F(E(g,z_1,z_2)) \otimes F(E^{\succ}(g,z_2))\notag
\end{align}
Denote the corresponding endomorphism algebras of the tensor factors by
\begin{gather*}
	\MFa(g, z_1, z_2) = \Endo{F(E(g,z_1,z_2))} \quad , \quad \MFa(g,1) = \Endo{F(\Eig{g}{1})}\\ 
	\MFa^{\prec}(g,z_1) = \Endo{F(E^{\prec}(g,z_1))} \quad , \quad
	\MFa^{\succ}(g,z_2) = \Endo{F(E^{\succ}(g,z_2))}\ .
\end{gather*}
Just as in \cite[Sec.~3]{MurrayStevenson-basic_gerbe:2008} it follows that the bundle $E \to \cG_+$ with fibre $E(g,z_1,z_2)$ over $(g,z_1,z_2) \in \cG_+$ is a locally trivial vector bundle. Therefore $F(E)$ is as well. Observe that the endomorphism bundle of $F(E)$ has fibre $\MFa(g,z_1,z_2)$ over $(g,z_1,z_2) \in \cG_+$. The fibre of our half-bundle $\cE_{+}$ will be given by the following locally trivial bundle of right Hilbert $\MF$-modules: 
\[
	\cE_{(g,z_1,z_2)} = F(E(g,z_1,z_2)) \otimes \MF
\]
where the right multiplication is given by right multiplication on $\MF$. The transformation $\tau$ induces a $*$-isomorphism for $(g,z_1,z_2), (g,z_2,z_3) \in \cG_+$ of the form
\begin{equation} \label{eqn:MF_tensor_prod}
	\MFa(g,z_1,z_2) \otimes \MFa(g,z_2,z_3) \to \MFa(g,z_1,z_3)
\end{equation}
To define the multiplication on the fibres of $\cE_+$ we need the next lemma.

\begin{lemma} \label{lem:associative_isos_MF}
There is an isomorphism 
\(
	\varphi_{g,z_1,z_2} \colon \MFa(g,z_1,z_2) \otimes \MF \to \MF
\)
(constructed in the proof) which is associative in the sense that for $(g,z_1,z_2),$ $(g,z_2,z_3) \in \cG_+$ the following diagram commutes:
\[
	\begin{tikzcd}[column sep=2.2cm]
		\MFa(g,z_1,z_2) \otimes \MFa(g,z_2,z_3) \otimes \MF \ar[d] \ar[r,"\id{} \otimes \varphi_{g,z_2,z_3}"] & \MFa(g,z_1,z_2) \otimes \MF \ar[d,"\varphi_{g,z_1,z_2}"] \\
		\MFa(g,z_1,z_3) \otimes \MF \ar[r,"\varphi_{g,z_1,z_3}"] & \MF
	\end{tikzcd}
\]
where the vertical arrow on the left is the isomorphism from \eqref{eqn:MF_tensor_prod}.
\end{lemma}

\begin{proof}
To construct $\varphi_{g,z_1,z_2}$ first note that the decomposition \eqref{eqn:FCn_decomp} yields a corresponding decomposition of the algebra $\MFa$, which we will also denote by $\tau$ by a slight abuse of notation:
\begin{equation}
\begin{tikzcd} \label{eqn:tau}
	\MFa \ar[r,"\tau" above, "\cong" below] & \MFa(g,1) \otimes \MFa^{\prec}(g,z_1) \otimes \MFa(g,z_1,z_2) \otimes \MFa^{\succ}(g,z_2)\ .
\end{tikzcd}
\end{equation}
Let $\MFa^{1,\prec}(g,z_1) = \MFa(g,1) \otimes \MFa^{\prec}(g,z_1)$ and define
\[
	\varphi_{g,z_1,z_2}^{(k)} \colon \MFa(g,z_1,z_2) \otimes \MFa^{\otimes k} \to \MFa^{\otimes (k+1)}
\]
to be the following composition
\[
\begin{tikzcd}
\MFa(g,z_1,z_2) \otimes \MFa^{\otimes k} \ar[d,"\id{} \otimes \tau^{\otimes k}"] \\
\MFa(g,z_1,z_2) \otimes (\MFa^{1,\prec}(g,z_1) \otimes \MFa(g,z_1,z_2) \otimes \MFa^{\succ}(g,z_2))^{\otimes k} \ar[d,"\alpha^{(k)}_{g,z_1,z_2}"] \\
 (\MFa^{\prec}(g,z_1) \otimes \MFa(g,z_1,z_2) \otimes \MFa^{\succ}(g,z_2))^{\otimes (k+1)} \ar[d,"(\tau^{-1})^{\otimes (k+1)}"] \\
 \MFa^{\otimes (k+1)}
\end{tikzcd}
\]
with the endomorphism $\alpha^{(k)}_{g,z_1,z_2}$ given by
\begin{align*}
	& \alpha^{(k)}_{g,z_1,z_2}(T \otimes (A_1 \otimes B_1 \otimes C_1) \otimes \dots \otimes (A_k \otimes B_k \otimes C_k)) \\
	=\ & (A_1 \otimes T \otimes C_1) \otimes (A_2 \otimes B_1 \otimes C_2) \otimes \dots \otimes (A_k \otimes B_{k-1} \otimes C_k) \otimes (1 \otimes B_k \otimes 1)
\end{align*}
The endomorphism $\varphi^{(k)}_{g,z_1,z_2}$ fits into the following commutative diagram
\[
\begin{tikzcd}[column sep=2.8cm, row sep=1.5cm]
	\MFa(g,z_1,z_2) \otimes \MFa^{\otimes k} \ar[r,"A \otimes B \mapsto A \otimes B \otimes 1"] \ar[d,"\varphi^{(k)}_{g,z_1,z_2}"] & \MFa(g,z_1,z_2) \otimes \MFa^{\otimes (k+1)}  \ar[d,"\varphi^{(k+1)}_{g,z_1,z_2}"] \\
		\MFa^{\otimes (k+1)} \ar[r,"B \mapsto B \otimes 1"] \ar[ur,"\psi^{(k+1)}_{g,z_1,z_2}"] & \MFa^{\otimes (k+2)} 
\end{tikzcd}
\]
where $\psi^{(k)}_{g,z_1,z_2}$ is an endomorphism constructed analogously to $\varphi^{(k)}_{g,z_1,z_2}$ by conjugating the endomorphism 
\[
\begin{tikzcd}
 (\MFa^{1,\prec}(g,z_1) \otimes \MFa(g,z_1,z_2) \otimes \MFa^{\succ}(g,z_2))^{\otimes k} \ar[d,"\beta^{(k)}_{g,z_1,z_2}"] \\
\MFa(g,z_1,z_2) \otimes (\MFa^{1,\prec}(g,z_1) \otimes \MFa(g,z_1,z_2) \otimes \MFa^{\succ}(g,z_2))^{\otimes k} 
\end{tikzcd}
\]
given by 
\begin{align*}
	& \beta^{(k)}_{g,z_1,z_2}((A_1 \otimes B_1 \otimes C_1) \otimes \dots \otimes (A_k \otimes B_k \otimes C_k)) \\
	=\ & B_1 \otimes (A_1 \otimes B_2 \otimes C_1) \otimes \dots \otimes (A_{k-1} \otimes B_k \otimes C_{k-1}) \otimes (A_k \otimes 1 \otimes C_k)
\end{align*}
with the corresponding tensor products of $\tau$. We define the homomorphisms
\begin{align*}
	\varphi_{g,z_1,z_2} & \colon \MFa(g,z_1,z_2) \otimes \MFa^{\infty} \to \MFa^{\infty} \\
	\psi_{g,z_1,z_2} & \colon \MFa^{\infty} \to \MFa(g,z_1,z_2) \otimes \MFa^{\infty}	
\end{align*}
as the ones induced by $\varphi^{(k)}_{g,z_1,z_2}$ and $\psi^{(k)}_{g,z_1,z_2}$ on the colimits. The diagram above shows that $\varphi_{g,z_1,z_2}$ and $\psi_{g,z_1,z_2}$ are inverse to each other. The associativity condition stated above can be seen from the colimit of the following commutative diagram 
\[
\begin{tikzcd}[column sep=2.7cm]
	\MFa(g,z_1,z_2) \!\otimes\! \MFa(g,z_2,z_3) \!\otimes\! \MFa^{\otimes k} \ar[r,"\id{}\!\otimes\varphi^{(k)}_{g,z_2,z_3}"] \ar[d] & 
	\MFa(g,z_1,z_2) \!\otimes\! \MFa^{\otimes (k+1)} \ar[d,"\varphi^{(k+1)}_{g,z_1,z_2}"] \\
	\MFa(g,z_1,z_3) \!\otimes\! \MFa^{\otimes (k+1)} \ar[r,"\varphi^{(k+1)}_{g,z_1,z_3}"] & 
	\MFa^{\otimes (k+2)}		
\end{tikzcd}
\]  
in which the vertical arrow on the left is the map from \eqref{eqn:MF_tensor_prod} tensored with the homomorphism $A \mapsto A \otimes 1$.
\end{proof}

\begin{corollary} \label{cor:phi_plus}
Let $E \to \cG_+$ be the vector bundle with fibre $E(g,z_1,z_2)$ over $(g,z_1,z_2) \in \cG_+$. The isomorphisms $\varphi_{g,z_1,z_2}$ constructed in Lemma~\ref{lem:associative_isos_MF} yield a continuous isomorphism of $C^*$-algebra bundles of the form
\[
	\varphi_+ \colon \Endo{F(E)} \otimes \MF \to \cG_+ \times \MF \ .
\]
\end{corollary}

\begin{proof}
	Since $\MFa(g,z_1,z_2) = \Endo{F(E(g,z_1,z_2))}$, the isomorphisms from Lemma~\ref{lem:associative_isos_MF} indeed piece together to give a map $\varphi_+$ as described in the statement. Therefore the only issue left to prove is continuity of $\varphi_+$. First observe that $E$ is by definition a subbundle of the trivial bundle $\cG_+ \times \C^n$. Its orthogonal complement $E^{\perp}$ is a locally trivial vector bundle as well. By continuity of $F$ we obtain locally trivial bundles $F(E)$ and $F(E^\perp)$. Let $\MF(E)$, respectively $\MF(E^\perp)$, be the UHF-algebras obtained as the fibrewise infinite tensor product of $\Endo{F(E)}$, respectively $\Endo{F(E^\perp)}$, and note that the $*$-homomorphism $\tau$ from \eqref{eqn:tau} translates into a continuous isomorphism of $C^*$-algebra bundles
	\[
		\begin{tikzcd}[column sep=2cm]
			\cG_+ \times \MF \ar[r,"\tau"] & \MF(E) \otimes \MF(E^\perp)	
		\end{tikzcd}
	\]
	The maps $\alpha^{(k)}_{g,z_1,z_2}$ from Lemma~\ref{lem:associative_isos_MF} induce another continuous isomorphism of $C^*$-algebra bundles: 
	\[
		\alpha^{\infty} \colon \Endo{F(E)} \otimes \MF(E) \otimes \MF(E^\perp) \to \MF(E) \otimes \MF(E^\perp)
	\]
	which shifts $\Endo{F(E)}$ into the tensor factor $\MF(E)$. By definition $\varphi_+$ is obtained by conjugating $\alpha^{\infty}$ by $\tau$ and therefore is continuous.
\end{proof}

Let $E \to \cG_+$ be the vector bundle from Cor.~\ref{cor:phi_plus}. Let $\cE_0 = \cG_0 \times \MF$, 
\[
	\cE_+ = F(E) \otimes \MF
\]
and let $\cE_{0,+} = \cE_0 \cup \cE_+$. This is a locally trivial bundle of full right Hilbert $\MF$-modules, where $\MF$ acts by right multiplication on itself. The bundle of compact adjointable right $\MF$-linear operators on $\cE_+$ agrees with
\[
	\Endo{F(E)} \otimes \MF\ ,
\]
which we can identify with $\MF$ using $\varphi_+$ to define a left $\MF$-module structure on the fibres of $\cE_+$ given by $a \cdot (\xi \otimes b) := \varphi_+^{-1}(a)(\xi \otimes b)$. To turn $\cE_{0,+}$ into a saturated half-bundle we need to equip it with a bilinear and associative multiplication $\mu$. On $\cE_+$ we define $\mu$ by the following diagram:
\[
	\begin{tikzcd}[column sep=0.1cm]
		(F(E_{g,z_1,z_2}) \otimes \MF) \otimes_{\MF} (F(E_{g,z_2,z_3}) \otimes \MF) \ar[r,"\mu"] \ar[d, "\kappa" left, "\cong" right] & F(E_{g,z_1,z_3}) \otimes \MF \\
		F(E_{g,z_1,z_2}) \otimes F(E_{g,z_2,z_3}) \otimes \MF \ar[r,"\tau \otimes \id{}" below, "\cong" above] & F(E_{g,z_1,z_2} \oplus E_{g,z_2,z_3}) \otimes \MF \ar[u,equal]
	\end{tikzcd}
\]
where the map $\kappa$ is given by
\(
	\kappa((\xi \otimes a) \otimes (\eta \otimes b)) = \xi \otimes a\cdot (\eta \otimes b)\ .
\)
This is an isomorphism with inverse $\xi \otimes \eta \otimes a \mapsto (\xi \otimes 1) \otimes (\eta \otimes a)$. Let 
\[
	\ell_{z_i,z_j} \colon \MF \otimes F(E_{z_i,z_j}) \otimes \MF \to F(E_{z_i,z_j}) \otimes \MF
\] 
be defined by left multiplication. The associativity condition in Lemma~\ref{lem:associative_isos_MF} implies that the following diagram commutes: 

\[
	\begin{tikzcd}[column sep=1.5cm]
		\MF \otimes F(E_{z_1,z_2}) \otimes \MF \otimes F(E_{z_2,z_3}) \otimes \MF \ar[r,"\circled{1}"] \ar[d,"\ell_{z_1,z_2} \otimes \id{}" left] & \MF \otimes F(E_{z_1,z_3}) \otimes \MF \ar[d,"\ell_{z_1,z_3}"] \\
		F(E_{z_1,z_2}) \otimes \MF \otimes F(E_{z_2,z_3}) \otimes \MF \ar[r,"(\tau \otimes \id{}) \circ (\id{}\! \otimes \ell_{z_2,z_3})" below] & F(E_{z_1,z_3}) \otimes \MF
	\end{tikzcd}
\]
where we eliminated the group element $g$ from the notation for clarity, and where the map~$\circled{1}$ is given by $\id{\MF} \otimes ((\tau \otimes \id{\MF}) \circ (\id{F(E_{z_1,z_2})} \otimes \,\ell_{z_2,z_3}))$. As a consequence we obtain that the multiplication $\mu$ is associative on $\cE_+$. 

On $\cE_0$ we define the multiplication by the composition of the groupoid elements and the multiplication in $\MF$, which is clearly bilinear and associative.  If $(g,z_1,z_2) \in \cG_+$ and $(g,z_2,z_3) \in \cG_0$, then by definition $E(g,z_1,z_2)$ and $E(g,z_1,z_3)$ agree. This identification and the right multiplication by $\MF$ defines the multiplication on elements from the set 
\[
	\{ (e_1,e_2) \in \cE_{0,+}^{2}\ |\ \pi(e_1) \in \cG_+, \pi(e_2) \in \cG_0, (\pi(e_1),\pi(e_2)) \in \cG^{(2)} \}\ .
\]
Using the left multiplication by $\MF$ we can also extend $\mu$ over 
\[
	\{ (e_1,e_2) \in \cE_{0,+}^{2}\ |\ \pi(e_1) \in \cG_0, \pi(e_2) \in \cG_+, (\pi(e_1),\pi(e_2)) \in \cG^{(2)} \}\ .
\]
in an analogous way. The resulting multiplication map is still associative. The fibrewise inner products on the right Hilbert $A$-modules yield a global continuous inner product, i.e.\ for the elements $\xi \otimes a,\ \eta \otimes b \in F(E(g,z_1,z_2)) \otimes \MF$ we define
\[	
	\rscal{\xi \otimes a}{\eta \otimes b}{A} = \left(\rscal{\xi}{\eta}{\C}\,a^*b, (g,z_2)\right)
\]
This ensures that all of the properties in Def.~\ref{def:sat_half_bundle} b) hold. The multiplication also satisfies Def.~\ref{def:sat_half_bundle} c) by construction. Thus, we have proven:
 
\begin{lemma} \label{lem:cE_is_a_half_bundle}
	The triple $(\cE_{0,+}, \mu, \rscal{\,\cdot\,}{\,\cdot\,}{A})$ constructed above is a saturated half-bundle.  
\end{lemma}

\begin{remark}
	Note that $\Endo{F(\C^n)} = F(\C^n) \otimes F(\C^n)^* \cong F(\C^n \oplus (\C^n)^*)$ and this isomorphism is natural. As a module over itself the algebra $\MF$ morally turns out to be
	\[
		\MF \cong F( \C^n \oplus (\C^n)^* \oplus \C^n \oplus (\C^n)^* \oplus \dots )\ ,
	\] 
	where the left action is on the $\C^n$-summands and the right action on their dual spaces. Likewise, we have
	\[
		\cE_{(g,z_1,z_2)} \cong F( E(g,z_1,z_2) \oplus \C^n \oplus (\C^n)^* \oplus \C^n \oplus (\C^n)^* \oplus \dots )\ .
	\]
\end{remark}

\subsection{The group action on $\cE_{0,+}$} \label{subsec:groupaction} 
The group $G = SU(n)$ acts on $\cG = Y^{[2]}$ by conjugation, i.e.\ for $h \in G$ and $(g,z_1,z_2) \in \cG$ we define
\[
	h\cdot (g,z_1,z_2) = (hgh^{-1}, z_1, z_2)\ .
\]
Observe that conjugation is a group automorphism and does not change the set of eigenvalues. Therefore this action is admissible in the sense of Def.~\ref{def:admissible_action}. Let $(g,z_1,z_2) \in \cG_{0,+}$. Any element $h \in G$ defines an isomorphism 
\[
	h \colon E(g,z_1,z_2) \to E(hgh^{-1},z_1,z_2) \quad , \quad \xi \mapsto h\xi\ ,
\]
where $h$ acts on $\Eig{g}{\lambda} \subset \C^n$ using the standard representation of $SU(n)$. The exponential functor $F$ turns this into a unitary isomorphism
\[
	F(h) \colon F(E(g,z_1,z_2)) \to F(E(hgh^{-1},z_1,z_2))\ .
\] 
The naturality of the structure isomorphism $\tau$ of $F$ ensures that the following diagram commutes:
\[
\begin{tikzcd}
	F(E(g,z_1,z_2)) \otimes F(E(g,z_2,z_3)) \ar[d,"F(h) \otimes F(h)" left] \ar[r,"\tau"] & F(E(g,z_1,z_3)) \ar[d,"F(h)"] \\
	F(E(hgh^{-1},z_1,z_2)) \otimes F(E(hgh^{-1},z_2,z_3)) \ar[r,"\tau" below] & F(E(hgh^{-1},z_1,z_3))
\end{tikzcd}
\]
Similarly, $G$ acts by conjugation on $\MFa$ and therefore also on the infinite tensor product $\MF$. Denote this action by $\alpha \colon G \to 
\Aut{\MF}$. This turns $\MF$ into a $G$-algebra. Combining $F(h)$ and $\alpha$ we obtain isomorphisms of $A$-$A$ bimodules 
\begin{equation} \label{eqn:G-action_on_fibres}
	F(E(g,z_1,z_2)) \otimes \MF \to F(E(hgh^{-1},z_1,z_2)) \otimes \MF
\end{equation}
inducing a continuous action of $G$ on $\cE_{0,+}$ covering the action of $G$ on $\cG_{0,+}$. 

\begin{lemma} \label{lem:associative_isos_equiv}
Let $E \to \cG_{+}$ be the vector bundle with fibre $E(g,z_1,z_2)$ over $(g,z_1,z_2) \in \cG_+$. The isomorphism $\varphi_+ \colon \Endo{F(E)} \otimes \MF \to \cG_+ \times \MF$ constructed in Cor.~\ref{cor:phi_plus} is $G$-equivariant (where $h \in G$ acts on $\Endo{F(E)}$ via $\mathrm{Ad}_{F(h)}$ and on $\MF$ via $\alpha_h$). In particular, \eqref{eqn:G-action_on_fibres} is an isomorphism of bimodules and $\mu$ from Lem.~\ref{lem:cE_is_a_half_bundle} is $G$-equivariant.
\end{lemma}

\begin{proof}
	We use the notation introduced in Cor.~\ref{cor:phi_plus}. The action of $h \in SU(n)$ maps the eigenspace $\Eig{g}{\lambda}$ unitarily onto $\Eig{hgh^{-1}}{\lambda}$. This induces the given action of $G$ on $E$ and another unitary action of $G$ on $E^\perp$ in such a way that $\cG_+ \times \C^n = E \oplus E^\perp$ is an equivariant direct sum decomposition. With respect to the induced actions on $\MF(E)$ and $\MF(E^\perp)$ the isomorphism $\tau \colon \cG_+ \times \MF \to \MF(E) \otimes \MF(E^\perp)$ from Cor.~\ref{cor:phi_plus} is $G$-equivariant. Since $G$ acts in the same way on each tensor factor of the infinite tensor product $\MF(E)$ the shift isomorphism $\alpha^{\infty}$ from Cor.~\ref{cor:phi_plus} is equivariant as well. But these are the building blocks of $\varphi_+$. Thus, this implies the statement.
\end{proof}

Combining Thm.~\ref{thm:extension_of_Fell_bdls} and Cor.~\ref{cor:G-action_on_ext} we obtain the main result of this section:

\begin{corollary} \label{cor:Fell_bundle}
	The triple $(\cE_{0,+}, \mu, \rscal{\,\cdot\,}{\,\cdot\,}{A})$ together with the $G$-action defined above is a $G$-equivariant saturated half-bundle in the sense of Def.~\ref{def:eq_sat_half_bundle}. In particular, there is a saturated Fell bundle $\pi \colon \cE \to \cG$ with the properties 
	\begin{enumerate}[a)]
		\item $\left.\cE\right|_{\cG_{0,+}} = \cE_{0,+}$,
		\item $e_1^* \cdot e_2 = \rscal{e_1}{e_2}{A}$ for all $e_1,e_2\in \cE$ lying in the same fibre,
		\item the group $G = SU(n)$ acts continuously on $\cE$ such that $\pi \colon \cE \to \cG$ is equivariant and $g \cdot e^* = (g \cdot e)^*$.
	\end{enumerate}
	The Fell bundle $\cE$ is unique up to isomorphism.
\end{corollary}

\begin{remark}
Note that for $F = \extp^{\text{top}}$ the algebra $\MF$ agrees with $\C$ and $F(E)$ is the determinant line bundle of $E$. Moreover, the fibre~$(\cE_{-})_{(g,z_1,z_2)}$ can be identified with~$(\cE_{+}^*)_{(g,z_2,z_1)}$. Thus, our definition generalises the equivariant basic gerbe constructed in \cite{MurrayStevenson-basic_gerbe:2008}, which is based on \cite{Meinrenken-basicgerbe:2003,Mickelsson:2003}. 
\end{remark}

\section{The $C^*$-algebra associated to $\cE$}
In this section we will review the construction of the $C^*$-algebra $C^*(\cE)$ associated to the Fell bundle $\cE$. A priori there are several $C^*$-completions of the section algebra of $\cE$, but an amenability argument shows that all of them have to agree. We will also see that $C^*(\cE)$ is a continuous $C(G)$-algebra, which is stably isomorphic to a section algebra of a locally trivial bundle of $C^*$-algebras\footnote{The fact that we only get a stable isomorphism is in line with the classical case. The proof relies on a generalisation of the Fell condition, which only works after stabilisation, or equivalently is a statement up to Morita-Rieffel equivalence.}. As a consequence we obtain a Mayer-Vietoris sequence in equivariant $K$-theory. 

We start by reviewing the construction of the reduced $C^*$-algebra associated to the Fell bundle $\cE$. Let $A = C_0(Y, \MF)$ where $Y$ is as in Sec.~3.2. We can equip the space of compactly supported sections $C_c(Y^{[2]},\cE)$ with an $A$-valued inner product as follows:
\[
	\rscal{\sigma}{\tau}{A}(g,z) = \int_{\bT \setminus \{1\}} \sigma(g,w,z)^* \cdot \tau(g,w,z) dw\ ,  
\]
where $\sigma, \tau \in C_c(Y^{[2]},\cE)$, the dot denotes the Fell bundle multiplication and we used the Lebesgue measure on $\bT \setminus \{1\}$ with respect to which the subset $\EV{g} \cap (\bT \setminus \{1\})$ is of measure zero. The space $C_c(Y^{[2]},\cE)$ also carries a natural right $A$-action given for $a \in A$ and $\sigma \in C_c(Y^{[2]},\cE)$ by
\[
	(\sigma \cdot a)(g,z_1,z_2) = \sigma(g,z_1,z_2) \cdot a(g,z_2)\ .
\]
Denote by $L^2(\cE)$ the completion of $C_c(Y^{[2]},\cE)$ to a right Hilbert $A$-module with respect to the norm 
\[
	\lVert \sigma \rVert^2 = \sup_{(g,z) \in Y} \lVert\rscal{\sigma}{\sigma}{A}(g,z)\rVert\ .
\]
The space $C_c(Y^{[2]},\cE)$ can also be equipped with a convolution product, which assigns to $\sigma, \tau \in C_c(Y^{[2]},\cE)$ the section
\[
	(\sigma \ast \tau)(g,z_1,z_2) = \int_{\bT \setminus \{1\}} \sigma(g,z_1,w) \cdot \tau(g,w,z_2)\,dw
\]
Likewise, the $*$-operation on the Fell bundle induces an involution that maps $\sigma \in C_c(Y^{[2]},\cE)$ to 
\(
	\sigma^*(g,z_1,z_2) = \sigma(g,z_2,z_1)^*
\) and we have 
\[
	\rscal{\sigma \ast \tau_1}{\tau_2}{A} = \rscal{\tau_1}{\sigma^* \ast \tau_2}{A}\ ,
\]
i.e.\ convolution by $\sigma$ is an adjointable and hence bounded operator on $L^2(\cE)$. Let $\rbdd{A}{L^2(\cE)}$ be the adjointable right $A$-linear operators on the Hilbert $A$-module $L^2(\cE)$. By the above considerations we obtain a well-defined $*$-homomorphism 
\[
	C_c(Y^{[2]},\cE) \to \rbdd{A}{L^2(\cE)}\ .
\]
\begin{definition} \label{def:red_CStar_of_E}
We define $C^*_r(\cE)$ to be the $C^*$-algebra obtained as the norm-closure of $C_c(Y^{[2]},\cE)$ in $\rbdd{A}{L^2(\cE)}$. It is called the \emph{reduced $C^*$-algebra associated to the Fell bundle $\cE$}.  	
\end{definition}
Denote by $C_{\text{max}}^*(\cE)$ the maximal cross-sectional $C^*$-algebra of $\cE$. The grou\-poid $\cG = Y^{[2]}$ is equivalent in the sense of Renault to the trivial groupoid 
\[
\begin{tikzpicture}
	\node (G1) at (0,0) {$G$};
	\node (G2) at (1.5,0) {$G$};
	\draw[-latex] ([yshift=2]G1.east) -- ([yshift=2]G2.west) node[midway,above] {\scriptsize id};
	\draw[-latex] ([yshift=-2]G1.east) -- ([yshift=-2]G2.west)node[midway,below] {\scriptsize id};
\end{tikzpicture}
\]
which is (topologically) amenable. Since amenability is preserved by equivalence, the same is true for $Y^{[2]}$. Therefore \cite[Thm.~1]{SimsWilliams-Amenability:2013} implies that the reduced and the universal norm agree on $C_c(Y^{[2]},\cE)$ and thus $C_{\text{max}}^*(\cE) \cong C^*_r(\cE)$. Hence, we will drop the subscript from now on and write $C^*(\cE)$ for this $C^*$-algebra.

Observe that $C^*(\cE)$ carries a continuous $G$-action defined on sections $\sigma \in C_c(Y^{[2]},\cE)$ by
\[
	(g \cdot \sigma)(h,z_1,z_2) = g \cdot \sigma(g^{-1}hg, z_1,z_2)\ .
\]
It is also a $C(G)$-algebra in a natural way via the action that is defined on sections $\sigma \in C_c(Y^{[2]},\cE)$ with $f \in C(G)$ as follows 
\[
	(f\cdot \sigma)(g,z_1,z_2) = f(g)\sigma(g,z_1,z_2)\ .
\]
Note that this is indeed central and therefore provides a $*$-homomorphism $C(G) \to Z(M(C^*(\cE)))$
\begin{lemma} \label{lem:C(G)-algebra}
The multiplication by elements in $C(G)$ defined above turns $C^*(\cE)$ into a continuous $C(G)$-algebra. For $g \in G$ let $Y^{[2]}_g$ be the subgroupoid defined by
\[
	Y^{[2]}_g = \left(\pi^{-1}(g)\right)^{[2]} \ .
\]
and let $\cE_g = \left. \cE \right|_{Y^{[2]}_g}$. Then the fibre of $C^*(\cE)$ over $g$ is given by $C^*(\cE_g)$.
\end{lemma}

\begin{proof}
We can identify $G$ with the orbit space of the action of $Y^{[2]}$ on $Y$. Thus, \cite[Cor.~10]{SimsWilliams-Amenability:2013} implies that $C^*(\cE)$ is a $C(G)$-algebra with fibres $C^*(\cE_g)$. The only statement left to show is that $g \mapsto \lVert a_g \rVert$ is lower semi-continuous for every $a \in C^*(\cE)$, where $a_g$ denotes the image of $a$ in $C^*(\cE_g)$. Without loss of generality we may assume that $a$ is a section $\sigma \in C_c(Y^{[2]},\cE)$. Let $g \in G$ and $\epsilon > 0$. Denote by $\tau_g \in L^2(\cE_g)$ the restriction of $\tau \in L^2(\cE)$. Note that 
\[
	\lVert \tau_g \rVert_{L^2(\cE_g)}^2 = \sup_{z \in \bT}\,\lVert \rscal{\tau}{\tau}{A}(g,z) \rVert\ .
\] 
Take $\tau \in L^2(\cE)$ with $\lVert \tau \rVert_{L^2(\cE)} = 1$, $\lVert \tau_g \rVert_{L^2(\cE_g)} = 1$ and 
\[
	\rVert (\sigma \ast \tau)_g \lVert_{L^2(\cE_g)} \geq \lVert \sigma_g \rVert_{C^*(\cE_g)} - \frac{\epsilon}{2}\ .
\] 
Since the inner product on $L^2(\cE)$ takes values in $A = C_0(Y, \MF)$, the function $f \colon Y \to \R$ given by 
\(
	f(h,z) = \lVert \rscal{\sigma \ast \tau}{\sigma \ast \tau}{A}(h,z) \rVert
\)
is continuous and extends to $G \times \bT$. Since $\bT$ is compact, there is $z_0 \in \bT$ with $(g,z_0) \in Y$ and $f(g,z_0) = \sup_{z \in \bT} f(g,z)$. By continuity of $h \mapsto f(h,z_0)$ there is an open neighbourhood $U$ of $g$ such that for all $h \in U$
\[
	\lVert \rscal{\sigma \ast \tau}{\sigma \ast \tau}{A}(h,z_0) \rVert^{\frac{1}{2}} \geq \lVert \rscal{\sigma \ast \tau}{\sigma \ast \tau}{A}(g,z_0) \rVert^{\frac{1}{2}} - \frac{\epsilon}{2} \geq  \lVert \sigma_g \rVert_{C^*(\cE_g)} - \epsilon\ .
\] 
But $\lVert \rscal{\sigma \ast \tau}{\sigma \ast \tau}{A}(h,z_0) \rVert^{\frac{1}{2}} \leq \lVert (\sigma \ast \tau)_h\rVert_{L^2(\cE_h)}$ and since $\lVert \tau_h \rVert_{L^2(\cE_h)} \leq 1$ we have
\[
\lVert \sigma_h \rVert_{C^*(\cE_h)} \geq \lVert \sigma_g \rVert_{C^*(\cE_g)} - \epsilon
\] 
for all $h \in U$, which shows that the map is lower semi-continuous.
\end{proof}

We are going to prove that $C^*(\cE)$ is stably isomorphic to the section algebra of a locally trivial bundle with fibre $\MF \otimes \bK$. The following lemma will provide a first step and shows that local sections of $\pi \colon Y \to G$ give rise to trivialisations via Morita equivalences. 

\begin{lemma} \label{lem:equiv_trivialisation}
	Let $\sigma \colon V \to Y$ be a continuous section of $\pi \colon Y \to G$ over a closed subset $V \subset G$. Let $Y_V = \pi^{-1}(V)$. Denote the corresponding restriction of $\cG$ (respectively $\cE$) by $\cG_V$ (respectively $\cE_V$). Let $p_{\bT} \colon Y \to \bT$ be the restriction of the projection map to $Y$, let $t = p_{\bT} \circ \sigma$ and let
	\[
		\iota \colon Y_V \to \cG_V \qquad, \qquad (g,z) \mapsto (g,z,t(g))
	\]
	The Banach bundle $\cC_V = \iota^*\cE_V$ gives rise to a Morita equivalence $\sfX_V$ between $C^*(\cE_V)$ and $C(V, \MF)$. If $V$ is $G$-invariant, then $\sfX_V$ is a $G$-equivariant Morita equivalence. 
\end{lemma}

\begin{proof}
We will prove the first part of the statement by showing that $\cC_V$ provides an equivalence of Fell bundles in the sense of \cite[Sec.~6]{MuhlyWilliams-Disintegration:2008} between $\cE_V$ and the $C^*$-algebra bundle $V \times \MF$ over $V$. 

First note that the space $Y_V$ is an equivalence between $\cG_V$ and the trivial groupoid over $V$, which we will also denote $V$ by a slight abuse of notation. Let $\beta \colon V \to \cG$ be given by $\beta(g) = (g, t(g), t(g))$. We can and will identify $V \times \MF$ with $\beta^*\cE_V$. The bundle $\kappa \colon \cC_V \to Y_V$ carries a left action of $\cE_V$ and a right action of $\beta^*\cE_V = V \times \MF \to V$ such that \cite[Def.~6.1~(a)]{MuhlyWilliams-Disintegration:2008} holds. The two sesquilinear forms
\begin{align*}
	\cC_V \times \cC_V \to \cE_V \quad &, \quad (c,d) \mapsto  \lscal{\cE_V}{c}{d} := c \cdot d^* \ ,\\
	\cC_V \times_{\kappa} \cC_V \to V \times \MF \quad &, \quad (c,d) \mapsto  \rscal{c}{d}{\beta^*\cE_V}  := c^* \cdot d
\end{align*}
satisfy the conditions listed in \cite[Def.~6.1~(b)]{MuhlyWilliams-Disintegration:2008}. Since $\cE_V \to \cG_V$ is a saturated Fell bundle, \cite[Def.~6.1~(c)]{MuhlyWilliams-Disintegration:2008} is also true for the bundle $\cC_V \to Y_V$. Therefore, by \cite[Thm.~6.4]{MuhlyWilliams-Disintegration:2008} the completion $\sfX_V$ of $C_c(Y_V,\cC_V)$ with respect to the norms induced by the above inner products is an imprimitivity bimodule between the $C^*$-algebras $C^*(\cE_V)$ and $C(V,\MF)$.

For the proof of the second part suppose that $V$ is $G$-invariant. Note that the adjoint action lifts to $Y_V$. Moreover, the action described in Sec.~\ref{subsec:groupaction} restricts to $C^*(\cE_V)$. Let $\alpha \colon G \to \Aut{\MF}$ be the action of $G$ on $\MF$ induced by $\Ad_{F(\rho)}$ where $\rho \colon G \to U(n)$ is the standard representation. Then $G$ acts on $C(V, \MF)$ via the adjoint action on $V$ and by $\alpha$ on $\MF$. It is straightforward to check on sections that these definitions turn $\sfX_V$ into an equivariant Morita equivalence.
\end{proof}

\begin{remark}
	By \cite[Lem.~9]{SimsWilliams-Amenability:2013} the $C(G)$-algebra structure is compatible with the Fell bundle restriction in the sense that restricting the sections to~$\cG_V$ induces a natural $*$-isomorphism $C^*(\cE)(V) \cong C^*(\cE_V)$. 
\end{remark}

\begin{definition} \label{def:Fell_condition}
	Let $X$ be a locally compact metrisable space. A continuous $C_0(X)$-algebra $A$ whose fibres are stably isomorphic to strongly self-absorbing $C^*$-algebras is said to satisfy the \emph{(global) generalised Fell condition} if for each $x \in X$ there exists a closed neighbourhood $V$ of $x$ and a projection $p \in A(V)$ such that $[p(v)] \in GL_1(K_0(A(v)))$ for all $v \in V$. 
\end{definition}

\begin{lemma} \label{lem:Fell_condition}
	The fibre algebra $C^*(\cE_g)$ is Morita equivalent to $\MF$ and the continuous $C(G)$-algebra $C^*(\cE)$ satisfies the generalised Fell condition.
\end{lemma}

\begin{proof}
Let $V = \{g\} \subset G$ and choose $z_0 \in \bT$ such that $(g,z_0) \in Y$. The first statement is now a consequence of Lemma~\ref{lem:equiv_trivialisation} for $V = \{g\} \subset G$ and $\sigma \colon \{g\} \to Y$ given by $\sigma(g) = (g,z_0)$. For any $g \in G$ let $\sfX_g$ be the resulting Morita equivalence between $C^*(\cE_g)$ and $\MF$. 

It remains to be proven that $C^*(\cE)$ satisfies the generalised Fell condition. Let $g \in G$ and choose an open neighbourhood $U$ of $g$ with the property that 
\[
	S = \{ z \in \bT \setminus \{1\}\ |\ z \notin \EV{h} \text{ for any } h \in U \} 
\]
contains an open interval $J \subset S$. Note that $U \times J^2 \subset Y^{[2]}$. Since there are no eigenvalues in between any two points of~$J$, the restriction of $\cE$ to this subspace is just the trivial bundle with fibre~$\MF$. Thus, extension of a section by $0$ produces an inclusion of convolution algebras
\[
	C_c(U \times J^2, \MF) \to C_c(Y^{[2]},\cE)
\]
and the completion of the left hand side in the representation on $L^2(\cE)$ is isomorphic to $C_0(U, \bK(L^2(J)) \otimes \MF)$. The resulting $*$-homomorphism 
\[
	C_0(U, \bK \otimes \MF) \to C^*(\cE)
\]
is an inclusion of $C(G)$-algebras. Pick a closed neighbourhood $V \subset U$ of~$g$. If we restrict both sides to $V$ we obtain $C(V, \bK \otimes \MF) \to C^*(\cE)(V)$. Let $e \in \bK$ be a rank $1$-projection and define $p \in C^*(\cE)(V)$ to be the image of $1_{C(V,\MF)} \otimes e$ with respect to this inclusion. Fix $v \in V$. The isomorphism 
\[
	K_0(C^*(\cE)(v)) \cong K_0(C^*(\cE_v)) \cong K_0(\MF)
\]
induced by the Morita equivalence $\sfX_v$ maps the $K$-theory element $[p(v)] \in K_0(C^*(\cE)(v))$ to the class of the right Hilbert $\MF$-module $p(v)\sfX_v$ in $K_0(\MF)$.  We can choose the value of $z_0 \in \bT$ used to define $\sfX_v$ such that $z_0 \in J$. Moreover, we can without loss of generality assume that $e \in \bK(L^2(J))$ is the projection onto the subspace spanned by a compactly supported function $f \in C_c(J) \subset L^2(J)$. Then we have 
\[
	p(v)\sfX_v \cong p(v) \overline{C_c(J, \MF)}^{\lVert\cdot\lVert_{L^2}} \cong eL^2(J) \otimes \MF \cong \MF\ ,
\] 
which represents the unit in $K_0(\MF)$ and is therefore invertible.
\end{proof}

\begin{corollary} \label{cor:CStarBundle}
The continuous $C(G)$-algebra $C^*(\cE)$ is stably isomorphic to the section algebra of a locally trivial bundle $\cA \to G$ of $C^*$-algebras with fibre $\MF \otimes \bK$. In particular, it is classified by a continuous map 
\[
	G \to B\Aut{\MF \otimes \bK} \simeq BGL_1\left(KU\left[d_F^{-1}\right]\right) 
\] 
where $d_F = \dim(F(\C))$. 
\end{corollary}

\begin{proof}
	By Lemma~\ref{lem:Fell_condition} the algebra $C^*(\cE)$ satisfies the generalised Fell condition and its fibres are Morita equivalent to the infinite UHF-algebra $\MF$. Therefore the statement follows from \cite[Cor.~4.9]{DadarlatP-DixmierDouady:2016}.
\end{proof}

\subsection{The spectral sequence}
For $G = SU(n)$ let $\ell = n-1$ be the rank of~$G$. Choose a maximal torus $\bT^{\ell}$ of $G$ with Lie algebra $\liet$. Let $\Lambda \subset \liet$ be the integral lattice with dual lattice $\Lambda^*$. Denote by $\rscal{\,\cdot\,}{\,\cdot\,}{\lieg}$ the basic inner product on $\lieg$. Choose a collection $\alpha_1, \dots, \alpha_{\ell} \in \Lambda^*$ of simple roots and let 
\[
	\liet_+ = \left\{ \xi \in \liet \ |\ \rscal{\alpha_j}{\xi}{\lieg} \geq 0 \ \forall j \in \{0, \dots, \ell\} \right\}
\]
be the corresponding positive Weyl chamber. Let $\Delta^{\ell}$ be the standard $\ell$-simplex defined as
\[
	\Delta^{\ell} = \left\{ (t_0, \dots, t_{\ell}) \in \R^{\ell+1} \ |\ \sum_{i=0}^\ell t_i = 1 \text{ and } t_j \geq 0\  \forall j \in \{0,\dots, \ell\} \right\}
\]
This simplex can be identified with the fundamental alcove of $G$, which is the subset cut out from $\liet_+$ by the additional inequality $\rscal{\alpha_0}{\xi}{\lieg} \geq -1$, where $\alpha_0$ is the lowest root. The alcove parametrises conjugacy classes of $G$ in the sense that each such class contains a unique element $\exp(\xi)$ with $\xi \in \Delta^\ell$. Denote the corresponding continuous quotient map by
\[
	q \colon G \to \Delta^\ell\ .
\] 
A sketch of the situation in the case $n = 3$ is shown in Fig.~\ref{fig:alcoveSU(3)}.

\begin{figure}[htp]
\centering
\begin{tikzpicture}[scale=1.8]
	\coordinate (0) at (0,0);
	\coordinate (alpha2) at (0,{sqrt(2)});
	\coordinate (alpha1) at ({sqrt(2)*sin(120)}, {sqrt(2)*cos(120)});
	\coordinate (alpha3) at ($(alpha1)+(alpha2)$);
	
	\coordinate (mu0) at ($1/3*(alpha1) + 2/3*(alpha2)$);
	\coordinate (mu1) at ($2/3*(alpha1) + 1/3*(alpha2)$);
	
	\draw[-latex] (0) -- (alpha2) node[midway,left] {$\alpha_2\!$};
	\draw[-latex] (0) -- (alpha1) node[midway,right] {$\alpha_1$};
	\draw[-latex] (0) -- ($-1*(alpha3)$) node[midway,right] {$\,\,\alpha_0$};
		
	\foreach \n in {-2,...,0} {
		\foreach \m in {0,...,2} {
			\draw[black!20,fill=black!20] ($-\n*(mu0) - \m*(mu1)$) circle (0.8pt);
			\draw[black!20,fill=black!20] ($\n*(mu0) + \m*(mu1)$) circle (0.8pt);
		}
	}

	\draw[blue!10,fill=blue!10] (0,0) -- ($(mu0)$) -- ($(mu1)$) -- cycle;
	\node at ($0.36*(mu0) + 0.36*(mu1)$) {$\Delta^2$};

	\draw[dashed,blue] ($(mu0)$) -- ($(mu1)$);
	\draw[dashed,blue] (0,0) -- ($2*(mu0)$);
	\draw[dashed,blue] (0,0) -- ($2*(mu1)$);

	\foreach \n in {0,...,2} {
		\pgfmathsetmacro{\k}{2-\n}
		\foreach \m in {0,...,\k} {
			\draw[blue,fill=blue] ($\n*(mu0) + \m*(mu1)$) circle (0.9pt);
		}
	}
\end{tikzpicture}
\caption{\label{fig:alcoveSU(3)}The root system of $G = SU(3)$ with positive roots $\alpha_1$ and $\alpha_2$ and lowest root $\alpha_0$. The blue dots mark the weights inside the Weyl chamber and the fundamental alcove $\Delta^2$ is shown in blue.}
\end{figure}

For a non-empty subset $I \subset \{0, \dots, \ell\}$ we let $\Delta_I \subset \Delta^{\ell}$ be the closed subsimplex spanned by the vertices in $I$. Let $\xi_I \in \lieg$ be the barycentre of $\Delta_I$ and let $G_I$ be the centraliser of $\exp(\xi_I)$. In fact, the subgroup $G_I$ does not depend on our choice of $\xi_I$ as long as it is an element in the interior of~$\Delta_I$. For $J \subset I$ we have $G_I \subset G_J$, which induces a $G$-map $G/G_I \to G/G_J$. Let 
\[
	\mathsf{G}_n = \coprod_{\lvert I \rvert = n+1} G/G_I\ .
\]
Denote the set $\{0, \dots, n\}$ by $[n]$. Let $f \colon [m] \to [n]$ be an order-preserving injective map. For each $I \subset \{0, \dots, \ell\}$ with $\lvert I \rvert = n+1$ there is a unique order-preserving identification $[n] \cong I$. Let $J \subset I$ be the subset corresponding to $f([m]) \subset [n]$ in this way. The above construction induces a continuous map
\(
	f^*_I \colon G/G_I \to G/G_J
\)
and those maps combine to 
\[
	f^* \colon \mathsf{G}_n \to \mathsf{G}_m\ .
\]
This turns $[n] \mapsto \mathsf{G}_n$ with $f \mapsto f^*$ into a contravariant functor. Therefore $\mathsf{G}_{\bullet}$ is a semi-simplicial space. The group $G$ can be identified with its geometric realisation, i.e. \ 
\[
	G \cong \lVert \mathsf{G}_{\bullet}\rVert = \left(\coprod_{I} G/G_I \times \Delta_I\right)/\!\!\sim
\]
where the equivalence relation identifies the faces of $\Delta_I$ using the maps $G/G_I \to G/G_J$ in the other component. The map $q \colon G \to \Delta^{\ell}$ is induced by the projection maps $G/G_I \times \Delta_I \to \Delta_I$ in this picture. Let
\[
	A_i = \left\{ (t_0, \dots, t_{\ell}) \in \Delta^{\ell}\ \left|\ \sum_{k \neq i} t_k \leq \delta_n \right.\right\} \subset \Delta^{\ell}\ ,
\]
where $0 < \delta_n < 1$ is chosen such that the closed sets $(A_i)_{i \in \{0,\dots,\ell\}}$ cover $\Delta^{\ell}$. Then $(V_i)_{i \in \{0,\dots,\ell\}}$ with $V_k = q^{-1}(A_k)$ is a cover of $G$ by closed sets. Note that each $V_k$ is $G$-homotopy equivalent to the open star of the $k$th vertex. For each non-empty subset $I \subset \{0,\dots, \ell\}$ let 
\[
	A_I = \bigcap_{i \in I} A_i \quad \text{and} \quad V_I = q^{-1}(A_I) = \bigcap_{i \in I} V_i\ .
\]
Note that the barycentre $\xi_I$ of $\Delta_I$ is contained in $A_I$. Therefore there is a canonical embedding $\iota_I \colon G/G_I \to V_I$. 

\begin{lemma} \label{lem:GmodGI}
	The embedding $\iota_I \colon G/G_I \to V_I$ defined above is a $G$-equi\-va\-riant deformation retract. 
\end{lemma}

\begin{proof}
Observe that $\Delta_{\{0,\dots,\ell\} \setminus \{i\}} \cap A_i = \emptyset$, which implies that $\Delta_K \cap A_j = \emptyset$ if $j \notin K$. Hence, the intersection $\Delta_K \cap A_I$ can only be non-empty if $I \subset K$. Therefore
\[
	V_I = \left(\coprod_{I \subset K} G/G_K \times (\Delta_K \cap A_I)\right)/\!\!\sim 
\]
and the quotient maps $G/G_K \to G/G_I$ induce a well-defined $G$-equivariant continuous map $r_I \colon V_I \to G/G_I$ with the property that $r_I \circ \iota_I = \id{G/G_I}$. Note that the set $A_I$ is convex and consider the contraction $H^A$ given by
	\[
		H^A \colon A_I \times [0,1] \to A_I \qquad, \qquad (\eta, s) \mapsto \xi_I + (1-s)(\eta - \xi_I) \ .
	\]
	Since $\xi_I \in \Delta_I$, each $H^A_s$ maps $A_I \cap \Delta_K$ to itself for all sets $K$ with $I \subset K$. Thus, we can lift $H^A$ to a $G$-equivariant continuous map
	\[
		H \colon V_I \times [0,1] \to V_I
	\] 
	which provides a homotopy between $\iota_I \circ r_I$ and $\id{V_I}$ that leaves $\iota_I(G/G_I)$ invariant.
\end{proof}

\begin{prop}\label{prop:spectral_seq}
	Let $\rho \colon G \to U(n)$ be the standard representation of $G$. For each non-empty subset $I \subset \{0,\dots, \ell\}$ let $\rho_I \colon G_I \to U(n)$ be the restriction of $\rho$ to $G_I$. There is a cohomological spectral sequence with $E_1$-page
	\[
		E_1^{p,q} = \bigoplus_{\lvert I \rvert = p+1} K^{G_I}_q(\MF) \cong
		\begin{cases}
			\bigoplus_{\lvert I \rvert = p+1} R(G_I)\left[F(\rho_I)^{-1}\right] & \text{for $q$ even}\ ,\\
			0 & \text{for $q$ odd}\ ,
		\end{cases}  
	\]
	where $R(H)$ denotes the representation ring of $H$. It converges to the associated graded of a filtration of $K^G_{*}(C^*(\cE))$.
\end{prop}

\begin{proof}
	The cover of $G$ by closed sets gives rise to a semi-simplicial space $\mathsf{V}_{\bullet}$ with 
	\[
		\mathsf{V}_n = \coprod_{\lvert I \rvert = n+1} V_I\ .
	\]
	Similar to the construction of $[n] \mapsto G_n$ we can turn $[n] \mapsto \mathsf{V}_n$ into a contravariant functor. Let $\cA \to G$ be the $C^*$-algebra bundle found in Cor.~\ref{cor:CStarBundle} and denote by $\cA_I \to V_I$ its restriction to $V_I$. The spaces $\cA_I$ can be assembled to form a simplicial bundle of $C^*$-algebras $\mathsf{A}_{\bullet} \to \mathsf{V}_{\bullet}$ with
	\[
		\mathsf{A}_n = \coprod_{\lvert I \rvert = n+1} \cA_I\ .
	\]
	Replacing each $V_I$ by $G$ and each $\cA_I$ by $\cA$ in the definitions of $\mathsf{V}_{\bullet}$ and $\mathsf{A}_{\bullet}$ we also obtain two `constant' semi-simplicial spaces $\mathsf{A}^c_{\bullet}$ and $\mathsf{G}^c_{\bullet}$, respectively. Their geometric realisations are $\lVert \mathsf{G}^c_{\bullet} \rVert \cong G \times \Delta^{\ell}$ and $\lVert \mathsf{A}^c_{\bullet} \rVert \cong \cA \times \Delta^{\ell}$. The canonical morphisms $\mathsf{A}_{\bullet} \to \mathsf{A}^c_{\bullet}$ and $\mathsf{V}_{\bullet} \to \mathsf{G}^c_{\bullet}$ give rise to the following diagram:
	\[
	\begin{tikzcd}
		\lVert \mathsf{A}_{\bullet} \rVert \ar[r] \ar[d] & \cA \times \Delta^{\ell} \ar[r,"\text{pr}_{\cA}"] \ar[d] & \cA \ar[d]\\
		\lVert \mathsf{V}_{\bullet} \rVert \ar[r] & G \times \Delta^{\ell} \ar[r,"\text{pr}_G" below] & G
	\end{tikzcd}
	\]
	The second square is a pullback. It is not hard to check that the first square is a pullback as well (compare this with \cite[Rem.~2.23]{HenriquesGepner-Orbispaces:2007}). Moreover, the composition $q \colon \lVert \mathsf{V}_{\bullet}\rVert \to G \times \Delta^{\ell} \to G$ in the diagram is a $G$-homotopy equivalence.
	
	For each pair $(X,f)$ where $X$ is a compact Hausdorff $G$-space and $f \colon X \to G$ is a $G$-equivariant continuous map, consider the contravariant functor
	\[
		(X,f) \mapsto K_*^G(C(X, f^*\cA))
	\] 
	from the category of compact Hausdorff $G$-spaces over $G$ to abelian groups. This functor satisfies the analogues of conditions (i) -- (iv) in \cite[\S 5]{Segal-SpecSeq:1968} in this category. Using the same argument as in the proof of \cite[Prop.~5.1]{Segal-SpecSeq:1968} we therefore end up with a spectral sequence with $E_1$-page
	\[
		E_1^{p,q} = \bigoplus_{\lvert I \rvert = p+1} K_q^G(C(V_I, \cA_I)) \cong \bigoplus_{\lvert I \rvert = p+1} K_q^G(C^*(\cE)(V_I))\ , 
	\]
	whose termination is $K^G_*(C^*(\cE)) \cong K^G_*(C(G,\cA))$, since the $*$-homo\-mor\-phism 
	\(
		C(G,\cA) \to C(\lVert \mathsf{V}_{\bullet} \rVert, q^*\cA) 
	\) 
	induces an isomorphism in equivariant $K$-theory by $G$-homotopy invariance. What remains to be done is to identify these $K$-theory groups. By Lemma~\ref{lem:GmodGI} the map $G/G_I \to V_I$ induces an isomorphism $K_q^G(C^*(\cE)(V_I)) \to K_q^G(C^*(\cE)(G/G_I))$. By the same lemma each $V_k$ is $G$-equivariantly contractible. Therefore $\left.\cA\right|_{G/G_I}$ is equivariantly trivialisable and 
	\[
		K_q^G(C^*(\cE)(G/G_I)) \cong K^G_q(C(G/G_I,\MF)) \cong K^{G_I}_q(\MF)\ .
	\]
	The matrix algebra $\MFa^{\otimes k}$ is $G$-equivariantly Morita equivalent via the imprimitivity bimodule $F(\C^n)^{\otimes k}$ to $\C$ with the trivial $G$-action. Therefore $K^{G_I}_0(\MFa^{\otimes k}) \cong K^{G_I}_0(\C) \cong R(G_I)$ and $K_1^{G_I}(\MFa^{\otimes k}) = 0$. The $*$-homomorphism $\MFa^{\otimes k} \to \MFa^{\otimes (k+1)}$ given by $a \mapsto a \otimes 1$ induces the multiplication with the $G_I$-representation $F(\C^n)$ on $K_0^{G_I}$. This implies
	\[
		K_{2q}^{G_I}(\MF) \cong R(G_I)\left[F(\rho_I)^{-1}\right] \qquad \text{and} \qquad K_{2q+1}^{G_I}(\MF) = 0
	\]
	for the $K$-theory of the direct limit.  
\end{proof}

\subsection{The module structure of $K_*^G(\cE)$}
An important consequence of Prop.~\ref{prop:spectral_seq} is that $K_*^G(C^*(\cE))$ is a module over the ring $K_0^G(\MF) \cong R(G)[F(\rho)^{-1}]$. To see this, we need the following observation about strongly self-absorbing $C^*$-dynamical systems. 

\begin{lemma} \label{lem:absorption}
	Let $G$ be a compact Lie group and let $\sigma \colon G \to U(n)$ be a unitary representation of $G$. Let $D = M_n(\C)^{\otimes \infty}$ be the infinite UHF-algebra obtained from $M_n(\C)$ and let $\alpha \colon G \to \Aut{D}$ be the action of $G$, which acts on each tensor factor of $D$ via $\Ad_{\sigma}$. Let $X$ be a compact Hausdorff $G$-space. Then the first tensor factor embedding
	\[
		\iota_X \colon C(X, D) \to C(X, D) \otimes D \quad , \quad f \mapsto f \otimes 1_D
	\]
	is strongly asymptotically G-unitarily equivalent to a $*$-isomorphism. 
\end{lemma}

\begin{proof}
By \cite[Prop.~6.3]{Szabo-strselfabsdyn-II:2018}, $(D,\alpha)$ is strongly self-absorbing in the sense of \cite[Def.~3.1]{Szabo-strselfabsdyn:2018}. Using the notation introduced in \cite[Def.~2.4]{Szabo-strselfabsdyn-II:2018} we see that
\[
	\left(D_{\infty,\alpha}\right)^{\alpha_{\infty}} = \left(D^{\alpha}\right)_{\infty}
\]
by integrating over $G$. The fixed-point algebra $D^{\alpha}$ is an AF-algebra and thus has a path-connected unitary group. By \cite[Prop.~2.19]{Szabo-strselfabsdyn-II:2018} the action $\alpha$ is unitarily regular. The result will therefore follow from \cite[Thm.~3.2]{Szabo-strselfabsdyn-III:2017} if we can construct a unital equivariant $*$-homomorphism 
\(
	\theta \colon D \to F_{\infty,\alpha}(C(X,D))
\). Let $s_k \colon D \to D$ be an approximately central sequence of unital equivariant $*$-homomorphisms and define
\[
	\theta(d)_k = 1_{C(X)} \otimes s_k(d)\ .
\]
This $*$-homomorphism satisfies all conditions. 
\end{proof}

There is a canonical isomorphism $K^G_0(\C) \cong R(G)$. As we have seen above, we also have $K^G_0(\MF) \cong R(G)[F(\rho)^{-1}]$, where the isomorphism can be chosen in such a way that the unit map $\C \to \MF$ induces the localisation homomorphism $R(G) \to R(G)[F(\rho)^{-1}]$. Note that the multiplication in $R(G)$ corresponds to the tensor product in $K^G_0(\C)$. Likewise, the identification $K^G_0(\MF) \cong R(G)[F(\rho)^{-1}]$ is also an isomorphism of rings. The  multiplication in $R(G)[F(\rho)^{-1}]$ corresponds to
\[
\begin{tikzcd}
	K_0^G(\MF) \otimes K_0^G(\MF) \ar[r] & K_0^G(\MF \otimes \MF) & \ar[l,"\cong"] K_0^G(\MF) 
\end{tikzcd}
\]
where the second map is induced by the first factor embedding. To see why this is true it suffices to note that the following diagram commutes 
\[
\begin{tikzcd}
	K_0^G(\MF) \otimes K_0^G(\MF) \ar[r] & K_0^G(\MF \otimes \MF) & \ar[l,"\cong" above] K_0^G(\MF) \\
	K_0^G(\C) \otimes K_0^G(\C) \ar[r] \ar[u] & K_0^G(\C) \ar[u] \ar[ur] 
\end{tikzcd}
\]
where the vertical homomorphism are induced by unit maps and turn into isomorphisms after localisation.

\begin{prop} \label{prop:module_structure}
	The first factor embedding $C^*(\cE) \to C^*(\cE) \otimes \MF$ given by $a \mapsto a \otimes 1_{\MF}$ induces an isomorphism in equivariant $K$-theory and turns $K_*^G(C^*(\cE))$ into a module over the ring $K_0^G(\MF)\cong R(G)[F(\rho)^{-1}]$ via
	\[
	\begin{tikzcd}
		K_*^G(C^*(\cE)) \otimes K_0^G(\MF) \ar[r] & K_*^G(C^*(\cE) \otimes \MF) & \ar[l,"\cong" midway] K_*^G(C^*(\cE))
	\end{tikzcd}
	\]
	where the first homomorphism is induced by the tensor product in $K$-theory. The sequence constructed in Prop.~\ref{prop:spectral_seq} is a spectral sequence of modules. 
\end{prop}

\begin{proof}
	Fix a non-empty subset $I \subset \{0,\dots, \ell\}$. Let $X_I = \iota_I(G/G_I) \subset G$, denote the barycentre of $\Delta_I$ by $\xi_I$ and note that $X_I = q^{-1}(\xi_I)$. 	We will first show that the embedding $C^*(\cE)(X_I) \to C^*(\cE)(X_I) \otimes \MF$ induces an isomorphism in equivariant $K$-theory. The group elements $g \in X_I$ share the same eigenvalues. Thus, there exists $z_0 \in \bT$ with the property that $(g,z_0) \in Y$ for one (and hence all) $g \in X_I$. Define $\sigma_I \colon X_I \to Y$ by $\sigma_I(g) = (g,z_0)$ and let $\sfX_I$ be the $G$-equivariant Morita equivalence resulting from  Lemma~\ref{lem:equiv_trivialisation} using the section $\sigma_I$. The claimed isomorphism is then a consequence of the following commutative diagram
\[
	\begin{tikzcd}
		K^G_*(C^*(\cE)(X_I)) \ar[r] \ar[d,"\cong"] & K^G_*(C^*(\cE)(X_I) \otimes \MF) \ar[d,"\cong"] \\
		K^G_*(C(X_I,\MF)) \ar[r,"\cong"] & K^G_*(C(X_I,\MF \otimes \MF))
	\end{tikzcd}
\]
in which the vertical maps are induced by $\sfX_I$ and the horizontal isomorphism follows from Lemma~\ref{lem:absorption}. 

The first factor embedding $C^*(\cE) \to C^*(\cE) \otimes \MF$ induces a natural transformation between the spectral sequences from Prop.~\ref{prop:spectral_seq} associated to the functors $(X,f) \mapsto K^G_*(C^*(f^*\cE))$ and $(X,f) \mapsto K^G_*(C^*(f^*\cE) \otimes \MF)$, which is an isomorphism on all pages by our previous observation. This implies that $K^G_*(C^*(\cE)) \to K^G_*(C^*(\cE) \otimes \MF)$ is also an isomorphism, which gives rise to the module structure as described. A diagram chase shows that this structure is compatible with the $K$-theoretic description of the multiplication in $K^G_0(\MF)$ described above.
\end{proof}

\begin{remark}
It would be interesting to know whether the first factor embedding $C^*(\cE) \to C^*(\cE) \otimes \MF$ itself is strongly asymptotically $G$-unitarily equivalent to a $*$-isomorphism. The analogous non-equivariant statement is true by \cite[Lemma~3.4]{DadarlatP-DixmierDouady:2016} and Cor.~\ref{cor:CStarBundle} .  
\end{remark}

\section{The equivariant higher twisted $K$-theory of $SU(n)$}
In this section we will compute the equivariant higher twisted $K$-theory of $G = SU(n)$ for $n \in \{2,3\}$ with respect to the adjoint action of $G$ on itself and the equivariant twist described by the Fell bundle $\cE$ constructed in Cor.~\ref{cor:Fell_bundle}. This is defined to be the $G$-equivariant operator algebraic $K$-theory of the $G$-$C^*$-algebra $C^*(\cE)$, i.e.
\(
	K^G_{\ast}(C^*(\cE))
\). 

\subsection{The case $SU(2)$}
For $G = SU(2)$ we have $\ell =1$. The map $q \colon G \to \Delta^1$ can be described as follows: Since the eigenvalues of any $g \in SU(2)$ are conjugate to one another, each $g \neq \pm 1$ has a unique eigenvalue $\lambda_g$ with non-negative imaginary part. Note that $g \mapsto \arg(\lambda_g)$ extends to all of $G$. Let $q \colon SU(2) \to [0,1]$ be given by
\[
	q(g) = \frac{\arg(\lambda_g)}{\pi} \in [0,1]\ .
\]
If we pick $\delta_n = \frac{2}{3}$ the spectral sequence from Prop.~\ref{prop:spectral_seq} boils down to the Mayer-Vietoris sequence for the $G$-equivariant closed cover of $SU(2)$ by $V_0 = q^{-1}([0,\tfrac{2}{3}])$ and $V_1 = q^{-1}([\tfrac{1}{3},1])$ in this case, which reduces to the following six-term exact sequence due to the vanishing $K_1^G$-terms: 
\[
	\begin{tikzcd}[column sep=0.8cm]
		 K_0^G(C^*(\cE)) \ar[r] & R_F(SU(2)) \oplus R_F(SU(2)) \ar[r,"d"] & R_{F}(\bT) \ar[d] \\
		 0 \ar[u] & \ar[l] 0 & \ar[l] K_1^G(C^*(\cE))
	\end{tikzcd}
\] 
where $R_F(\bT) = R(\bT)[F(\rho_{\{0,1\}})^{-1}]$ and $R_F(SU(2)) = R(SU(2))[F(\rho)^{-1}]$. Let $\rho$ be the standard representation of $SU(2)$, then 
\begin{align*}
	R_F(SU(2)) & \cong \Z[\rho][F(\rho)^{-1}] \ ,\\
	R_F(\bT) & \cong \Z[t,t^{-1}][F(t + t^{-1})^{-1}]
\end{align*}
and the restriction to $\bT$ maps $\rho$ to $t + t^{-1}$. Note that $t^{-1} \in \Z[t,t^{-1}]$ corresponds to the dual representation $\C^*$ of $\bT$.

Observe that $F(\C)$ is a representation of $\bT$, which we will identify with its corresponding polynomial in $\Z[t,t^{-1}]$ and denote by $F(t)$. Let $\alpha \in \Aut{\Z[t,t^{-1}]}$ be the ring automorphism given by $\alpha(t) = t^{-1}$. Note that 
\(
	\alpha(F(t)) = F(t^{-1})
\),
where the right hand side agrees with $F(\C^*)$ and that $R_F(SU(2))$ agrees with the fixed points of $R_F(\bT)$ with respect to $\alpha$.

\begin{lemma} \label{lem:basis}
	As a module over $R_F(SU(2))$ the ring $R_F(\bT)$ is free of rank~$2$ and $\beta = \{1, t\}$ is a basis. 
\end{lemma}

\begin{proof}
	Let $q = t + t^{-1}$. Since localisations of free modules are free, it suffices to show that $f \in \Z[t,t^{-1}]$ can be uniquely written as $f = g_1 + tg_2$ with $g_i \in \Z[q] \subset \Z[t,t^{-1}]$. Let $\alpha \in \Aut{\Z[t,t^{-1}]}$ be given by $\alpha(t) = t^{-1}$. Let
	\[
		g_1 = \frac{t^{-1}f - t\alpha(f)}{t^{-1} - t} \qquad, \qquad g_2 = \frac{\alpha(f)  - f}{t^{-1} - t} \ .
	\]
	Note that $\alpha(g_i) = g_i$ for $i \in \{1,2\}$ and $f = g_1 + tg_2$. Let $m \in \Z$ and consider $f = t^m$. In this case the numerator is divisible by the denominator and $g_i \in \Z[q]$. Using the linearity of the expressions in $f$ we see that $g_i \in \Z[q]$ holds in general. Suppose that $g_1 + tg_2 = g_1' + tg_2'$ for another pair $g_i' \in \Z[q]$. Applying $\alpha$ to both sides we obtain
	\[
		\begin{pmatrix}
			1 & t \\
			1 & t^{-1}
		\end{pmatrix}
		\begin{pmatrix}
			g_1 \\
			g_2
		\end{pmatrix} = 
		\begin{pmatrix}
			1 & t \\
			1 & t^{-1}
		\end{pmatrix}
		\begin{pmatrix}
			g_1' \\
			g_2'
		\end{pmatrix}	\ .
	\]
	Multiplication by the matrix $\left(\begin{smallmatrix}
		t^{-1} & -t \\
		-1 & 1
	\end{smallmatrix}\right)$ yields $(t^{-1} - t)g_i = (t^{-1} - t)g_i'$ and a comparison of coefficients gives $g_i = g_i'$.
\end{proof}

\begin{lemma} \label{lem:differential}
	If we identify $R_F(\bT)$ with $R_F(SU(2)) \oplus R_F(SU(2))$ using the basis $\beta = \{1,t\}$ from Lemma~\ref{lem:basis}, then the homomorphism 
	\[
		d \colon R_F(SU(2)) \oplus R_F(SU(2)) \to R_F(\bT)
	\]
	is represented by the matrix 
	\[
		\begin{pmatrix}
			1 & -g_1(F) \\
			0 & -g_2(F)
		\end{pmatrix}
	\]
	for polynomials $g_i(F) \in R_F(SU(2))$ given by
	\begin{equation} \label{eqn:coeff_F}
		g_1(F) = \frac{t^{-1}F(t) - tF(t^{-1})}{t^{-1} - t} \qquad \text{and} \qquad g_2(F) = \frac{F(t^{-1}) - F(t)}{t^{-1} - t}\ .
	\end{equation}
	(Note that $g_i(F)$ satisfies $\alpha(g_i(F)) = g_i(F)$ and therefore describes an element in $R_F(SU(2))$).
\end{lemma}

\begin{proof}
	We can identify $K^G_0(\MF) \cong R_F(SU(2))$ and $K^{\bT}_0(\MF) \cong R_F(\bT)$. Using these isomorphisms, the homomorphism $d$ fits into the commutative diagram
	\begin{equation} \label{eqn:diag_d}
	\begin{tikzcd}
		K^G_0(C^*(\cE)(V_0)) \oplus K^G_0(C^*(\cE)(V_1)) \ar[d,"\cong" left] \ar[r] & K^G_0(C^*(\cE)(V_0 \cap V_1)) \ar[d,"\cong"] \\
		K^G_0(\MF) \oplus K^G_0(\MF) \ar[r,"d" below] & K^{\bT}_0(\MF)
	\end{tikzcd}		
	\end{equation}
	where the upper horizontal map is induced by the inclusions $V_0 \cap V_1 \to V_i$ for $i \in \{0,1\}$. The two vertical isomorphisms are constructed as follows: Both $X_{i} \cong G/G_{i} = \ast$ are one-point spaces. By Lemma~\ref{lem:GmodGI} the inclusions $X_{i} \to V_i$ are $G$-equivariant homotopy equivalences. We will choose specific Morita equivalences between $C^*(\cE)(V_i)$ and $C(V_i, \MF)$, which give rise to the following isomorphisms
	\[
	\begin{tikzcd}
		K^G_0(C^*(\cE)(V_i)) \ar[r,"\cong"] & K^G_0(C(V_i,\MF)) \ar[r,"\cong"] & K^G_0(\MF)
	\end{tikzcd}
	\]
	where the first map is induced by the Morita equivalence and the second by the inclusion $X_i \to V_i$. Let $\omega_0 = -1$ and $\omega_1 = \exp(\tfrac{\pi i}{6})$. Consider the sections $\sigma_i \colon V_i \to Y$ given by $\sigma_i(g) = (g,\omega_i)$ for $i \in \{0,1\}$, which are well-defined by our choice of $\delta$ in the definition of $V_i$. By Lem.~\ref{lem:equiv_trivialisation} they induce equivariant Morita equivalences $\sfX_{V_i}$ between $C^*(\cE)(V_i)$ and $C(V_i,\MF)$. 
	
	For the vertical isomorphism on the right hand side we restrict the Morita equivalence induced by $\sfX_{V_0}$ to $V_0 \cap V_1$ and use
	\[
	\begin{tikzcd}
		K^G_0(C^*(\cE)(V_0 \cap V_1)) \ar[r,"\cong"] & K^G_0(C(V_0 \cap V_1,\MF)) \ar[r,"\cong"] & K^G_0(C(G/\bT, \MF))
	\end{tikzcd}
	\]
	with the first map induced by the equivalence and the second by 
	\[
		G/\bT \cong X_{\{0,1\}} \to V_0 \cap V_1 \qquad , \qquad [g] \mapsto g\begin{pmatrix}
			i & 0 \\
			0 & -i
		\end{pmatrix}g^{-1}\ .
	\]
	We obtain the following description of $d$: If $d(H_0,H_1) = d^{(0)}(H_0) + d^{(1)}(H_1)$, then $d^{(i)}$ fits into the following commutative diagram:	
	\[
	\begin{tikzcd}[column sep=2cm]
		K_0^G(C^*(\cE)(V_i)) \ar[r,"\text{res}"] \ar[d,"\cong" left,"\sfX_{V_i}" right] & K_0^G(C^*(\cE)(V_0 \cap V_1)) \ar[d,"\cong" left, "\left.\sfX_{V_0}\right|_{V_0 \cap V_1}" right] \\
		K_0^G(C(V_i,\MF)) \ar[d,"\cong" left] \ar[r,"\sfX_{V_i}^{\text{op}} \otimes \sfX_{V_0}"] & K_0^G(C(V_0 \cap V_1,\MF)) \ar[d,"\cong" left] \\
		K_0^G(\MF) \ar[r, "d^{(i)}"] & K_0^{G}(C(G/\bT,\MF)) 
	\end{tikzcd}
	\]
	where the tensor product on the middle horizontal arrow is taken over $C^*(\cE)(V_0 \cap V_1)$ and the bimodules have to be restricted to $V_0 \cap V_1$. In the case $i = 0$ the bimodule $\sfX_{V_0}^{\text{op}} \otimes \sfX_{V_0}$ is trivial. Let $H_0 \in R(G) \subset K^G_0(\MF)$. Using the isomorphism $K_0^{G}(C(G/\bT,\MF)) \cong K_0^{\bT}(\MF)$ the element $d^{(0)}(H_0)$ agrees with the restriction of $H_0$ to $\bT$. This gives the first column of the matrix in the statement.
	
	Let $I = \{0,1\}$ and define $\sigma_{I} \colon V_I \to \cG$ by $\sigma_{I}(g) = (g,\omega_1,\omega_0)$. Note that this is well-defined and we have the following isomorphism of bimodules:
	\[
		\sfX_{V_1}^{\text{op}} \otimes \sfX_{V_0} \cong C(V_{I}, \sigma_{I}^*\cE)\ .
	\]
	All elements of $X_I$ have eigenvalues $\pm i$. Therefore over $X_I$ this bimodule restricts to continuous sections of the bundle with fibre $E_g \otimes \MF$, where
	\[
		E_{g} = F(\Eig{g}{i})\ ,
	\]
	i.e.\ it corresponds to taking the tensor product with the vector bundle $E \to G/\bT$, which is isomorphic to $F(L)$, where $L \to G/\bT$ is the canonical line bundle associated to the principal $\bT$-bundle $G \to G/\bT$. The fibre of $F(L)$ over $[e] \in G/\bT$ is the representation $F(\C)$, where $\bT$ acts on $\C$ by its defining representation. The proof of Lemma~\ref{lem:basis} gives a decomposition of $F(t) \in R(\bT)$ with respect to the basis $\beta$ in terms of $g_1(F)$ and $g_2(F)$. Together with the sign convention in the exact sequence this explains the second column.
\end{proof}

\begin{theorem} \label{thm:twisted_eq_K_of_SU2}
	Let $F$ be an exponential functor with $F(\C) \ncong F(\C^*)$ as $\bT$-representations. The equivariant higher twisted $K$-theory of $G = SU(2)$ with twist described by the Fell bundle~$\cE$ constructed from $F$ is given by
	\begin{align*}
		K^G_0(C^*(\cE)) & = 0\ , \\
		K^G_1(C^*(\cE)) & = R_F(SU(2))/J_F
	\end{align*}
	as $R_F(SU(2))$-modules, where $J_F$ is generated by the representation $\sigma_F$ with character $\chi_F$ given by
	\[
		\chi_F = \frac{1}{\Delta}\det\begin{pmatrix}
			F(t) & F(t^{-1}) \\
			1 & 1
		\end{pmatrix}
	\] 
	where $\Delta=t - t^{-1}$. In particular, note that $K^G_1(C^*(\cE))$ is always a quotient ring of $R_F(SU(2))$.
\end{theorem}

\begin{proof}
	First note that $\chi_F$ coincides with the polynomial $g_2(F) \in R_F(SU(2))$ from \eqref{eqn:coeff_F}. The localisation of an integral domain at any multiplicative subset continues to be an integral domain. By hypothesis we have $g_2(F) \neq 0$. Thus, the matrix representation of $d$ from Lemma~\ref{lem:differential} implies that $d$ is injective, which proves $K^G_0(C^*(\cE)) = 0$. The group $K^G_1(C^*(\cE))$ is isomorphic to the cokernel of $d$ and the matrix representation of $d$ implies that it has the claimed form. 
\end{proof}

\subsubsection{Explicit computations for $SU(2)$} We will conclude with a section containing some explicit computations. The case of the classical twist over $SU(2)$ at level $k \in \N$ corresponds to the choice
\[
	F = \left(\extp^{\textrm{top}}\right)^{\otimes k}
\]
In this situation we have $\MF \cong \C$. This implies 
\[
	R_F(SU(2)) \cong R(SU(2))\ .
\] 
Together with  
\[
	\chi_F = \frac{t^k - t^{-k}}{t - t^{-1}} = \rho_{k-1}
\] 
we obtain $K^G_1(C^*(\cE)) = R(SU(2))/(\rho_{k-1})$, ie.\ the Verlinde ring of $LSU(2)$ as expected.

To see what happens in the case of higher twists, let $b_1,\dots,b_k \in \N$ and $W_j = \C^{b_j}$. Consider 
\[
	F_j(V) = \bigoplus_{m \in \N_0} W_j^{\otimes m} \otimes \extp^m(V)\ .
\]
and define $F = F_1 \otimes \dots \otimes F_k$. By \cite[Sec.~2.2]{Pennig:2018} each $F_j$ is an exponential functor and so is~$F$. The character of the irreducible representation $\rho_m$ of $SU(2)$ with highest weight $m$ in $R(\bT)$ is 
\[
	t^{m} + t^{m-2} + ... + t^{-m+2} + t^{-m} = \frac{t^{m+1} - t^{-(m+1)}}{t - t^{-1}}\ .
\]
With $F(t) = F_1(t)\cdots F_k(t)$ we compute $\chi_F$ with $\chi_{F_i} = 1 + b_i\,t$ as follows: 
\begin{align*}
	\chi_F & = \frac{\prod_{i=1}^k (1+b_i\,t) - \prod_{i=1}^k (1+b_i\,t^{-1})}{t - t^{-1}} 
	= \sum_{\ell=1}^k\left(\sum_{ \genfrac{}{}{0pt}{3}{I \subset \{1,\dots,k\}}{\lvert I \rvert = \ell}} b_I\right)\rho_{\ell-1} 
\end{align*}
where $b_I$ is the product over all $b_i$ with $i \in I$. In case $b_1 = \dots = b_k = 1$, i.e.\ for $F(V) = \extp^*(V)^{\otimes k}$ we obtain 
\[
	\chi_F = \sum_{\ell=1}^k \binom{k}{\ell} \rho_{\ell-1}\ .
\]
Using the fusion rules for $SU(2)$ the corresponding rings can be computed explicitly by expressing the ideals in terms of $\rho = \rho_1$. For example,
\begin{align*}
	k = 3:& \qquad \chi_F = 3 + 3\rho_1 + \rho_2 = \rho^2  + 3\rho  + 2 = (\rho + 2)(\rho+1)\ ,\\
	k = 4:& \qquad \chi_F = 4 + 6\rho_1 + 4\rho_2 + \rho_3 = \rho(\rho + 2)^2\ ,\\	
	k = 5:& \qquad \chi_F = 5 + 10\rho_1 + 10\rho_2 + 5\rho_3 + \rho_4 = (\rho+2)^2(\rho^2 + \rho -1)\ ,\\	
	k = 6:& \qquad \chi_F = 6 + 15\rho_1 + 20\rho_2 + 15\rho_3 + 6\rho_4 + \rho_5 = (\rho+2)^3(\rho^2-1)\ .	
\end{align*}
Note that the element $\rho+2 = \extp^*(\rho)$ is a unit in $R_F(SU(2))$. For odd tensor powers of the full exterior algebra twist, it is possible to compute all of these polynomials explicitly. Let $k = 2m+1$ and consider
\[
	p_m(F) = \sum_{\ell=1}^{2m+1} \binom{2m+1}{\ell} \rho_{\ell-1}\ .
\]
With $\rho_1 = \rho$ and $\rho_0 = 1$ tensor products of irreducible representations $\rho_i$ of $SU(2)$ decompose as follows
\[
	\rho \cdot \rho_i = \rho_{i-1} + \rho_{i+1}\ .
\]

\begin{lemma} 
	Let $0 \leq \ell \leq m$. In the ring $R(SU(2))$ we have 
	\[
		(\rho_m + \rho_{m-1})(\rho+2)^\ell = \sum_{k=-(\ell+1)}^\ell \binom{2\ell+1}{k+\ell+1} \rho_{m+k}
	\]
	(where we define $\rho_{-1} = 0$).
\end{lemma}
\begin{proof}
	The statement is true for $\ell = 0$. Now assume that it is true for $0, \dots, \ell$ with $\ell \leq m-1$ and note that
	\begin{align*}
		& (\rho_m + \rho_{m-1})(\rho + 2)^{\ell+1} = \left(\sum_{k=-(\ell+1)}^\ell \binom{2\ell+1}{k+\ell+1} \rho_{m+k}\right)(\rho + 2)\\
		=& \sum_{k=-(\ell+1)}^\ell \binom{2\ell+1}{k+\ell+1}(\rho_{m+k-1} + \rho_{m+k+1}) + \sum_{k=-(\ell+1)}^\ell 2\binom{2\ell+1}{k+\ell+1} \rho_{m+k} \\
		=& \sum_{k=-((\ell+1)+1)}^{\ell+1} \left[\binom{2\ell+1}{k+\ell} + 2\binom{2\ell+1}{k+\ell+1} + \binom{2\ell+1}{k+\ell+2}\right]\rho_{m+k}
	\end{align*}
	Now since $\binom{N+2}{K+2} = \binom{N}{K} + 2\binom{N}{K+1} +\binom{N}{K+2}$ by Pascal's triangle, the last sum agrees with 
	\[
		\sum_{k=-((\ell+1)+1)}^{\ell+1}\binom{2(\ell+1)+1}{k+(\ell+1)+1} \rho_{m+k}\ ,
	\]
	which proves the statement.
\end{proof}

For $m = \ell$ we immediately obtain the following corollary by shifting the summation index appropriately:

\begin{corollary}
	The polynomial $p(F)$ factors as follows:
	\[
		p_m(F) = (\rho_m + \rho_{m-1})(\rho + 2)^m\ .
	\]
\end{corollary}
To understand the quotient ring $R_F(SU(2))/(p_m(F))$ it is beneficial to choose a new set of generators $\nu_k$ with $\nu_0 = 1$ and
\[
	\nu_k = \rho_k + \rho_{k-1}\ ,
\]
which satisfy $\nu_k \cdot \nu_1 = \nu_{k+1} + \nu_k + \nu_{k-1}$. In the quotient ring the relation $\nu_m = 0$ holds, and $\nu_1 + 1$ is a unit due to the localisation. Surprisingly, the localisation turns out to be redundant as the following lemma shows:
\begin{lemma} \label{lem:fusion_ring_SU2}
	The inverse of $\nu_1 + 1$ in the ring $R = R_F(SU(2))/(p_m(F))$ is given by the element
	\[
		x = \sum_{i=0}^{m-1} (-1)^i(m-i)\,\nu_i = m + \sum_{i=1}^{m-1} (-1)^i(m-i)\,\nu_i\ .
	\]
	Since $x \in R(SU(2))$ we have $R(SU(2))/(\nu_m) \cong R$. 
\end{lemma}

\begin{proof}
	First note that
	\begin{align*}
		(\nu_1 + 1)x = \nu_1x + x = m\nu_1 + m + \sum_{i=1}^{m-1} (-1)^i (m-i) \left[\nu_{i-1} + 2\nu_i + \nu_{i+1}\right]\ .
	\end{align*}
	Shifting the summations appropriately we end up with
	\begin{align*}
		 & m\nu_1 + m - (m-1) + (m-2)\nu_1 + (-1)^{m-1} 2\nu_{m-1} - 2(m-1)\nu_1  \\
		- & (-1)^{m-1} 2\nu_{m-1} + \sum_{j = 2}^{m-2} (-1)^j \left( -(m-j-1) + 2(m-j) - (m-j+1) \right)\nu_j 
	\end{align*}
	which gives $1$ after cancelling all terms.
\end{proof}

We note that the case $k=5$ reproduces the Yang-Lee fusion rules as shown by our observations above.
\begin{corollary}
	The equivariant higher twisted $K$-theory of $SU(2)$ with twist given by the exponential functor $F=\left(\extp^*\right)^{\otimes 5}$ satisfies
	\[
		K^G_1(C^*(\cE)) \cong \Z \oplus \Z 
	\] 
	with basis $\{1,x\}$, where $x$ is the class of $-\rho$. It carries a ring structure given by the Yang-Lee fusion rules
	\(
		x^2 = x + 1
	\).
\end{corollary}

\begin{remark}\label{rem:fusion_rules}
	In general the fusion rules for the ring $R$ from Lemma~\ref{lem:fusion_ring_SU2} with respect to the generator $\nu_1$ are represented by the following tadpole diagram:
	\[
		\begin{tikzpicture}[	every loop/.style={}]  
			\draw[black,fill=black] (0,0) circle (2pt) node[below] {$\nu_0$};
			\draw[black,fill=black] (1,0) circle (2pt) node[below] {$\nu_1$};
			\draw[black,fill=black] (2,0) circle (2pt) node[below] {$\nu_2$};
			\draw[black,fill=black] (3,0) circle (2pt) node[below] {$\nu_3$};
			\draw[black,fill=black] (6,0) circle (2pt) node[below] {$\nu_{m-1}$};
			\draw[black] (0,0) -- (1,0) -- (2,0) -- (3,0) -- (4,0);
			\draw[black,dotted] (4,0) -- (5,0);
			\draw[black] (5,0) -- (6,0);
			
			\draw[black,scale=3,xshift=0.33cm] (0,0) to[out=40,in=140,loop] (0,0);
			\draw[black,scale=3,xshift=0.666cm] (0,0) to[out=40,in=140,loop] (0,0);
			\draw[black,scale=3,xshift=1cm] (0,0) to[out=40,in=140,loop] (0,0);
			\draw[black,scale=3,xshift=2cm] (0,0) to[out=40,in=140,loop] (0,0);
		\end{tikzpicture}
	\]
	This graph is reminiscent of the following: Let $\rho_0 = \id{}, \rho_2, \dots, \rho_{2(m-1)}$ be the objects in the even part of the fusion category for the loop group of $SU(2)$ at level $2m+1$. The fusion rules with respect to the generator $\rho_2$ also produce the above tadpole graph. 
\end{remark}

\subsection{The case $SU(3)$}
The group $G = SU(3)$ has rank $\ell = 2$. Let $F$ be an exponential functor and let $\rho$ be the standard representation of $G$ on $\C^3$. Consider the following localisations of representation rings:
\begin{align*}
	R_F(G_I) &= R(G_I)\left[F(\left.\rho\right|_{G_I})^{-1}\right] \ .
\end{align*}
For $\lvert I \rvert = 1$ we have $G_I = SU(3)$ and we will denote $G_{\{i\}}$ by $G_i$. In case $\lvert I \rvert = 2$ the group $G_I$ is isomorphic to $U(2)$ and the choice of $I$ determines an embedding $U(2) \subset SU(3)$. If $I = \{0,1,2\}$, then $G_I$ is the subgroup of all diagonal matrices, which is our choice of maximal torus $\bT^2 \subset SU(3)$. The $E_1$-page of the spectral sequence from Prop.~\ref{prop:spectral_seq} vanishes in odd rows and has the following chain complex in the even rows:
\begin{equation} \label{eqn:chain_cplx}
	\begin{tikzcd}[column sep=1.5cm]
		R_F(SU(3))^3 \ar[r,"d_0" above] & R_F(U(2))^3 \ar[r,"d_1" above] & R_F(\bT^2)
	\end{tikzcd}	
\end{equation}
The generators for the representation rings are chosen as follows: 
\begin{align*}
	R(SU(3)) &\cong \Z[s_1,s_2] \ ,\\
	R(U(2)) &\cong \Z[s,d,d^{-1}] \ ,\\
	R(\bT^2) &\cong \Z[t_1^{\pm 1},t_2^{\pm 1},t_3^{\pm 1}]/(t_1t_2t_3 - 1)\ ,	
\end{align*}
where $s_1 = \rho$ is the standard representation of $SU(3)$, $s_2 = \extp^2s_1$, $s$ denotes the standard representation of $U(2)$ on $\C^2$ and $d$ is its determinant representation. The characters $t_i$ are obtained by restricting the standard representation of $SU(3)$ to the maximal torus $\bT^2$ and projecting to the $i$th diagonal entry. Let $r \colon R_F(SU(3)) \to R_F(U(2))$ be the restriction homomorphism\footnote{The three inclusions $G_I \subset G$ for $\lvert I \rvert = 2$ induce the same map on representation rings.}. Then we have
\begin{align*}
	r(s_1) &= s + d^{-1} \ ,\\
	r(s_2) &= d^{-1}s + d\ .
\end{align*}
Let $\lambda_F = F(d^{-1})$ and $\mu_F = F(s)$. Note that $F(s + d^{-1})= F(s) \cdot F(d^{-1})$ is a unit in $R_F(U(2))$. Hence, the same is true for $\lambda_F,\mu_F \in R_F(U(2))$. To express the differential $d_0$ in terms of the $r$, $\lambda_F$ and $\mu_F$ we first need to give an explicit description of the map $q \colon SU(3) \to \Delta^2$: The eigenvalues of each $g \in SU(3)$ can be uniquely written in the form 
\[
	\exp(2\pi i \kappa_0),\quad \exp(2\pi i \kappa_1),\quad \exp(2\pi i \kappa_2)\ ,
\]
where $\kappa_0, \kappa_1, \kappa_2 \in \R$ satisfy $\sum_{j=0}^2 \kappa_j = 0$ and
\[
	\kappa_0 \geq \kappa_1 \geq \kappa_2 \geq \kappa_0 - 1\ .
\] 
Let $\mu_1 = \text{diag}\left(\tfrac{2}{3}, -\tfrac{1}{3}, -\tfrac{1}{3}\right) \in \liet$ and $\mu_2 = \text{diag}\left(\tfrac{1}{3}, \tfrac{1}{3}, -\tfrac{2}{3}\right) \in \liet$ be the (duals of the) fundamental weights. For any triple $\left(\kappa_0, \kappa_1, \kappa_2\right)$ as above, there are unique values $s,t \geq 0$ with $s + t \leq 1$ and 
\begin{equation} \label{eqn:kappa}
	\text{diag}\left(\kappa_0, \kappa_1, \kappa_2\right) = s\,\mu_1 + t\,\mu_2\ .
\end{equation}
The map $q \colon SU(3) \to \Delta^2$ sends $g \in SU(3)$ to the point $(1 - (s+t), s,t)$ in the simplex, i.e.\ if $\{e_0, e_1, e_2\}$ denotes the standard basis of $\R^3$, then $q(\exp(2\pi i \mu_j))$ agrees with $e_j$. We choose $\delta_n = \frac{17}{24}$ as the constant for the closed cover of $\Delta^2$ given by $A_0, A_1, A_2$. The result is shown in Fig.~\ref{fig:cover_of_Delta}. 
\begin{figure}[h]
\centering

\begin{tikzpicture}[scale=3.8]
	\coordinate (0) at (0,0);
	\coordinate (alpha2) at (0,{sqrt(2)});
	\coordinate (alpha1) at ({sqrt(2)*sin(120)}, {sqrt(2)*cos(120)});
	\coordinate (alpha3) at ($(alpha1)+(alpha2)$);
	
	\coordinate (mu0) at ($1/3*(alpha1) + 2/3*(alpha2)$);
	\coordinate (mu1) at ($2/3*(alpha1) + 1/3*(alpha2)$);
	
	\draw[blue!10,fill=blue!10] (0,0) -- ($(mu0)$) -- ($(mu1)$) -- cycle;
	\draw[blue,fill=blue!50,opacity=0.3] (0,0) -- ($17/24*(mu0)$) -- ($17/24*(mu1)$) -- cycle;
	\draw[blue,fill=red!50,opacity=0.3] ($(mu0)$) -- ($7/24*(mu0)$) -- ($17/24*(mu1) + 7/24*(mu0)$) -- cycle;
	\draw[blue,fill=green!50,opacity=0.3] ($(mu1)$) -- ($7/24*(mu1)$) -- ($17/24*(mu0) + 7/24*(mu1)$) -- cycle;

	\draw[blue] ($(mu0)$) -- ($(mu1)$);
	\draw[blue] (0,0) -- ($(mu0)$);
	\draw[blue] (0,0) -- ($(mu1)$);
	
	\draw[fill=blue, blue, thick] (0,0) circle (0.3pt) node[below,black] {$0$};
	\draw[fill=blue, blue, thick] ($(mu0)$) circle (0.3pt) node[above,black] {$2$};
	\draw[fill=blue, blue, thick] ($(mu1)$) circle (0.3pt) node[below,black] {$1$};

	\node[blue] at ($3/20*(mu0) + 3/20*(mu1)$) {$A_0$};
	\node[red] at ($14/20*(mu0) + 3/20*(mu1)$) {$A_2$};
	\node[green!70!blue!100] at ($3/20*(mu0) + 14/20*(mu1)$) {$A_1$};
	
\end{tikzpicture}

\caption{\label{fig:cover_of_Delta}The three closed sets $A_i$ covering $\Delta^2$.}	
\end{figure}

To express the differential $d_1$ in terms of the representation rings, we first observe that we have three inclusions $\iota_{I} \colon G_{\{0,1,2\}} \to G_I$ for $I \subset \{0,1,2\}$ with $\lvert I \rvert = 2$. These induce three restriction maps 
\[
	r_I \colon R_F(G_I) \to R_F(G_{\{0,1,2\}}) \cong R_F(\bT^2)\ . 
\]
Let $\nu_F = F(t_1)$ for $t_1 \in R(\bT^2) \subset R_F(\bT^2)$ as defined above. In the next lemma we write $r_{ij}$ for $r_I$ with $I = \{i,j\}$.

\begin{lemma} \label{lem:diff_d0_d1}
The trivialisations $R_F(G_I) \cong K^G_0(C^*(\cE)(X_I))$ in the spectral sequence can be chosen in such a way that the differential $d_0$ is given by the following expression
\begin{align*}
	 & d_0(x_0,x_1,x_2) \\
	=& (-r(x_0) + \lambda_F \cdot r(x_1), -r(x_1) + \mu_F^{-1} \cdot r(x_2), -r(x_0) + \lambda_F^{-1}\cdot r(x_2))	
\end{align*}
where $x_i \in R_F(G_i) = R_F(SU(3))$ and the three components on the right hand side correspond to the subsets $I = \{0,1\}$, $\{1,2\}$ and $\{0,2\}$ respectively. Moreover, $d_1$ takes the following form
\[
	d_1(y_{01},y_{12},y_{02}) = r_{01}(y_{01}) + \nu_F \cdot r_{12}(y_{12}) - r_{02}(y_{02}) \ ,
\]
where $y_{ij} \in R_F(G_{\{i,j\}})$.
\end{lemma}

\begin{proof}
As above we write $X_i$ for $X_{\{i\}}$ and similarly for $G_i$ and $V_i$. Observe that $G_{i} = G$ implies that $X_{i}$ is a one-point space for $i \in \{0,1,2\}$. We will first discuss the construction of the differential $d_0$. Restriction along the $G$-equivariant homotopy equivalence $X_{i} \to V_{i}$ induces an isomorphism   
\[
	\begin{tikzcd}
		K_0^G(C^*(\cE)(V_{i})) \ar[r,"\cong" above] & K_0^G(C^*(\cE)(X_{i}))\ .
	\end{tikzcd}
\]
The differential $d_0$ is an alternating sum of restriction homomorphisms along the inclusions of the form $V_{\{i,j\}} \to V_{k}$ with $k \in \{i,j\}$ composed with isomorphisms as shown in the following diagram:
\[
	\begin{tikzcd}
		R_F(G_k) \ar[r,"\cong"] & K_0^G(C^*(\cE)(X_{k})) & \ar[l,"\cong" above] K_0^G(C^*(\cE)(V_{k})) \ar[d] \\ 
		R_F(G_{\{i,j\}}) \ar[r,"\cong"] & K_0^G(C^*(\cE)(X_{\{i,j\}})) & \ar[l,"\cong" above] K_0^G(C^*(\cE)(V_{\{i,j\}}))
	\end{tikzcd}
\]
We will fix the isomorphisms on the left hand side by choosing an equivariant trivialisation of $C^*(\cE)(V_k)$ via Morita equivalences given by Lemma~\ref{lem:equiv_trivialisation}. Let
\begin{align*}
	\omega_0 = -1 \ ,\qquad
	\omega_1 = \exp\left(2\pi i \tfrac{1}{6}\right)\ ,\qquad
	\omega_2 = \exp\left(2\pi i \tfrac{5}{6}\right)
\end{align*}
and define $\sigma_k \colon V_k \to Y$ by $\sigma_k(g) = (g,\omega_k)$. We claim that this is well-defined and will show this for $k = 0$. The other cases follow similarly. To see that $\omega_0 \notin \EV{g}$ for all $g \in V_0$ it suffices to prove that all coordinates $\kappa_i$ of $q(g)$ are different from $\pm \frac{1}{2}$ for all $g \in V_0$. By \eqref{eqn:kappa} we have
\[
	\kappa_0 = \frac{2}{3}s + \frac{1}{3}t = \frac{1}{2} \qquad \Leftrightarrow \qquad s = \frac{3 - 2t}{4}
\]
and the condition $s+t \leq 1$ implies that $0 \leq t \leq \frac{1}{2}$. Likewise, 
\[
	\kappa_2 = -\frac{1}{3}s' - \frac{2}{3}t' = -\frac{1}{2} \qquad \Leftrightarrow \qquad s' = \frac{3 - 4t'}{2}
\]
and $s'+t' \leq 1$, $s' \geq 0$ yield the constraints $\frac{1}{2} \leq t' \leq \frac{3}{4}$. Note that the coordinate $\kappa_1$ is never equal to $\pm \frac{1}{2}$, since this would contradict the constraints imposed by $\kappa_0 \geq \kappa_1 \geq \kappa_2 \geq \kappa_0 -1$ and $\kappa_0 + \kappa_1 + \kappa_2 = 0$. Therefore the matrices with at least one eigenvalue equal to $-1$ are parametrised by the two line segments described above and shown in Fig.~\ref{fig:triv_over_V0}. Our choice for $\delta_n$ was made in such a way that the resulting $A_0$ avoids this set proving our claim for $V_0$. The situation will look similar for $V_1$ and $V_2$ insofar as Fig.~\ref{fig:triv_over_V0} just has to be rotated accordingly.
\begin{figure}[h]
\centering

\begin{tikzpicture}[scale=3.2]
	\coordinate (0) at (0,0);
	\coordinate (alpha2) at (0,{sqrt(2)});
	\coordinate (alpha1) at ({sqrt(2)*sin(120)}, {sqrt(2)*cos(120)});
	\coordinate (alpha3) at ($(alpha1)+(alpha2)$);
	
	\coordinate (mu0) at ($1/3*(alpha1) + 2/3*(alpha2)$);
	\coordinate (mu1) at ($2/3*(alpha1) + 1/3*(alpha2)$);
	
	\draw[blue!10,fill=blue!10] (0,0) -- ($(mu0)$) -- ($(mu1)$) -- cycle;
	\draw[blue,fill=blue!40,opacity=0.3] (0,0) -- ($17/24*(mu0)$) -- ($17/24*(mu1)$) -- cycle;

	\draw[blue] ($(mu0)$) -- ($(mu1)$);
	\draw[blue] (0,0) -- ($(mu0)$);
	\draw[blue] (0,0) -- ($(mu1)$);
	
	\draw[fill=blue, blue, thick] (0,0) circle (0.3pt) node[below,black] {$0$};
	\draw[fill=blue, blue, thick] ($(mu0)$) circle (0.3pt) node[above,black] {$2$};
	\draw[fill=blue, blue, thick] ($(mu1)$) circle (0.3pt) node[below,black] {$1$};
	
	\draw[red, very thick] ($3/4*(mu1)$) -- ($1/2*(mu0) + 1/2*(mu1)$);
	\draw[red, very thick] ($3/4*(mu0)$) -- ($1/2*(mu0) + 1/2*(mu1)$);
	
	\node[blue] at ($1/5*(mu0) + 1/5*(mu1)$) {$A_0$};
\end{tikzpicture}

\caption{\label{fig:triv_over_V0}The red lines correspond to $SU(3)$ elements with at least one eigenvalue equal to~$-1$. As can be seen from this picture the set $A_0$ avoids those two lines.}	
\end{figure}

By Lem.~\ref{lem:equiv_trivialisation} the section $\sigma_k$ constructed above gives an equivariant Morita equivalence $\sfX_{V_k}$ between $C^*(\cE)(V_k)$ and $C(V_k, \MF)$. Let $I \subset \{0,1,2\}$ and denote the minimal element of $I$ by $i_0$. The restriction of $\sfX_{V_{i_0}}$ to $V_I$ is a Morita equivalence between $C^*(\cE)(V_I)$ and $C(V_I,\MF)$ and the trivialisation $R_F(G_I) \cong K_0^G(C^*(\cE)(X_I))$ is induced by the restricting further to $X_I \subset V_I$. The differential $d_0$ is a signed sum of components of the form
\[
	d_k^I \colon K^G_0(\MF) \to K_0^G(C(G/G_I,\MF)) 
\] 
with $k \in I$. Just as in Lem.~\ref{lem:differential} the $d_k^I$ fits into the following commutative diagram:
	\[
	\begin{tikzcd}[column sep=2cm]
		K_0^G(C^*(\cE)(V_k)) \ar[r,"\text{res}"] \ar[d,"\cong" left,"\sfX_{V_k}" right] & K_0^G(C^*(\cE)(V_I)) \ar[d,"\cong" left, "\sfX_{V_{i_0}}" right] \\
		K_0^G(C(V_k,\MF)) \ar[d,"\cong" left] \ar[r,"\sfX_{V_k}^{\text{op}} \otimes \sfX_{V_{i_0}}"] & K_0^G(C(V_I,\MF)) \ar[d,"\cong" left] \\
		K_0^G(\MF) \ar[r, "d_k^{I}"] & K_0^{G}(C(G/G_I,\MF)) 
	\end{tikzcd}
	\]
From this we see that if $I = \{i,j\}$ with $i < j$ and $k = i$, then after identifying the domain of $d_k^I$ with $R_F(G)$ and the codomain with $R_F(G_I)$ the map agrees with the restriction homomorphism. Let $I = \{i,j\}$ with $i < j$ and define $\sigma_I \colon V_{I} \to Y$ by $\sigma_I(g) = (g,\omega_j, \omega_i)$. In this situation we have
\[
	\sfX_{V_j}^{\text{op}} \otimes \sfX_{V_{i}} \cong C(V_I, \sigma_I^*\cE)\ .
\] 
Let $E \to V_I$ be the vector bundle with fibre over $g \in V_I$ given by
\[
	E_g = F(\Eig{g}{\lambda})
\]
where $\lambda \in \EV{g}$ is the eigenvalue of $g$ between $\omega_j$ and $\omega_i$. Then $\sigma_I^*\cE \to V_I$ is either of the form $E \otimes \MF$ if $\omega_j < \omega_i$ or $(E \otimes \MF)^{\text{op}}$ if $\omega_i < \omega_j$. Note that $\left.E\right|_{X_I} \cong F(Q)$, where $Q \to X_I$ is the vector bundle associated to the principal $G_I$-bundle $G \to G/G_I$ either using the inverse determinant representation~$d^{-1}$ or the standard representation $s$ depending on whether $\dim(\Eig{g}{\lambda}) = 1$ or $2$ respectively. Using the identifications $K_0^G(\MF) \cong R_F(G)$ and $K_0^G(C(X_I,\MF)) \cong R_F(G_I)$, the map $d_j^I$ therefore corresponds to a factor of the form $F(d^{-1})^{\pm 1}$ or $F(s)^{\pm 1}$ times the restriction homomorphism. The resulting factors are listed in Fig.~\ref{tab:factor_table}. 
\begin{figure}[ht]
\centering
\renewcommand{\arraystretch}{1.5}
\begin{tabular}{|c|c|c|c|c|}
\hline
I & order & eigenvalues & representation & factor \\
\hline
$\{0,1\}$ & $\omega_1 < \omega_0$ & $\color{red} e(\frac{2}{6})$, $e(-\frac{1}{6})$, $e(-\frac{1}{6})$ & $d^{-1}$ & $F(d^{-1}) = \lambda_F$ \\
$\{1,2\}$ & $\omega_1 < \omega_2$ & $\color{red} e(\frac{1}{2})$, $1$, $\color{red} e(-\frac{1}{2})$ & $s$ & $F(s)^{-1} = \mu_F^{-1}$ \\
$\{0,2\}$ & $\omega_0 < \omega_2$ & $e(\frac{1}{6})$, $e(\frac{1}{6})$, $\color{red} e(-\frac{2}{6})$ & $d^{-1}$ & $F(d^{-1})^{-1} = \lambda_F^{-1}$ \\
\hline
\end{tabular}
\caption{\label{tab:factor_table}The table shows the eigenvalue $\lambda$ of $g \in X_I$ between $\omega_i$ and $\omega_j$ in red with $e(\varphi) = e^{2\pi i \varphi}$ and the resulting factor in the right hand column.}
\end{figure}
Together with the sign convention for the exact sequence this explains the form of $d_0$. 

The same reasoning can be used for $d_1$. Let $I \subset \{0,1,2\}$ be a subset with $\lvert I \rvert = 2$ and let $J = \{0,1,2\}$. The differential $d_1$ decomposes into a sum
\[
	d_1(x_{01}, x_{12}, x_{02}) = d^{\{0,1\}}(x_{01}) + d^{\{1,2\}}(x_{12}) - d^{\{0,2\}}(x_{02})
\] 
with three maps $d^I \colon K_0^G(C(X_I,\MF)) \to K_0^G(C(X_J,\MF))$ that fit into the following commutative diagram:
	\[
	\begin{tikzcd}[column sep=2cm]
		K_0^G(C(V_I,\MF)) \ar[d,"\cong" left] \ar[r,"\sfX_{V_I}^{\text{op}} \otimes \sfX_{V_{J}}"] & K_0^G(C(V_J,\MF)) \ar[d,"\cong" left] \\
		K_0^G(C(X_I,\MF)) \ar[r, "d^I"] & K_0^{G}(C(X_J,\MF)) 
	\end{tikzcd}
	\]
Just as above we see that $d^I$ agrees with the restriction homomorphism $r_I$ in the cases $I = \{0,1\}$ and $I = \{0,2\}$, since $\sfX_{V_I}^{\text{op}} \otimes \sfX_{V_{J}}$ is trivial then. The only remaining case is $I = \{1,2\}$, where we have   
\[
	\sfX_{V_I}^{\text{op}} \otimes \sfX_{V_{J}} \cong \sfX_{V_1}^{\text{op}} \otimes \sfX_{V_{0}} \cong C(V_J, \sigma_{\{0,1\}}^*\cE)
\]
The eigenvalues for all $g \in X_J$ are $e(\frac{1}{3})$, $1$, $e(-\frac{1}{3})$ and only the first one lies between $\omega_1$ and $\omega_0$. Let $\lambda = e(\frac{1}{3})$ and let $E \to X_J$ be the vector bundle with fibres given by 
\[
	E_g = F(\Eig{g}{\lambda})\ .
\] 
It is isomorphic to $F(Q)$, where $Q \to X_J$ is the vector bundle associated to the principal $G_J$-bundle $G \to G/G_J$ via the representation $t_1$. Thus, by the same argument as before the map $d^{\{1,2\}}$ agrees with $F(t_1)$ times the restriction homomorphism $r_{12}$. 
\end{proof}

\subsubsection{Restriction to maximal torus} As above, denote by $\liet$ the Lie algebra of the maximal torus $\bT^2 \subset SU(3)$. In this section we will prove that the chain complex in Lemma~\ref{lem:diff_d0_d1} computes the equivariant (Bredon) cohomology of $\liet$ with respect to an extended Weyl group action and a twisted coefficient system. This approach is reminiscent of the method used in \cite{AdemCantareroGomez-TwistedK:2018} to compute the (rational) twisted equivariant $K$-theory of actions with isotropy of maximal rank and classical twist. We will focus here on the action of $G = SU(3)$ on itself by conjugation with a non-classical twists. An extension of this approach to $G = SU(n)$ will be part of upcoming work.

Let $W = S_3$ be the Weyl group of $SU(3)$. Our identification of $\bT^2$ with the diagonal matrices induces a corresponding isomorphism  
\[
	\liet \cong \{ (h_1,h_2,h_3) \in \R^3 \ |\ h_1 + h_2 + h_3 = 0 \} \ .
\] 
The fundamental group $\pi_1(\bT,e)$ agrees with the lattice $\Lambda$ in $\liet$ obtained as the kernel of the exponential map. We will identify the two, which gives
\begin{equation} \label{eqn:lattice}
	\pi_1(\bT^2,e) = \Lambda = \{ (k_1,k_2,k_3) \in \Z^3 \ |\ k_1 + k_2 + k_3 = 0\}
\end{equation}
The Weyl group acts on $\liet$ and $\Lambda$ by permuting the coordinates and we define
\[
	\affW = \pi_1(\bT^2) \rtimes W\ .
\]
Note that $W$ also acts on $\Z^3$ in the same way. Let 
\(
	\extW = \Z^3 \rtimes W 
\)
and observe that $\affW \subset \extW$ as a normal subgroup. Given an exponential functor $F$ we obtain a group homomorphism
\[
	\varphi \colon \pi_1(\bT^2,e) \to GL_1(R_F(\bT^2)) \  , \  \varphi(k_1, k_2, k_3) = F(t_1)^{k_1} \cdot F(t_2)^{k_2} \cdot F(t_3)^{k_3}\ .
\]
If we define $F(-t_i) = F(t_i)^{-1}$ we can rewrite the right hand side as 
\[
	\varphi(k_1,k_2,k_3) = F(k_1t_1 + k_2t_2 + k_3t_3) \ .
\]
Combining $\varphi$ with the permutation action of $W$ on $R_F(\bT^2)$ results in an action of $\affW$ on $R_F(\bT^2)$. Just as in \cite{AdemCantareroGomez-TwistedK:2018} this gives rise to local coefficient systems $\cR$ and $\cRQ$ as follows
\[
	\cR(\affW/H) = R_F(\bT^2)^H \qquad , \qquad \cRQ(\affW/H) = R_F(\bT^2)^H \otimes \mathbb{Q}\ .
\] 
The simplex $\Delta^2 \subset \liet$ is a fundamental domain for the action of $\affW$ on $\liet$ and turns it into a $\affW$-CW-complex, in which the $k$-cells are labelled by subsets $I \subset \{0,1,2\}$ with $\lvert I \rvert = k+1$. We have three $0$-cells, three $1$-cells and one $2$-cell. Let $\affW_I$ be the stabiliser of $\xi_I$. Likewise let $W_I \subset W$ be the stabiliser of $\exp(2\pi i \xi_I)$. The restriction maps $R_F(G_I) \to R_F(\bT^2)$ induce ring isomorphisms 
\[
	r_I \colon R_F(G_I) \to R_F(\bT^2)^{W_I}\ .
\] 
As above let $q \colon G \to \Delta^2$ be the quotient map that parametrises conjugacy classes. Let $\hat{q} = \left.q\right|_{\bT^2} \circ q_{\liet}$, where $q_{\liet} \colon \liet \to \bT^2$ is the universal covering. Let $B_I = \hat{q}^{-1}(A_I) \subset \liet$. Note that $\{B_0, B_1, B_2\}$ is a $\affW$-equivariant cover of $\liet$ as shown in Fig.~\ref{fig:cover_of_t}. It has the property that the inclusion map $\affW \cdot \xi_I \to B_I$ is an equivariant homotopy equivalence.  

\begin{figure}[h]
\centering

\begin{tikzpicture}[scale=1.5]
	\coordinate (0) at (0,0);
	\coordinate (alpha2) at (0,{sqrt(2)});
	\coordinate (alpha1) at ({sqrt(2)*sin(120)}, {sqrt(2)*cos(120)});
	\coordinate (alpha3) at ($(alpha1)+(alpha2)$);
	
	\coordinate (mu0) at ($1/3*(alpha1) + 2/3*(alpha2)$);
	\coordinate (mu1) at ($2/3*(alpha1) + 1/3*(alpha2)$);

	\clip (-1.6,-1.4) rectangle + (3.2,2.8);
	
	\foreach \k in {-1,0,1}
	{
	\coordinate (origin) at ($+\k*(mu0) + \k*(mu1)$);
	\coordinate (mu0) at ($1/3*(alpha1) + 2/3*(alpha2)$);
	\coordinate (mu1) at ($2/3*(alpha1) + 1/3*(alpha2)$);
	
	\foreach \i in {0,60,...,300}
	{
	\draw[blue!10,fill=blue!10] let \p0 = ($(origin)$), \p1 = ($(mu0)+(origin)$), \p2 = ($(mu1)+(origin)$) in 
		[rotate around={\i:(origin)}] (\p0) -- (\p1) -- (\p2) -- cycle;
	\draw[blue,fill=blue!50,opacity=0.3] let \p0 = ($(origin)$), \p1 = ($17/24*(mu0)+(origin)$), \p2 = ($17/24*(mu1)+(origin)$) in
		[rotate around={\i:(origin)}] (\p0) -- (\p1) -- (\p2) -- cycle;
	}

	\foreach \i in {0,120,240}
	{
		\draw[blue,fill=red!50,opacity=0.3] let \p1 = ($(mu0)+(origin)$), \p2 = ($7/24*(mu0)+(origin)$), \p3 = ($17/24*(mu1) + 7/24*(mu0)+(origin)$) in [rotate around={\i:(origin)}] (\p1) -- (\p2) -- (\p3) -- cycle;		
		\draw[blue,fill=green!50,opacity=0.3] let \p1 = ($(mu1)+(origin)$), \p2 = ($7/24*(mu1)+(origin)$), \p3 = ($17/24*(mu0) + 7/24*(mu1)+(origin)$) in [rotate around={\i:(origin)}] (\p1) -- (\p2) -- (\p3) -- cycle; 
	}	

	\foreach \i in {60,180,300}
	{
		\draw[blue,fill=green!50,opacity=0.3] let \p1 = ($(mu0)+(origin)$), \p2 = ($7/24*(mu0)+(origin)$), \p3 = ($17/24*(mu1) + 7/24*(mu0)+(origin)$) in [rotate around={\i:(origin)}] (\p1) -- (\p2) -- (\p3) -- cycle;		
		\draw[blue,fill=red!50,opacity=0.3] let \p1 = ($(mu1)+(origin)$), \p2 = ($7/24*(mu1)+(origin)$), \p3 = ($17/24*(mu0) + 7/24*(mu1)+(origin)$) in [rotate around={\i:(origin)}] (\p1) -- (\p2) -- (\p3) -- cycle; 
	}
	}

	\foreach \k in {-1,1}
	{
	\coordinate (origin) at ($-\k*(mu0) + 2*\k*(mu1)$);
	\coordinate (mu0) at ($1/3*(alpha1) + 2/3*(alpha2)$);
	\coordinate (mu1) at ($2/3*(alpha1) + 1/3*(alpha2)$);
	
	\foreach \i in {0,60,...,300}
	{
	\draw[blue!10,fill=blue!10] let \p0 = ($(origin)$), \p1 = ($(mu0)+(origin)$), \p2 = ($(mu1)+(origin)$) in 
		[rotate around={\i:(origin)}] (\p0) -- (\p1) -- (\p2) -- cycle;
	\draw[blue,fill=blue!50,opacity=0.3] let \p0 = ($(origin)$), \p1 = ($17/24*(mu0)+(origin)$), \p2 = ($17/24*(mu1)+(origin)$) in
		[rotate around={\i:(origin)}] (\p0) -- (\p1) -- (\p2) -- cycle;
	}

	\foreach \i in {0,120,240}
	{
		\draw[blue,fill=red!50,opacity=0.3] let \p1 = ($(mu0)+(origin)$), \p2 = ($7/24*(mu0)+(origin)$), \p3 = ($17/24*(mu1) + 7/24*(mu0)+(origin)$) in [rotate around={\i:(origin)}] (\p1) -- (\p2) -- (\p3) -- cycle;		
		\draw[blue,fill=green!50,opacity=0.3] let \p1 = ($(mu1)+(origin)$), \p2 = ($7/24*(mu1)+(origin)$), \p3 = ($17/24*(mu0) + 7/24*(mu1)+(origin)$) in [rotate around={\i:(origin)}] (\p1) -- (\p2) -- (\p3) -- cycle; 
	}	

	\foreach \i in {60,180,300}
	{
		\draw[blue,fill=green!50,opacity=0.3] let \p1 = ($(mu0)+(origin)$), \p2 = ($7/24*(mu0)+(origin)$), \p3 = ($17/24*(mu1) + 7/24*(mu0)+(origin)$) in [rotate around={\i:(origin)}] (\p1) -- (\p2) -- (\p3) -- cycle;		
		\draw[blue,fill=red!50,opacity=0.3] let \p1 = ($(mu1)+(origin)$), \p2 = ($7/24*(mu1)+(origin)$), \p3 = ($17/24*(mu0) + 7/24*(mu1)+(origin)$) in [rotate around={\i:(origin)}] (\p1) -- (\p2) -- (\p3) -- cycle; 
	}
	}

	\foreach \k in {-1,1}
	{
	\coordinate (origin) at ($-2*\k*(mu0) + \k*(mu1)$);
	\coordinate (mu0) at ($1/3*(alpha1) + 2/3*(alpha2)$);
	\coordinate (mu1) at ($2/3*(alpha1) + 1/3*(alpha2)$);
	
	\foreach \i in {0,60,...,300}
	{
	\draw[blue!10,fill=blue!10] let \p0 = ($(origin)$), \p1 = ($(mu0)+(origin)$), \p2 = ($(mu1)+(origin)$) in 
		[rotate around={\i:(origin)}] (\p0) -- (\p1) -- (\p2) -- cycle;
	\draw[blue,fill=blue!50,opacity=0.3] let \p0 = ($(origin)$), \p1 = ($17/24*(mu0)+(origin)$), \p2 = ($17/24*(mu1)+(origin)$) in
		[rotate around={\i:(origin)}] (\p0) -- (\p1) -- (\p2) -- cycle;
	}

	\foreach \i in {0,120,240}
	{
		\draw[blue,fill=red!50,opacity=0.3] let \p1 = ($(mu0)+(origin)$), \p2 = ($7/24*(mu0)+(origin)$), \p3 = ($17/24*(mu1) + 7/24*(mu0)+(origin)$) in [rotate around={\i:(origin)}] (\p1) -- (\p2) -- (\p3) -- cycle;		
		\draw[blue,fill=green!50,opacity=0.3] let \p1 = ($(mu1)+(origin)$), \p2 = ($7/24*(mu1)+(origin)$), \p3 = ($17/24*(mu0) + 7/24*(mu1)+(origin)$) in [rotate around={\i:(origin)}] (\p1) -- (\p2) -- (\p3) -- cycle; 
	}	

	\foreach \i in {60,180,300}
	{
		\draw[blue,fill=green!50,opacity=0.3] let \p1 = ($(mu0)+(origin)$), \p2 = ($7/24*(mu0)+(origin)$), \p3 = ($17/24*(mu1) + 7/24*(mu0)+(origin)$) in [rotate around={\i:(origin)}] (\p1) -- (\p2) -- (\p3) -- cycle;		
		\draw[blue,fill=red!50,opacity=0.3] let \p1 = ($(mu1)+(origin)$), \p2 = ($7/24*(mu1)+(origin)$), \p3 = ($17/24*(mu0) + 7/24*(mu1)+(origin)$) in [rotate around={\i:(origin)}] (\p1) -- (\p2) -- (\p3) -- cycle; 
	}
	}

\end{tikzpicture}

\caption{\label{fig:cover_of_t}The cover of $\liet$ induced by the cover of $\Delta^2$.}	
\end{figure}

For any subset $I \subset \{0,1,2\}$ the Bredon cohomology $H^k_{\affW}(B_I, \cR)$ is only non-zero in degree $k = 0$ where we have a natural isomorphism 
\[
	H^0_{\affW}(B_I; \cR) \cong R_F(\bT^2)^{\affW_I}\ .
\] 
For $J \subset I$ the restriction homomorphism $H^0_{\affW}(B_J; \cR) \to H^0_{\affW}(B_I, \cR)$ translates into the natural inclusion $R_F(\bT^2)^{W_J} \subset R_F(\bT^2)^{W_I}$. The sum over all $H^q_{\affW}(B_I; \cR)$ with $\lvert I \rvert = p+1$ forms the $E^1$-page of a spectral sequence that converges to $H^{p+q}_{\affW}(\liet; \cR)$. By our above considerations this $E^1$-page boils down to the chain complex
\[
	C^k_{\affW}(\liet; \cR) = \bigoplus_{\lvert I \rvert = k+1} R_F(\bT^2)^{\affW_I} 
\]
with the differentials $d_k^{\mathrm{cell}} \colon C^k_{\affW}(\liet; \cR) \to C^{k+1}_{\affW}(\liet; \cR)$ given by alternating sums of inclusion homomorphisms. We can identify $W$ with the subgroup of $\affW$ consisting of elements of the form $(0,w) \in \pi_1(\bT^2,e) \rtimes W$. Observe that $W_i = W = \affW_0$ for $i \in \{0,1,2\}$, $W_{\{0,2\}} = \affW_{\{0,2\}}$, $W_{\{0,1\}} = \affW_{\{0,1\}}$ and we have group isomorphisms
\begin{align*}
	\affW_1 \to W_1 \qquad &,& \qquad x \mapsto ((-1,0,0),e_W) \cdot x \cdot ((1,0,0),e_W) \ ,\\	
	\affW_2 \to W_2 \qquad &,& \qquad x \mapsto ((0,0,1),e_W) \cdot x \cdot ((0,0,-1),e_W) \ .
\end{align*}
Here, $e_W$ denotes the neutral element of $W$ and we used the conjugation action of $\extW$ on $W$. The first isomorphism restricts to $\affW_{\{1,2\}} \to W_{\{1,2\}}$. These identifications induce corresponding isomorphisms of the fixed point rings 
\begin{align*}
	\widetilde{r}_1 &\colon R_F(\bT^2)^{W_1} \to R_F(\bT^2)^{\affW_1} \quad , \quad p \mapsto F(t_1) \cdot p \\
	\widetilde{r}_2 &\colon R_F(\bT^2)^{W_2} \to R_F(\bT^2)^{\affW_2} \quad , \quad p \mapsto F(t_3)^{-1} \cdot p \\
	\widetilde{r}_{\{1,2\}} &\colon R_F(\bT^2)^{W_{\{1,2\}}} \to R_F(\bT^2)^{\affW_{\{1,2\}}} \quad , \quad p \mapsto F(t_1) \cdot p
\end{align*}
Define $\widetilde{r}_I \colon R_F(\bT^2)^{W_I} \to R_F(\bT^2)^{\affW_I}$ for all other $I \subset \{0,1,2\}$ to be the identity. Let $\hat{r}_I = \widetilde{r}_I \circ r_I \colon R_F(G_I) \to R_F(\bT^2)^{\affW_I}$.

\begin{lemma} \label{lem:coh_with_coeff}
The isomorphisms $\hat{r}_I$ fit into the following commutative diagram:
\[
	\begin{tikzcd}
		\bigoplus_{\lvert I \rvert = 1} R_F(G_I) \ar[r,"d_0"] \ar[d,"\bigoplus_{\lvert I \rvert = 1} \hat{r}_I" right, "\cong" left] & \bigoplus_{\lvert I \rvert = 2} R_F(G_I) \ar[r,"d_1"] \ar[d,"\bigoplus_{\lvert I \rvert = 2} \hat{r}_I" right, "\cong" left] & R_F(\bT^2) \ar[d,"\hat{r}_{\{0,1,2\}} = \id{}" right, "\cong" left] \\
		C_{\affW}^0(\liet;\cR) \ar[r,"d_0^{\mathrm{cell}}" below] & C_{\affW}^1(\liet;\cR) \ar[r,"d_1^{\mathrm{cell}}" below] & C_{\affW}^2(\liet;\cR)
	\end{tikzcd}	
\]
In particular, the chain complex from Lemma~\ref{lem:diff_d0_d1} computes the $\affW$-equivariant Bredon cohomology $H^*_{\affW}(\liet; \cR)$ of $\liet$ with coefficients in $\cR$.  
\end{lemma}

\begin{proof}
Let $r^{(k)}$ (respectively $\widetilde{r}^{(k)}$) for $k \in \{1,2,3\}$ be the sum over all $r_I$ (respectively $\widetilde{r}_I$) for all $I \subset \{0,1,2\}$ with $\lvert I \rvert = k+1$. Consider
\begin{align*}
	\widehat{d}_0 &\colon \bigoplus_{\lvert I \rvert = 1} R_F(\bT^2)^{W_I} \to \bigoplus_{\lvert I \rvert = 2} R_F(\bT^2)^{W_I} \ ,\\  
	\widehat{d}_1 &\colon \bigoplus_{\lvert I \rvert = 2} R_F(\bT^2)^{W_I} \to R_F(\bT^2)   
\end{align*}
with $\widehat{d}_0(x_0,x_1,x_2) = (-x_0  + F(t_1)x_1, -x_1 + F(t_1 + t_3)^{-1}x_2, -x_0 + F(t_3)^{-1}x_2)$ and $\widehat{d}_1(y_{01}, y_{12}, y_{02}) = y_{01} + F(t_1)y_{12} - y_{02}$. Then we have $r^{(2)} \circ d_0 = \widehat{d}_0 \circ r^{(1)}$ and $\widehat{d}_1 \circ r^{(2)} = d_1$. The statement follows from the following two observations: 
\begin{align*}
	 & (\widetilde{r}^{(2)} \circ \widehat{d}_0)(x_0,x_1,x_2) \\
	=\ & (-x_0  + F(t_1)x_1, -F(t_1)x_1 + F(t_3)^{-1}x_2, -x_0 + F(t_3)^{-1}x_2) \\
	=\ & (-\widetilde{r}_0(x_0)  + \widetilde{r}_1(x_1), -\widetilde{r}_1(x_1) + \widetilde{r}_2(x_2), -\widetilde{r}_0(x_0) + \widetilde{r}_2(x_2)) \\
	=\ & (d_0^{\mathrm{cell}} \circ \widetilde{r}^{(1)})(x_0,x_1,x_2)
\end{align*}
and
\[
 (d_1^{\mathrm{cell}} \circ \widetilde{r}^{(2)})(y_{01},y_{12},y_{02}) = y_{01} + F(t_1)y_{12} - y_{02} = \widehat{d}_1(y_{01},y_{12},y_{02})\ .
\]
\end{proof}
The above lemma reduces the problem of computing the equivariant higher twisted $K$-theory of $SU(3)$ to the computation of Bredon cohomology groups with local coefficients. We will determine these groups after rationalising the coefficients, i.e.\ we compute $H^*_{\affW}(\liet; \cRQ)$. The inclusion $\cR \to \cRQ$ induces a homomorphism
\begin{equation} \label{eqn:rational_coh}
	H^*_{\affW}(\liet; \cR) \to H^*_{\affW}(\liet; \cRQ)\ .
\end{equation}
Even though \cite[Thm.~3.11]{AdemCantareroGomez-TwistedK:2018} is only stated for coefficient systems $\cRQ$ where the module structure is induced by a homomorphism $\pi_1(\bT^2) \to \hom(\bT^2, S^1)$, the proof works verbatim in our situation, where $R_F(\bT^2) \otimes \Q$ carries a $\pi_1(\bT^2)$-action induced by $\pi_1(\bT^2) \to GL_1(R_F(\bT^2) \otimes \Q)$. Hence, we obtain
\begin{equation} \label{eqn:needs_Q}
	H^*_{\affW}(\liet; \cRQ) \cong H^*_{\pi_1(\bT^2)}(\liet; \cRQ)^W\ . 
\end{equation}

\begin{lemma} \label{lem:reg_seq}
	Let $F$ be an exponential functor with $\deg(F(t_1)) > 0$, then $(F(t_2) - F(t_1), F(t_3) - F(t_2))$ is a regular sequence in $R_F(\bT^2) \otimes \Q$. 
\end{lemma}

\begin{proof}
	Let $R_F = R_F(\bT^2) \otimes \Q$, $R = R(\bT^2) \otimes \Q \subset R_F$ and let $q_F = F(\rho) = F(t_1 + t_2 + t_3)$. Let $I_{jk} \subset R_F$ be the ideal generated by $F(t_k) - F(t_j)$.  We have to show that multiplication by $F(t_3) - F(t_2)$ is injective on $R_F/I_{12}$. On this quotient $F(t_3) - F(t_2)$ agrees with $F(t_3) - F(t_1)$. Suppose we have elements $p,q \in R_F$ with the property that $p$ is not divisible by $F(t_2) - F(t_1)$ and  
	\begin{equation} \label{eqn:div}		
		p\cdot (F(t_3) - F(t_1)) = q \cdot (F(t_2) - F(t_1))\ .
	\end{equation}
	Multiplying both sides by an appropriate power of $q_F$ we may assume that $p, q \in R$. Now we can use the relation $t_3 = (t_1t_2)^{-1}$ to express both sides of~\eqref{eqn:div} in terms of $t_1, t_2, t_1^{-1}, t_2^{-1}$. Since we may multiply both sides by $t_1^kt_2^l$ for appropriate $k,l \in \N_0$, we can without loss of generality assume that $p, q \in \Q[t_1,t_2]$. Let
	\[
		F(t_1) = \sum_{k=0}^m a_kt_1^k
	\]
	with $a_m \neq 0$. We have $\deg(F(t_1)) = \deg(F(t_2)) = m$ and by our assumption $m > 0$. However, note that $\deg(F(t_3)) \leq 0$. The highest order term of $F(t_2)$ can be expressed as follows
	\[
		a_mt_2^m = a_mt_1^m - \sum_{k=1}^{m-1} a_k(t_2^k - t_1^k) + F(t_2) - F(t_1)\ .
	\]
	Since we are working over $\Q$, we can therefore assume that $p$ is a linear combination of terms $t_1^kt_2^l$ with $l < m$ by adapting $q$ accordingly. Suppose that $p$ has total degree $r$ and let $p_r$ be the corresponding homogeneous part. Comparing the terms of highest degree in \eqref{eqn:div} we obtain
	\[
		-p_r t_1^m = q_r (t_2^m - t_1^m)\ ,
	\]   
	where $q_r$ is the homogeneous part of $q$ of degree $r$. Since the left hand side contains no summands $t_1^kt_2^l$ with $l \geq m$, this equation implies $q_r = 0$, therefore $p_r = 0$ and hence $p = 0$. This is a contradiction to our initial divisibility assumption. Hence $p$ must be divisible by $F(t_2) - F(t_1)$ proving that multiplication by $F(t_3) - F(t_2)$ is injective on $R_F/I_{12}$.
\end{proof}

\begin{theorem} \label{thm:Koszul}
	Let $F$ be an exponential functor with $\deg(F(t_1)) > 0$. Then $H^k_{\pi_1(\bT^2)}(\liet; \cRQ) = 0$ for $k \neq 2$ and 
	\[
		H^2_{\pi_1(\bT^2)}(\liet; \cRQ) \cong R_F(\bT^2) \otimes \Q/(F(t_2) - F(t_1), F(t_3) - F(t_2)) \ .
	\]
	Moreover, $W$ acts on $H^2_{\pi_1(\bT^2)}(\liet;\cRQ)$ by signed permutations. 
\end{theorem}

\begin{proof}
	Let $\Lambda = \pi_1(\bT^2)$ be as in \eqref{eqn:lattice}. The two vectors $\kappa_1 = (1,-1,0)$ and $\kappa_2 = (0,1,-1)$ form a basis of $\Lambda$. Note that 
	\[
		\Q[\Lambda] \cong \Q[r_1,r_2,r_3]/(r_1r_2r_3 - 1)
	\]
	and $\kappa_1$ corresponds to $r_1r_2^{-1}$ under this isomorphism. Likewise $\kappa_2$ agrees with $r_2r_3^{-1}$. The action of $\Lambda$ on $R_F(\bT^2) \otimes \Q$ extends to a ring homomorphism
	\(
		\varphi \colon \Q[\Lambda] \to R_F(\bT^2) \otimes \Q 
	\)
	given by
	\[
	\varphi\left(\sum_{k,l,m} a_{klm} r_1^k r_2^l r_3^m\right) = \sum_{k,l,m} a_{klm} F(t_1)^kF(t_2)^lF(t_3)^m\ .
	\]
	In particular, $\varphi(r_1r_2^{-1}) = F(t_1)F(t_2)^{-1}$ and similarly for $r_2r_3^{-1}$. The equivariant cohomology groups $H^*_{\Lambda}(\liet; \cRQ)$ are computed by the cochain complex
	\begin{equation} \label{eqn:cochain_complex}
		\hom_{\Q[\Lambda]}(C_*(\liet) \otimes \Q, R_F(\bT^2) \otimes \Q)
	\end{equation}
	(see \cite[I.9, (9.3)]{Bredon-Cohomology:1967}), where $C_*(\liet)$ is the cellular chain complex of $\liet$ viewed as a $\Lambda$-CW-complex. Note that $\liet$ has a $\Lambda$-CW-structure with one $0$-cell given by the orbit of the origin, two $1$-cells corresponding to the orbits of the edge from $(0,0,0)$ to $(1,-1,0)$ and to $(0,1,-1)$, respectively, and one $2$-cell. As pointed out in the proof of \cite[Thm.~4.2]{AdemCantareroGomez-TwistedK:2018} (see also \cite[p.~96]{Charlap-FlatMfds:1986}) the chain complex $C_*(\liet)$ can be identified with the Koszul complex 
	\[
		K_n = \extp^n\Z^2 \otimes \Z[\Lambda]
	\] 
	for the sequence $(r_1r_2^{-1} - 1, r_2r_3^{-1} - 1)$. If we identify $C_*(\liet)$ and $K_*$ in this way, the cochain complex in \eqref{eqn:cochain_complex} turns into 
	\[
		C^n_F = \extp^n R_F^2 \quad \text{with} \quad d^n(y) = x \wedge y 
	\]
	where $R_{F} = R_F(\bT^2) \otimes \Q$ and $x = (F(t_1)F(t_2)^{-1}-1, F(t_2)F(t_3)^{-1}-1)$, which is a regular sequence in $R_F$ by Lemma~\ref{lem:reg_seq}. The first part of the statement now follows from  \cite[Cor.~17.5]{Eisenbud-CommAlg:1995}. We can identify $\Z^2$ with $\Lambda$ using $\kappa_1$ and $\kappa_2$. This induces an action of $W$ on $\Z^2$. The group $W$ acts diagonally on $K_n$ using its natural action on $\Z[\Lambda]$. If we equip the cochain complex
	\[
		\hom_{\Q[\Lambda]}(K_n \otimes \Q, R_F(\bT^2) \otimes \Q)
	\]
	with the $W$-action given by $(g \cdot \varphi)(y) = g\varphi(g^{-1}y)$, then the isomorphism of this cochain complex with \eqref{eqn:cochain_complex} is equivariant. Likewise, $W$ acts diagonally on $C_F^n \cong \extp^n\Q^2 \otimes R_F$ making the last isomorphism equivariant as well. In particular, $W$ acts on $\extp^2\Q^2$ via the sign representation. This proves the second statement.
\end{proof}

To distinguish the signed permutation action of $W$ on $R_F(\bT^2) \otimes \Q$ from the usual one, we denote the two modules by $R_F^{\mathrm{sgn}}$ and $R_F$ respectively as in~\cite{AdemCantareroGomez-TwistedK:2018}. We also define $I_F^{\mathrm{sgn}} = (F(t_2) - F(t_1), F(t_3) - F(t_2))$. Over the rational numbers taking invariants with respect to a finite group action is an exact functor. Hence, Thm.~\ref{thm:Koszul} gives isomorphisms of $R_F^W$-modules
\[
	H^2_{\affW}(\liet; \cRQ) \cong H^2_{\Lambda}(\liet; \cRQ)^W \cong (R_F^{\mathrm{sgn}})^W/(I_F^{\mathrm{sgn}})^W\ .
\]
The ring $R_F^W$ is a localisation of the quotient of the ring of symmetric polynomials in the variables $t_1, t_2, t_3$ by the ideal generated by $(t_1t_2t_3 -1)$. The module $(R_F^{\mathrm{sgn}})^W$ is a similar quotient of the antisymmetric polynomials by the submodule $(t_1t_2t_3 -1) (R_F^{\mathrm{sgn}})^W$. Any antisymmetric polynomial is divisible by the Vandermonde determinant
\[
	\Delta = \Delta(t_1,t_2,t_3) = (t_1 - t_2)(t_2 - t_3)(t_3 - t_1)\ .
\]
This induces an $R_F^W$-module isomorphism 
\[
	\Psi \colon (R_F^{\mathrm{sgn}})^W \to R_F^W  \quad , \quad p \mapsto \frac{p}{\Delta}
\]
(compare with \cite[Sec.~5.1]{AdemCantareroGomez-TwistedK:2018}). Let $\theta_{jk} = F(t_j) - F(t_k)$. 

\begin{lemma} \label{lem:submod}
	The submodule $(I_F^{\mathrm{sgn}})^W$ is generated by the two antisymmetric polynomials $q_+ = \theta_{12}t_3 + \theta_{23}t_1 + \theta_{31}t_2$ and $q_- = \theta_{12}t_3^{-1} + \theta_{23}t_1^{-1} + \theta_{31}t_2^{-1}$. The corresponding submodule $\Psi((I_F^{\mathrm{sgn}})^W)$ is generated by 
	\begin{equation} \label{eqn:determinant}
		\Psi(q_{\pm}) = -\frac{1}{\Delta} \det \begin{pmatrix}
			F(t_1) & F(t_2) & F(t_3) \\[2mm]
			t_1^{\pm 1} & t_2^{\pm 1} & t_3^{\pm 1} \\[2mm]
			1 & 1 & 1
		\end{pmatrix}\ .
	\end{equation}
\end{lemma}

\begin{proof}
	First note that $\theta_{12} + \theta_{23} + \theta_{31} = 0$. Therefore $q_{\pm} \in (I_F^{\mathrm{sgn}})^W$. The module $R(\bT^2)$ is free of rank $6$ over $R(\bT^2)^W$ by \cite[Thm.~2.2]{Steinberg-Pittie:1975}. Thus, the same is true for $R_F$ as a module over $R_F^W$. An explicit basis is given by $\beta = \{e, t_2, t_3, t_2^{-1}, t_1^{-1}, t_1^{-1}t_3 \}$. Consider the averaging map
	\[
		\alpha \colon R_F^{\mathrm{sgn}} \to (R_F^{\mathrm{sgn}})^W \quad, \quad p \mapsto \frac{1}{6} \sum_{g \in W} g \cdot p\ .
	\]
	This is a surjective module homomorphism, which maps the submodule $I_F^{\mathrm{sgn}}$ onto $(I_F^{\mathrm{sgn}})^W$. Let $q \in (I_F^{\mathrm{sgn}})^W$ and choose $p \in I_F^{\mathrm{sgn}}$ such that $\alpha(p) =q$. Then we have 
	\[
		p = \theta_{12}p_{1} + \theta_{23}p_2
	\]
	for suitable $p_i \in R_F$. We have to see that $q = \alpha(p)$ is in the submodule generated by $q_{\pm}$. After decomposing $p_1$ and $p_2$ with respect to $\beta$ we see that it suffices to check that $\alpha(\theta_{12}y)$ and $\alpha(\theta_{23}y)$ lie in this submodule for all $y \in \beta$. Since $\alpha(\theta_{12}) = \alpha(\theta_{23}) = 0$, this is true for $y = e$. We have 
	\begin{align*}
		\alpha(\theta_{12}t_2) &= \frac{1}{6}\left( (\theta_{12} + \theta_{23})t_2 + (\theta_{12} + \theta_{31})t_1 + (\theta_{31} + \theta_{23})t_3 \right) \\
		&= -\frac{1}{6}\left(\theta_{31}t_2 + \theta_{23}t_1 + \theta_{12}t_3\right) = - \frac{1}{6}q_+ \\
		\alpha(\theta_{12}t_3) &= \frac{1}{3}(\theta_{12}t_3 + \theta_{31} t_2 + \theta_{23}t_1) = \frac{1}{3} q_+
	\end{align*}
	The cases $\alpha(\theta_{23}t_2)$ and $\alpha(\theta_{23}t_3)$ work in a similar way. The expressions $\alpha(\theta_{12}t_2^{-1})$, $\alpha(\theta_{23}t_2^{-1})$, $\alpha(\theta_{12}t_1^{-1})$ and $\alpha(\theta_{23}t_1^{-1})$ produce corresponding multiples of $q_-$. In the remaining two cases a short computation shows that
	\begin{align*}
		\alpha(\theta_{12}t_1^{-1}t_3) &= \frac{1}{6} q_+ \cdot (t_1^{-1} + t_2^{-1} + t_3^{-1}) \ ,\\
		\alpha(\theta_{23}t_1^{-1}t_3) &= \frac{1}{6} q_- \cdot (t_1 + t_2 + t_3)\ .
	\end{align*}
	This shows that the submodule $(I_F^{\mathrm{sgn}})^W$ is generated by $q_+$ and $q_-$. The determinant formula follows from a straightforward computation. 
\end{proof}

\begin{example}
In case of the classical twist, i.e.\ for $F = (\extp^{\mathrm{top}})^{\otimes m}$ we have $F(t_i) = t_i^m$. For $\Psi(q_+)$ equation \eqref{eqn:determinant} boils down to the definition of the Schur polynomial for the partition with just one element \cite[I.3, p.~40]{Macdonald-symmetric:2015}. In this case the Schur polynomial agrees with the complete homogeneous symmetric polynomial $h_{m-2}$. Using the properties of the determinant we also have 
\[
	\Psi(q_-) = -\frac{1}{\Delta} \det \begin{pmatrix}
		t_1^{m+1} &  t_3^{m+1} & t_3^{m+1} \\
		1 & 1 & 1 \\
		t_1 & t_2 & t_3 
	\end{pmatrix} = \frac{1}{\Delta} \det \begin{pmatrix}
		t_1^{m+1} &  t_3^{m+1} & t_3^{m+1} \\
		t_1 & t_2 & t_3 \\
		1 & 1 & 1 
	\end{pmatrix} 
\]
which produces $h_{m-1}$. Altogether we obtain
\[
	\Psi(q_+) = -h_{m-2} \quad , \quad \Psi(q_-) = h_{m-1}\ .
\]
For $m = 0$ we have $q_- = q_+ = 0$, for $m=1$ we get $q_+ = 0$ and $q_- = 1$ and in the case $m = 2$ the submodule $(I^{\mathrm{sgn}}_F)^W$ is generated by $q_+ = 1$ and $q_- = h_1$. Thus, for $m \in \{1,2\}$ the quotient $R_F^W/I_F^W$ vanishes.
\end{example}

\begin{example} \label{ex:full_twist}
For the $m$th power of the exterior algebra twist $F = \left( \extp^* \right)^{\otimes m}$ we have $F(t_i) = (1 + t_i)^m$ and since
\[
	F(t_j) = \sum_{l=0}^m \binom{m}{l} t_j^l
\]
the determinant formula for $\Psi(q_{\pm})$ gives
\[
	\Psi(q_+) = -\sum_{l=2}^m \binom{m}{l} h_{l-2} \quad , \quad \Psi(q_-) = \sum_{l=1}^m \binom{m}{l} h_{l-1}\ .
\]
\end{example}

We are now in the position to provide a complete computation of the equivariant higher twisted $K$-theory for $SU(3)$ after rationalisation and summarise our results in the following theorem.

\begin{theorem} \label{thm:main_theorem_SU(3)}
	For the rational equivariant higher twisted $K$-theory of $SU(3)$ with twist given by an exponential functor $F$ with $\deg(F(t)) > 0$ we have the following isomorphism of $R_F(SU(3))$-modules: 
	\[
		K_0^G(C^*(\cE)) \otimes \Q \cong (R_F(SU(3)) \otimes \Q)/J_F \quad, \quad K_1^G(C^*(\cE)) \otimes \Q \cong 0
	\] 
	where $J_F$ is the submodule generated by the two representations $\sigma_1^F$ and $\sigma_2^F$ whose characters $\chi_1, \chi_2$ are the symmetric polynomials  
	\[
		\chi_1 = \frac{1}{\Delta} \det \left(\begin{smallmatrix}
			F(t_1) & F(t_2) & F(t_3) \\[1mm]
			t_1 & t_2 & t_3 \\[1mm]
			1 & 1 & 1
		\end{smallmatrix}\right) \quad , \quad
		\chi_2 = \frac{1}{\Delta} \det \left(\begin{smallmatrix}
			F(t_1)t_1 & F(t_2)t_2 & F(t_3)t_3 \\[1mm]
			t_1 & t_2 & t_3 \\[1mm]
			1 & 1 & 1
		\end{smallmatrix}\right)\ .
	\] 
	In the case $F = \left(\extp^*\right)^{\otimes m}$ the submodule $J_F$ is generated by the representations
	\[
		\sigma_1^F = \sum_{l=2}^m \binom{m}{l} {\mathrm{Sym}}^{l-2}(\rho) \quad , \quad \sigma_2^F = \sum_{l=1}^m \binom{m}{l} {\mathrm{Sym}}^{l-1}(\rho)\ .
	\]
\end{theorem}

\begin{proof}
	Let $\mathcal{Q}$ be the universal UHF-algebra equipped with the trivial action. By continuity of the $K$-functor the $K$-theory of $C^*(\cE) \otimes \mathcal{Q}$ is the rationalisation of the $K$-theory of $C^*(\cE)$. To compute $K_*^G(C^*(\cE) \otimes \mathcal{Q})$ we can use the corresponding spectral sequence from Prop.~\ref{prop:spectral_seq}. The resulting cochain complex will have $R_F(G_I) \otimes \Q$ in place of $R_F(G_I)$ in each degree. Lemma~\ref{lem:coh_with_coeff} identifies it as the complex computing the $\affW$-equivaraint cohomology of $\liet$ with respect to the coefficient system~$\cRQ$. Thus, the homomorphism \eqref{eqn:rational_coh} is now an isomorphism. Combining Thm.~\ref{thm:Koszul} with Lemma~\ref{lem:submod} we obtain the first part of the statement. The second part follows from Example~\ref{ex:full_twist} by identifying the characters given by the homogeneous symmetric polynomials with symmetric powers of the standard representation.
\end{proof}

\begin{remark}
	The choice of the orientation of $\bT$ that went into the construction of $\cE$ through the choice of order on $\bT \setminus \{1\}$ features in the computations of the twisted $K$-groups as follows: Changing the orientation to its opposite corresponds to replacing all factors of the form $F(t)$ by $F(t)^{-1}$. Since the ideal $I_F^{\mathrm{sgn}}$ is invariant under this transformation, we obtain isomorphic higher twisted $K$-groups in both cases.  
\end{remark}

We expect Thm.~\ref{thm:main_theorem_SU(3)} also to be true without the rationalisation. Since the modules $R_F(G_I)$ are free over $R_F(SU(3))$ the differentials from Lemma~\ref{lem:diff_d0_d1} in the cochain complex can be expressed in terms of matrices, which allowed us to perform an extensive computer analysis in the case of the full twist $F = \left( \extp^* \right)^{\otimes m}$ for $m \in \{1,\dots, 8\}$. For these levels we thereby confirmed the above conjecture. 

The approach presented above should also be seen as a blueprint for the computation of the rationalised equivariant higher twisted $K$-theory of $SU(n)$. We will return to this discussion in future work.


\appendix
\section{Proof of Thm.~\ref{thm:extension_of_Fell_bdls}}
Let $\text{inv} \colon \cG \to \cG$ be given by $\text{inv}(\g) = \g^{-1}$. We first extend $\cE_{0,+}$ over $\cG_-$ by defining $\pi \colon \cE_- \to \cG_-$ as $\cE_- = \left(\text{inv}^*\cE_+\right)^\op$. This means that we have $\left(\cE_-\right)_\g = \left(\cE_+\right)_{\g^{-1}}^\op$ fibrewise for all $\g \in \cG_-$. Let $\cE = \cE_- \cup \cE_{0,+}$. Together with the canonical quotient map to $\cG$ this is a Banach bundle. The conjugate linear bijection $e \mapsto e^*$ from \eqref{eqn:star_on_mod} yields a well-defined $*$-operation on $\cE$. It fits into the commutative diagram
\[
\begin{tikzcd}
	\cE \ar[r,"*"] \ar[d,"\pi" left] & \cE \ar[d,"\pi"] \\
	\cG \ar[r,"\text{inv}" below] & \cG
\end{tikzcd}
\]

Next we have to extend the multiplication map $\mu$ to all of $\cE$, i.e.\ we have to construct a bimodule isomorphism $\mu \colon \cE_{\g_1} \otimes_A \cE_{\g_2} \to \cE_{\g_1\g_2}$ for $(\g_1,\g_2) \in \cG^{(2)}$. Depending on which subset $\g_1, \g_2$ and $\g_1\g_2$ are contained in, there are six cases to consider:
\begin{center}
\begin{tabular}{|c|c|c|c|}
\hline
& $\g_1$ & $\g_2$ & $\g_1\cdot \g_2$\\
\hline
1 & $+$ & $+$ & $+$ \\
2 &$+$ & $-$ & $+$ \\
3 &$+$ & $-$ & $-$ \\
4 &$-$ & $+$ & $+$ \\
5 &$-$ & $+$ & $-$ \\
6 &$-$ & $-$ & $-$ \\
\hline
\end{tabular}
\end{center}
A $+$, respectively $-$, refers to the case that the groupoid element is in $\cG_{0,+}$, respectively $\cG_-$. We need the relation 
\[
	(e_1 \cdot e_2)^* = e_2^* \cdot e_1^*
\] 
to hold in a Fell bundle. This implies that if we have defined the multiplication in case $k$, then we have fixed it in case $(7-k)$ as well. This reduces the number of cases to consider to the first three. 

We will use the following graphical representation of groupoid elements: A morphism in $\cG_{0,+}$ will be drawn as an arrow pointing right, a morphism in $\cG_-$ corresponds to an arrow pointing left.\footnote{Contrary to the usual notation of morphisms we will draw the arrows from the codomain to the domain. This way the order of composition agrees with the composition of the arrows.} Let $\g_1,\g_2 \in \cG^{(2)}$. If $\g_1\cdot \g_2$ ends up in $\cG_{0,+}$, then the concatenation of the two corresponding arrows will end in a point to the right of the base of $\g_1$. Similarly, a composition ending up in $\cG_-$ will end up in a point left of the base of $\g_1$. Cases $1,2$ and $3$ are drawn as follows:
\begin{center}
\begin{tikzpicture}[->,>=stealth',node distance=1.4cm,scale=1]
	\node (11) {$\bullet$};
	\node (12) [right of=11] {$\bullet$};
	\node (13) [right of=12] {$\bullet$};
	
	\path (11) edge node [above] {$\g_{1}$} (12);
	\path (12) edge node [above] {$\g_{2}$} (13);

	\node at (4.5,0) (21) {$\bullet$};
	\node (22) [right of=21] {$\bullet$};
	\node (23) [right of=22] {$\bullet$};
	
	\path (21) edge[bend left] node [above] {$\g_{1}$} (23);
	\path (23) edge node [below] {$\g_{2}$} (22);

	\node at (9,0) (31) {$\bullet$};
	\node (32) [right of=31] {$\bullet$};
	\node (33) [right of=32] {$\bullet$};
	
	\path (32) edge node [above] {$\g_{1}$} (33);
	\path (33) edge[bend left] node [below] {$\g_{2}$} (31);
\end{tikzpicture}
\end{center}
The multiplication is already defined in case $1$. For case $2$, let $(\g_1, \g_2) \in \cG^{(2)}$ with $\g_1 \in \cG_{0,+}$, $\g_2 \in \cG_-$, $\g_1 \cdot \g_2 \in \cG_{0,+}$. Observe that $\g_1 = \g_1\g_2 \cdot \g_2^{-1}$ is a decomposition of $\g_1$ into elements that are contained in $\cG_{0,+}$. This can be easily read off from the above graphical representation. Since
\[
	\mu \colon (\cE_{0,+})_{\g_1\g_2} \otimes_A (\cE_{0,+})_{\g_2^{-1}} \to (\cE_{0,+})_{\g_1}
\]
is an isomorphism, we can extend $\mu$ by defining it to be the upper horizontal arrow in the diagram below, in which the vertical arrow on the right hand side is given by right multiplication:
\[
	\begin{tikzcd}
		(\cE_{0,+})_{\g_1} \otimes_A (\cE_{-})_{\g_2} \ar[rr,"\mu"] & & (\cE_{0,+})_{\g_1\g_2} \\
		(\cE_{0,+})_{\g_1\g_2} \otimes_A (\cE_{0,+})_{\g_2^{-1}} \otimes_A (\cE_{0,+})^\op_{\g_2^{-1}} \ar[u,"\mu \otimes \id{}", "\cong" right] \ar[rr,"\id{} \otimes \lscal{A}{\cdot}{\cdot}" below, "\cong" above ] & & (\cE_{0,+})_{\g_1\g_2} \otimes_A A \ar[u,"\cong"]
	\end{tikzcd}
\]
If we label the arrows by Fell bundle elements instead of groupoid morphisms, this definition will graphically be represented as follows:
\begin{center}
\begin{tikzpicture}[->,>=stealth',node distance=1.4cm,scale=1]
	\node at (4.5,0) (11) {$\bullet$};
	\node (12) [right of=11] {$\bullet$};
	\node (13) [right of=12] {$\bullet$};
	
	\path (11) edge[bend left] node [above] {$e_{1}$} (13);
	\path (13) edge node [below] {$e^*_{2}$} (12);

	\node (eq) [right of=13] {$\leadsto$};

	\node (21) [right of=eq] {$\bullet$};
	\node (22) [right of=21] {$\bullet$};
	\node (23) [right of=22] {$\bullet$};

	\path (21) edge node [above] {$e'_{12}$} (22);
	\path (22) edge[bend left] node [above] {$e'_{2}$} (23);
	\path (23) edge[bend left] node [below] {$e^*_{2}$} (22);

\end{tikzpicture}
\end{center}
where $e_{12}' \cdot e_2' = e_1$ and the inner product is used to replace the loop on the right hand side with an element in $A$. 

Consider case $3$, i.e.\  $(\g_1, \g_2) \in \cG^{(2)}$ with $\g_1 \in \cG_{0,+}$, $\g_2 \in \cG_-$, $\g_1 \cdot \g_2 \in \cG_{-}$. Using \eqref{eqn:op_of_tensor} and the multiplication $\mu$ we obtain an isomorphism 
\[
	\mu^\op \colon (\cE_{0,+})_{\g_1}^\op \otimes_A (\cE_{0,+})^\op_{(\g_1\g_2)^{-1}} \to (\cE_{0,+})^\op_{\g_2^{-1}}
\]
and we can extend the multiplication to this case using the upper horizontal arrow in the diagram below:
\[
	\begin{tikzcd}
		(\cE_{0,+})_{\g_1} \otimes_A (\cE_{-})_{\g_2} \ar[rr,"\mu"] & & (\cE_{-})_{\g_1\g_2} \\
		(\cE_{0,+})_{\g_1} \otimes_A (\cE_{0,+})_{\g_1}^\op \otimes_A (\cE_{0,+})^\op_{(\g_1\g_2)^{-1}} \ar[u,"\id{} \otimes \mu{}^\op", "\cong" right] \ar[rr,"\lscal{A}{\cdot}{\cdot} \otimes \id{}" below, "\cong" above ] & & A \otimes (\cE_{0,+})^\op_{(\g_1\g_2)^{-1}} \ar[u,"\cong"]
	\end{tikzcd}
\]
Graphically this definition is represented as follows:
\begin{center}
\begin{tikzpicture}[->,>=stealth',node distance=1.4cm,scale=1]
	\node at (9,0) (11) {$\bullet$};
	\node (12) [right of=11] {$\bullet$};
	\node (13) [right of=12] {$\bullet$};
	
	\path (12) edge node [above] {$e_{1}$} (13);
	\path (13) edge[bend left] node [below] {$e^*_{2}$} (11);

	\node (eq) [right of=13] {$\leadsto$};

	\node (21) [right of=eq] {$\bullet$};
	\node (22) [right of=21] {$\bullet$};
	\node (23) [right of=22] {$\bullet$};

	\path (22) edge[bend left] node [above] {$e_1$} (23);
	\path (23) edge[bend left] node [below] {$(e'_{1})^*$} (22);
	\path (22) edge node [below] {$(e'_{12})^*$} (21);

\end{tikzpicture}
\end{center}
i.e.\ we decompose $e_2^* = (e_1')^* \cdot (e_{12}')^*$ for some $(e_1')^* \in (\cE_{0,+})_{\g_1}^\op$ and $(e_{12}')^* \in (\cE_-)_{\g_1\g_2}$ and define
\[
	e_1 \cdot e_2^* = \lscal{A}{e_1}{e_1'}\cdot (e_{12}')^*\ .
\]
This finishes the extension of the multiplication map $\mu$.

Next we have to prove that the extended multiplication is still associative. Let $\g_1, \g_2, \g_3 \in \cG$ such that $(\g_1,\g_2) \in \cG^{(2)}$ and $(\g_2,\g_3) \in \cG^{(2)}$. Let $e_i \in \cE_{\g_i}$ for $i \in \{1,2,3\}$. We have to show that 
\[
	(e_1 \cdot e_2) \cdot e_3 = e_1 \cdot (e_2 \cdot e_3)
\]
Each $\g_i$ could be in $\cG_{0,+}$ or in $\cG_-$. Thus, if we neglect the compositions for a moment, this leaves us with six cases to consider. However, the above equality implies 
\[
	e_3^* \cdot (e_2^* \cdot e_1^*) = (e_3^* \cdot e_2^*) \cdot e_1^*\ .
\]
Therefore we can without loss of generality assume that $\g_2 \in \cG_{0,+}$. The diagrams of all remaining cases are shown in Figure~\ref{fig:associativity}. To prove the associativity condition in each case we make the following observations:

The given multiplication $\mu$ on $\cG_{0,+}$ is fibrewise a bimodule isomorphism. Thus, whenever we have to decompose an element of $e \in \cE$, we may without loss of generality assume that it is maximally decomposed with respect to the diagram. This means that in terms of the graphical representation we may make the following replacements:
\begin{center}
\begin{tikzpicture}[->,>=stealth',node distance=1.4cm,scale=1]
	\node (11) {$\bullet$};
	\node (12) [right of=11] {$\bullet$};
	\node (13) [right of=12] {$\bullet$};
	
	\path (11) edge[bend left] node[above] {$e$} (13);
	
	\node (eq) [right of=13] {$\leadsto$};

	\node (21) [right of=eq] {$\bullet$};
	\node (22) [right of=21] {$\bullet$};
	\node (23) [right of=22] {$\bullet$};

	\path (21) edge node[above] {$e_1$} (22);
	\path (22) edge node[above] {$e_2$} (23);

	\node at (-1.4,-1.6) (31) {$\bullet$};
	\node (32) [right of=31] {$\bullet$};
	\node (33) [right of=32] {$\bullet$};
	\node (34) [right of=33] {$\bullet$};

	\path (31) edge[bend left] node[above] {$e'$} (34);

	\node (eq) [right of=34] {$\leadsto$};

	\node (41) [right of=eq] {$\bullet$};
	\node (42) [right of=41] {$\bullet$};
	\node (43) [right of=42] {$\bullet$};
	\node (44) [right of=43] {$\bullet$};

	\path (41) edge node[above] {$e_1$} (42);
	\path (42) edge node[above] {$e_2$} (43);
	\path (43) edge node[above] {$e_3$} (44);

\end{tikzpicture}
\end{center}
where $e = e_1\cdot e_2$ in the first case and $e' = e_1\cdot e_2 \cdot e_3$ in the second. Note that associativity of $\mu$ over $\cG_{0,+}$ ensures that we may drop the brackets in the second case. Likewise, we may assume the analogous decomposition for the mirror images of these diagrams with arrows pointing to the left.

By our definition of the extension of $\mu$ to $\cG_-$ associativity is also satisfied in diagrams that have one of the following forms:
\begin{center}
	
\begin{tikzpicture}[->,>=stealth']
	\node (11)  {$\bullet$};
	\node (12) [right of=11] {$\bullet$};
	\node (13) [right of=12] {$\bullet$};

	\path (11) edge node [above] {} (12);
	\path (12) edge[bend left] node [below] {} (13);
	\path (13) edge[bend left] node [below] {} (12);

	\node (21) [right of=13] {$\bullet$};
	\node (22) [right of=21] {$\bullet$};
	\node (23) [right of=22] {$\bullet$};

	\path (22) edge[bend left] node [above] {} (23);
	\path (23) edge[bend left] node [below] {} (22);
	\path (22) edge node [below] {} (21);

	\node (31) [right of=23] {$\bullet$};
	\node (32) [right of=31] {$\bullet$};
	\node (33) [right of=32] {$\bullet$};

	\path (33) edge node [above] {} (32);
	\path (32) edge[bend right] node [below] {} (31);
	\path (31) edge[bend right] node [below] {} (32);

	\node (41) [right of=33] {$\bullet$};
	\node (42) [right of=41] {$\bullet$};
	\node (43) [right of=42] {$\bullet$};

	\path (42) edge[bend right] node [above] {} (41);
	\path (41) edge[bend right] node [below] {} (42);
	\path (42) edge node [below] {} (43);

\end{tikzpicture}

\end{center}

\begin{figure}[ht]
\centering
\begin{tikzpicture}[->,>=stealth',node distance=1.4cm,scale=1]
	\node[shape=circle,draw,inner sep=1pt] at (0,0.8) {1};
	
	\node (11) {$\bullet$};
	\node (12) [right of=11] {$\bullet$};
	\node (13) [right of=12] {$\bullet$};
	\node (14) [right of=13] {$\bullet$};
	
	\path (11) edge node [above] {$\g_{1}$} (12);
	\path (12) edge node [above] {$\g_{2}$} (13);
	\path (13) edge node [above] {$\g_{3}$} (14);

	\node[shape=circle,draw,inner sep=1pt] at (6,0.8) {2};

	\node at (6,0) (21) {$\bullet$};
	\node (22) [right of=21] {$\bullet$};
	\node (23) [right of=22] {$\bullet$};
	\node (24) [right of=23] {$\bullet$};

	\path (21) edge node [above] {$\g_{1}$} (22);
	\path (22) edge[bend left] node [above] {$\g_{2}$} (24);
	\path (24) edge node [below] {$\g_{3}$} (23);

	\node[shape=circle,draw,inner sep=1pt] at (0,-1.2) {3};

	\node at (0,-2) (31) {$\bullet$};
	\node (32) [right of=31] {$\bullet$};
	\node (33) [right of=32] {$\bullet$};
	\node (34) [right of=33] {$\bullet$};

	\path (31) edge[bend left] node [above] {$\g_{1}$} (33);
	\path (33) edge node [above] {$\g_{2}$} (34);
	\path (34) edge[bend left] node [below] {$\g_{3}$} (32);

	\node[shape=circle,draw,inner sep=1pt] at (6,-1.2) {4};

	\node at (6,-2) (41) {$\bullet$};
	\node (42) [right of=41] {$\bullet$};
	\node (43) [right of=42] {$\bullet$};
	\node (44) [right of=43] {$\bullet$};

	\path (42) edge node [above] {$\g_{1}$} (43);
	\path (43) edge node [above] {$\g_{2}$} (44);
	\path (44) edge[bend left] node [above] {$\g_{3}$} (41);

	\node[shape=circle,draw,inner sep=1pt] at (0,-3.2) {5};

	\node at (0,-4) (51) {$\bullet$};
	\node (52) [right of=51] {$\bullet$};
	\node (53) [right of=52] {$\bullet$};
	\node (54) [right of=53] {$\bullet$};

	\path (52) edge node [above] {$\g_{1}$} (51);
	\path (51) edge[bend right] node [below] {$\g_{2}$} (53);
	\path (53) edge node [below] {$\g_{3}$} (54);

	\node[shape=circle,draw,inner sep=1pt] at (6,-3.2) {6};

	\node at (6,-4) (61) {$\bullet$};
	\node (62) [right of=61] {$\bullet$};
	\node (63) [right of=62] {$\bullet$};
	\node (64) [right of=63] {$\bullet$};

	\path (63) edge[bend right] node [above] {$\g_{1}$} (61);
	\path (61) edge node [below] {$\g_{2}$} (62);
	\path (62) edge[bend right] node [below] {$\g_{3}$} (64);

	\node[shape=circle,draw,inner sep=1pt] at (0,-5.2) {7};

	\node at (0,-6) (71) {$\bullet$};
	\node (72) [right of=71] {$\bullet$};
	\node (73) [right of=72] {$\bullet$};
	\node (74) [right of=73] {$\bullet$};

	\path (74) edge[bend right] node [below] {$\g_{1}$} (71);
	\path (71) edge node [below] {$\g_{2}$} (72);
	\path (72) edge node [below] {$\g_{3}$} (73);

	\node[shape=circle,draw,inner sep=1pt] at (6,-5.2) {8};

	\node at (6,-6) (81) {$\bullet$};
	\node (82) [right of=81] {$\bullet$};
	\node (83) [right of=82] {$\bullet$};
	\node (84) [right of=83] {$\bullet$};

	\path (82) edge node [above] {$\g_{1}$} (81);
	\path (81) edge[bend right] node [above] {$\g_{2}$} (84);
	\path (84) edge node [above] {$\g_{3}$} (83);

	\node[shape=circle,draw,inner sep=1pt] at (0,-7.7) {9};

	\node at (0,-8.5) (91) {$\bullet$};
	\node (92) [right of=91] {$\bullet$};
	\node (93) [right of=92] {$\bullet$};
	\node (94) [right of=93] {$\bullet$};

	\path (93) edge[bend right] node [above] {$\g_{1}$} (91);
	\path (91) edge[bend right=50] node [below] {$\g_{2}$} (94);
	\path (94) edge[bend left] node [below] {$\g_{3}$} (92);

	\node[shape=circle,draw,inner sep=1pt] at (6,-7.7) {10};

	\node at (6,-8.5) (101) {$\bullet$};
	\node (102) [right of=101] {$\bullet$};
	\node (103) [right of=102] {$\bullet$};
	\node (104) [right of=103] {$\bullet$};

	\path (104) edge[bend right] node [above] {$\g_{1}$} (101);
	\path (101) edge[bend right] node [below] {$\g_{2}$} (103);
	\path (103) edge node [above] {$\g_{3}$} (102);

	\node[shape=circle,draw,inner sep=1pt] at (0,-10.7) {11};

	\node at (0,-11.5) (111) {$\bullet$};
	\node (112) [right of=111] {$\bullet$};
	\node (113) [right of=112] {$\bullet$};
	\node (114) [right of=113] {$\bullet$};

	\path (113) edge node [above] {$\g_{1}$} (112);
	\path (112) edge[bend right] node [below] {$\g_{2}$} (114);
	\path (114) edge[bend right] node [above] {$\g_{3}$} (111);

	\node[shape=circle,draw,inner sep=1pt] at (6,-10.7) {12};

	\node at (6,-11.5) (121) {$\bullet$};
	\node (122) [right of=121] {$\bullet$};
	\node (123) [right of=122] {$\bullet$};
	\node (124) [right of=123] {$\bullet$};

	\path (124) edge[bend right] node [above] {$\g_{1}$} (122);
	\path (122) edge node [below] {$\g_{2}$} (123);
	\path (123) edge[bend left=50] node [below] {$\g_{3}$} (121);
\end{tikzpicture}
\caption{\label{fig:associativity}The 12 compositions to consider in the proof of associativity.}
\end{figure}

Using this it follows that associativity holds in the cases $1$, $2$, $5$ and $8$ in Fig.~\ref{fig:associativity}. The multiplication isomorphism $\mu \colon \left(\cE_{0,+}\right)_{\g_1} \otimes_A \left(\cE_{0,+}\right)_{\g_2} \to \left(\cE_{0,+}\right)_{\g_1\g_2}$ preserves inner products, hence we obtain
\[
	\lscal{A}{e_1 \cdot e_2}{e_1'\cdot e_2'} = \lscal{A}{e_1\cdot {\lscal{A}{e_2}{e_2'}} }{e_1'}\ ,
\]
which is diagrammatically represented by the following relation:
\begin{center}
\begin{tikzpicture}[->,>=stealth',node distance=1.3cm,scale=1]
	\node (11) {$\bullet$};
	\node (12) [right of=11] {$\bullet$};
	\node (13) [right of=12] {$\bullet$};

	\path (11) edge[bend left] node [above] {$e_1 \cdot e_2$} (13);
	\path (13) edge[bend left] node [below] {$(e_1' \cdot e_2')^*$} (11);
	
	\node (eq) [right of=13] {$\leadsto$};
	
	\node (21) [right of=eq] {$\bullet$};
	\node (22) [right of=21] {$\bullet$};
	\node (23) [right of=22] {$\bullet$};

	\path (21) edge[bend left] node [above] {$e_1$} (22);
	\path (22) edge[bend left] node [above] {$e_2$} (23);
	\path (23) edge[bend left] node [below] {$(e_2')^*$} (22);
	\path (22) edge[bend left] node [below] {$(e_1')^*$} (21);

\end{tikzpicture}
\end{center}
Again an analogous relation is true for the mirror images of the above diagrams. Our observation shows we can drop the brackets in expressions represented by diagrams of the form depicted on the right hand side. This implies that associativity holds in the cases $3$,~$4$,~$6$,~$7$ in Fig.~\ref{fig:associativity}. 

The last relation needed is the compatibility of the two inner products in each fibre. Let $e_1, e_2, e_3 \in (\cE_{0,+})_{\g}$ for some $\g \in \cG_{0,+}$. Then \eqref{eqn:comp_inner_prod} implies in our context that:
\begin{align*}
	e_1\,\rscal{e_2}{e_3}{A} &= \lscal{A}{e_1}{e_2}\,e_3\ ,\\
	e_1^*\,\lscal{A}{e_2}{e_3} &= \rscal{e_1}{e_2}{A}\,e_3^*\ .
\end{align*}
Expressed graphically this means that associativity holds for the diagrams:
\begin{center}
\begin{tikzpicture}[->,>=stealth',node distance=1.3cm,scale=1]
	\node (11) {$\bullet$};
	\node (12) [right of=11] {$\bullet$};

	\path (11) edge[bend left=90] node [above] {$e_1$} (12);
	\path (12) edge node [above] {$e_2^*$} (11);
	\path (11) edge[bend right=80] node [below] {$e_3$} (12);

	\node (21) [right of=12] {$\bullet$};
	\node (22) [right of=21] {$\bullet$};

	\path (22) edge[bend right=90] node [above] {$e_1^*$} (21);
	\path (21) edge node [above] {$e_2$} (22);
	\path (22) edge[bend left=80] node [below] {$e_3^*$} (21);
	
\end{tikzpicture}
\end{center}
Using this relation it follows that associativity holds in the cases $9$, $10$, $11$ and $12$ as well. Since the computations are slightly more involved than in the previous cases, we give the details for diagram $10$. Let $e_1^* \in \left(\cE_-\right)_{\g_1}$, $e_2 \in \left(\cE_{0,+}\right)_{\g_2}$ and $e_3^* \in \left(\cE_-\right)_{\g_3}$. Let $e_2 = e_{21} \cdot e_{22}$ and $e_3^*= e_{31}^* \cdot e_{32}^* \cdot e_{33}^*$ be the maximal decompositions of $e_2$ and $e_3$. We have
\begin{align*}
	(e_1^* \cdot e_2) \cdot e_3^* &= \lscal{A}{{\rscal{e_1}{e_{21}}{A}}\,e_{22}}{e_{31}}\,e_{32}^* \cdot e_{33}^* \\
	e_1^* \cdot (e_2 \cdot e_3^*) &= e_1^*\,\lscal{A}{e_{21} \cdot e_{22}}{e_{32} \cdot e_{31}}\,e_{33}^*
\end{align*}
and by our considerations above we obtain: 
\begin{align*}
	e_1^*\,\lscal{A}{e_{21} \cdot e_{22}}{e_{32} \cdot e_{31}}\,e_{33}^* &= e_1^*\,\lscal{A}{e_{21}\,{\lscal{A}{e_{22}}{e_{31}}}}{e_{32}}e_{33}^* \\
	&= \rscal{e_1}{e_{21}\,{\lscal{A}{e_{22}}{e_{31}}}}{A}\,e_{32}^*\cdot e_{33}^* \\
	&= \rscal{e_1}{e_{21}}{A}\,\lscal{A}{e_{22}}{e_{31}}\,e_{32}^*\cdot e_{33}^* \\
	&= \lscal{A}{{\rscal{e_1}{e_{21}}{A}}\,e_{22}}{e_{31}}\,e_{32}^* \cdot e_{33}^*
\end{align*}
The other cases are similar. This finishes the proof of associativity.

Thus, we obtain a Banach bundle $\cE \to \cG$ with a continuous, bilinear, associative multiplication $\mu$ and a compatible continuous conjugate linear involution $\ast \colon \cE \to \cE$. Our definition implies that 
\[
	e^* \cdot e = \rscal{e}{e}{A}
\]
for all $e \in \cE$. This implies the $C^*$-norm condition $\lVert e^* e \rVert = \lVert \rscal{e}{e}{A}\rVert = \lVert e \rVert^2$, which also ensures that the norm is submultiplicative for all $e \in \cE$. Therefore $\cE \to \cG$ defines a saturated Fell bundle.  

To address the question about uniqueness let $\mathcal{F} \to \cG$ be another saturated Fell bundle that satisfies the conditions in the theorem. In particular, we have $\mathcal{F}_{0,+} = \cE_{0,+}$. Therefore the involution yields a (linear!) isomorphism of Banach bundles 
\[
	\Theta \colon \mathcal{F}_{-} \to \left(\cE_+\right)^\op = \cE_-
\] 
The relation $(e_1 \cdot e_2)^* = e_2^* \cdot e_1^*$ shows that $\Theta$ has to intertwine the restrictions $\mu_{\mathcal{F}_-} \colon \mathcal{F}_- \otimes_A \mathcal{F}_- \to \mathcal{F}_-$ and $\mu_{\cE_-}$. The relation $\rscal{f_1^*}{f_2^*}{A} = \lscal{A}{f_1}{f_2}$ for all $f_1^*,f_2^* \in (\mathcal{F}_-)_g$ implies that $\Theta$ preserves the inner product as well. We have already seen above that our extension of the multiplication map was forced upon us by associativity considerations, the fact that $\mu$ induces fibrewise bimodule isomorphisms and the relation $f_1^* \cdot f_2 = \rscal{f_1}{f_2}{A}$ for $f_1, f_2 \in \mathcal{F}_g$. Consequently, if we extend $\Theta$ by the identity on $\mathcal{F}_{0,+} = \cE_{0,+}$ it yields an isomorphism of Fell bundles $\mathcal{F} \to \cE$ over $\cG$.

\begin{remark}
	Let $A$ be a unital separable $C^*$-algebra and let $\mathcal{MR}(A)$ be the $2$-groupoid that has $A$ as its objects, the $A$-$A$ equivalence bimodules as $1$-morphisms and bimodule isomorphisms as $2$-morphisms\footnote{The notation $\mathcal{MR}$ is for Morita-Rieffel.}. The groupoid $\cG$ is a $2$-groupoid with just identity $2$-morphisms. If we forget about the topology, then a saturated Fell bundle is a functor
	\[
		\cE \colon \cG \to \mathcal{MR}(A)\ ,
	\]
	 whereas saturated half-bundles correspond to functors 
	\[
		\cE \colon \cG_{0,+} \to \mathcal{MR}(A)\ .
	\]
	Theorem~\ref{thm:extension_of_Fell_bdls} can be rephrased by saying that the restriction functor 
	\[
		\text{res} \colon \mathcal{F}un(\cG,\mathcal{MR}(A)) \to \mathcal{F}un(\cG_{0,+},\mathcal{MR}(A))
	\]
	induced by the inclusion $\cG_{0,+} \to \cG$ is an isomorphism of functor categories. This seems to suggest that there should be a proof of Theorem~\ref{thm:extension_of_Fell_bdls} based on category theory. However, a first step would require identifying the right topology on $\mathcal{MR}(A)$ to obtain a bijection between saturated Fell bundles and continuous functors. 
\end{remark}

\end{document}